\documentclass[11pt,letterpaper]{amsart}
\usepackage[english]{babel}

\usepackage[letterpaper,top=2cm,bottom=2cm,left=3cm,right=3cm,marginparwidth=1.75cm]{geometry}

\usepackage{amsmath, amsfonts, amssymb, amsthm}
\usepackage{mathtools}
\usepackage{graphicx}
\usepackage{tikz}
\usetikzlibrary{arrows.meta,calc}
\usepackage[colorlinks=true, allcolors=blue]{hyperref}

\renewcommand{\ge}{\geqslant}
\renewcommand{\le}{\leqslant}

\newcommand{\R}{\mathbb{R}}

\newcommand{\Id}{\mathrm{Id}}

\newtheorem{proposition}{Proposition}
\newtheorem*{proposition*}{Proposition}
\newtheorem{lemma}{Lemma}
\newtheorem*{lemma*}{Lemma}
\newtheorem{theorem}{Theorem}
\newtheorem*{theorem*}{Theorem}
\newtheorem{corollary}{Corollary}
\newtheorem*{corollary*}{Corollary}

\newtheorem*{conjecture*}{Conjecture}

\theoremstyle{definition}
\newtheorem{definition}{Definition}
\newtheorem*{definition*}{Definition}
\newtheorem{remark}{Remark}

\long\def\forget#1\forgotten{} 

\everymath{\displaystyle}

\title{Geodesic nets on the Euclidean plane and closed geodesic nets on Riemannian surfaces}
\author{Ivan Frolov, Denis Gorodkov, Dmitrii Korshunov, Alexander Nabutovsky}

\begin{document}

\maketitle

\begin{abstract} We prove that if $M$ is a closed Riemannian surface of diameter $d$ and area $v$ with sectional curvature in the $[-1,1]$ interval,
then a closed geodesic net of length $l$ has at most $f(l,d,v)$ branch points, where 
$f(l,d,v)=(400\bar{l})^{(180\bar{l})^4}$ for $\bar{l}=\max\{l, {\exp(d)\over \min\{1, {v\over 4}\}}\}$.


This answers a question posed by S. Becker-Kahn. We also prove that for each 
geodesic net in the Euclidean plane with
at most $n$ unbalanced (boundary) vertices such that all its unbalanced vertices have degree $1$, the number
of balanced vertices of degree $\geq 3$ (=branch points) does not exceed
$(25n)^{2n^2}$. This answers a question posed in [GM] and [NP].
\end{abstract}

\section{Introduction}
Let $M$ be a Riemannian manifold. The closed geodesic nets on $M$ are critical points of the length functional on the space of the embedded multigraphs in $M$. They can be regarded as 
a homological generalization of periodic geodesics
or as a $1$-dimensional version of stationary minimal surfaces. From the perspective of geometric measure theory, they can be regarded as especially nice stationary $1$-dimensional varifolds in $M$ (cf. a classical reference paper by
W. Allard and F. Almgren [AA]). One of many reasons to study stationary geodesic nets is that their properties are important for understanding singularities of area minimizing hypersurfaces mod $p$ (cf. [DMS] and [MS]. For example, P. Minter and S. Stanbury observe that Question $5$
in their paper [MS] is closely related to Question $5'$ in [MS] which is a problem about closed geodesic nets on the round $2$-sphere of radius $1$.)

More generally, given
a finite set (possibly empty) $U$ of points of $M$, geodesic nets with the set $U$ of unbalanced vertices are critical points
of the length functional on the space of embedded multigraphs in $M$, where the vertices
of $U$ are fixed during deformation.

There are two equivalent definitions of geodesic nets: 

\begin{definition}

A {\it (stationary) geodesic net} on $M$ is an embedded
multigraph on $M$ (that is, loops and multiple edges are allowed), where edges are allowed to have positive integer multiplicities such that
(1) All edges are nonself-intersecting geodesics on $M$; (2) Let $v$ be an arbitrary vertex of the graph which is
{\it not} in $U$. Consider the set $E_v$ of all the edges incident to $v$. For each edge, consider its unit tangent vector $t_{ev}\in T_vM$ at $v$ directed from the origin. Then the sum $I(v)=\Sigma_{e\in E_v} w(e) t_{ev}=0\in T_vM$, where
$w(e)$ denotes the (positive integer) weight of $e$. 
\end{definition}

If $v$ is an unbalanced vertex, then the vector $I(v)$ will be non-zero, and is called {\it the imbalance vector} at $v$. When $M=\mathbb{R}^n$,
and we can translate and add tangent vectors at different points of $M$, the sum of the imbalance vectors  $I(v)$ over all vertices $v$ of $G$ (including the vertices in $U$) is equal to the zero vector. Indeed,
in this sum for each edge $e=[v_iv_j]$ we are going to encounter terms $w(e)t_{ev_i}$ and $w(e)t_{ev_j}$ that will cancel each other. Denote the norm of $I(v)$ by $i(v)$.
We are going to call $i(v)$ the {\it imbalance} of the geodesic net at $v$.
By definition, $I(v)\not= 0$ if and only if $v\in U$. The sum $I$ of the imbalances $i(v)$ at all the unbalanced vertices $v$ of a geodesic net $N$ is called {\it the imbalance of the geodesic net}.

Equivalently, geodesic nets can be defined
as follows:

\begin{definition}
    A \textit{(stationary) geodesic net} on a Riemannian manifold $M$ is a map $g\colon G \hookrightarrow M$ of a finite $1$-dimensional polyhedron $G$ into $M$ where:
    \begin{itemize}
        \item The restriction of $g$ to each
        $1$-cell is an embedding; 
        \item Consider $G$ as a multigraph. The
        images of interiors of two edges 
        either do not intersect, or coincide.
        This implies that each edge in the embedded multigraph $g(G)\subset M$ has a positive integer multiplicity;
        \item There exists a (possibly) empty set $S$ of vertices of $G$ such that $M$, $g$, and $U=g(S)$ have the following property:
        For each 1-parametric flow $\Phi_t$ of diffeomorphisms of $M$ with $t\in (-\varepsilon, \varepsilon)$, $\Phi_0 = \Id$, fixing all points of $U$, $t = 0$ is the critical point of the function $l(t)$ defined as the length of $\Phi_t(g(G))$.
    \end{itemize}


\end{definition}

\begin{definition}
    Vertices in $U$ are called \textit{unbalanced} vertices, and all other vertices of the geodesic net are called \textit{balanced}. A (stationary) geodesic net without
    unbalanced vertices, ($S=\emptyset$), is called {\it closed} .
\end{definition}



We refer the reader to [NP] for an introduction 
to geodesic nets with an emphasis on geodesic nets
in the Euclidean plane. Observe that edges of a geodesic net in Euclidean space are straight line segments. There are no closed geodesic nets
in the plane or in a higher-dimensional Euclidean space. (Indeed, the furthest from the origin vertex cannot be balanced, as all tangent
vectors to incident edges are in the same half-space.)
Also, note that we can place any number of balanced vertices of degree $2$ in the interior of each edge of a geodesic net. We can also remove
each balanced vertex of degree $2$ of a geodesic net, so that the two edges incident to the removed vertex
will merge into one edge. (The only exception, when we cannot perform this operation, is the case when one of the connected components of a geodesic net is a periodic geodesic regarded as a vertex and a loop incident to this vertex.) Therefore, below we can consider only connected geodesic nets
without balanced vertices of degree $2$.

Here are the main results of the present paper.

\begin{theorem} Assume that a geodesic net in
the Euclidean plane has $N$ unbalanced vertices, and all these vertices have degree $1$. Then it has at most $(25N)^{2N^2}$ vertices of degree $>2$.
\end{theorem}

This theorem answers a question that was independently posed in [GM] and in a more general form in [NP]: Assume that unbalanced vertices in a geodesic net in the Euclidean plane can have
arbitrary degrees. Can one estimate the number
$B$ of balanced vertices in terms of the number $U$ of unbalanced vertices, and the total imbalance $I$ of the net? Recall that the imbalance of each vertex $v$ is the norm
of the sum of the unit vectors directed from $v$ along the incident edges and taken with positive integer multiplicities equal to the weights of the edges. The imbalance of the whole
net is the sum of imbalances of all unbalanced vertices.
It was shown in [NP] that this problem easily reduces
to the particular case, when all unbalanced
vertices have degree $1$ as in the previous theorem. Moreover, one can reduce it to the case when all unbalanced vertices are vertices of a convex polygon in the plane, and the geodesic net is contained inside this polygon. The idea is very simple: One can balance each unbalanced vertex $v$ by adding two rays, if $i(v)<1$, or $\lceil i(v)\rceil$
rays for $i(v)\geq 1$. Then we extend these
rays towards infinity, and cut them at the intersection with a circle of a very large radius. It is clear that now all unbalanced vertices have degree one, and the number $p$
of these vertices will not exceed $I+2U\leq 2(I+U)$. This (integer) number will be called
the {\it extended imbalance} of the geodesic
net in the Euclidean plane.

As a corollary, we obtain the answer for the question posed at [NP]:

\begin{corollary} Consider a geodesic net in the Euclidean plane with $U$ unbalanced vertices, $B$ balanced vertices , none of which has degree $2$, and imbalance $I$.
Then $B\leq (50(I+U))^{8(I+U)^2}$.
\end{corollary}

Another question that we are going to address
was posed by Spencer Becker-Kahn ([BK]) ten years ago, and became well-known among specialists in min-max theory: Let $M$ be a closed Riemannian manifold. Is there a function $f_M(x)$ such that for each $l$ given a closed (stationary) geodesic net of length $\leq l$, the number of its branch points (=vertices of degree $>2$) does not exceed $f_M(l)$?

Until now the answer was not known even in the case when $M$ is a round $2$-sphere. (This question was also posed by Spencer Becker-Kahn.) In this paper we provide an answer
valid for all closed surfaces.

\begin{theorem}
Let $M$ be a closed Riemannian surface.
Assume that it is
scaled so that its Gaussian curvature is between $-1$ and $1$. Denote the diameter of $M$ by $d$ and its area by $v$.
Then each closed stationary geodesic net on $M$ of length $l\geq \tilde c(M)={\exp(d)\over\min\{1, {v\over 4}\}}$ has 
at most $(400l)^{(180l)^4}$ vertices of degree $>2$ (that is, branch points). If $l\leq \tilde c(M)$,
then the number of vertices of degree $>2$ does not exceed $(400\tilde c(M))^{(180\tilde c(M))^4}$.

\end{theorem}
\par\noindent
{\bf Remarks.} 1. If the sectional curvature of $M^2$ is between $-x^2$ and $x^2$ for a positive $x$, we can 
rescale $M^2$ by a factor of $x$ without changing the number of branch points
of geodesic nets. 
If $l\geq {\exp(xd)\over\min\{1, {vx^2\over 4}\}}$, then
the number of branch points of a closed geodesic net of length $\leq l$ in $M$ does not exceed
$(400lx)^{(180lx)^4}$.
\par\noindent
2. If $M$ is {\it not} diffeomorphic to the torus or the Klein bottle, the Gauss-Bonnet theorem implies
that the area of $M$ is greater than $4$, and $\tilde c(M)=\exp(d)$.


\section {Main ideas and plan of the paper.}


\subsection{Geodesic nets in the plane}
In the next section, we reduce Theorem 1 and Corollary 1 to the following partition problem:

\par\noindent
{\bf Problem:}
Let $G$ be a plane polygon of perimeter $p$ such that all its side lengths are integer. Assume that it is partitioned into convex polygons with integer side lengths. Is it true that the number of polygons in the partition does not
exceed $f(p)$ for some function $f$?

It turns out that $f(N)$ is an upper bound for the number of balanced vertices of degree $>2$ in Theorem 1. Suppose that $p$ is the extended imbalance of the extended geodesic nets. Recall that $p\leq I+2U\leq 2(I+U)$. Then $f(p)\leq f(2(I+U))$ is also an upper bound for the number of balanced vertices of degree $>2$ in Corollary 1. Therefore, Theorem 1 and Corollary 1 follow from the following theorem:

\begin{theorem} A plane polygon $P$ with perimeter $p$ and integer sides cannot be partitioned into more than $(25p)^{2p^2}$ convex polygons with integer sides.

\end{theorem}


The ideas behind the reduction of Theorem 1 to Theorem 3
will be explained in sections 2.1 and 2.2. Here is a brief description: balance each unbalanced vertex $v$ by adding $\lceil i(v)\rceil$ or $\lceil i(v)\rceil+1$ rays of degree $1$. Extend these rays to infinity. Cut
the resulting geodesic net at a circle $C$ of a very big
radius. We will obtain a geodesic net where all unbalanced vertices $v$ have imbalances equal to one,
the total imbalance does not exceed $p=I+2U$, all unbalanced vertices are on $C$, and the net is in the disk bounded by $C$.

Now for each vertex consider vectors perpendicular to
all incident edges of length equal to the weight of the considered edge. As the sum of these vectors is equal to $0$, we can build a convex polygon formed by these new vectors (in the same cyclic order as the original vectors). As each edge $e$ of the net is incident to two vertices, we
have two polygons were one of the edges corresponds to $e$. We can glue all convex polygons by identifying
the corresponding edges. If done properly, the vertices
of the new complex correspond to the (convex) faces of the original net. The fact that the sum of outer angles of each face is equal to $2\pi$, implies that the sum of the angles incident to each vertex of the new complex
is equal to $2\pi$. This implies that the resulting (metric) dual complex is flat and embeds in the Euclidean plane. Its boundary will be a convex polygon
with at most $p$ sides each of which is perpendicular to one of the infinite rays of original geodesic net.

To prove Theorem 3, we observe that the isoperimetric inequality implies an upper bound 
${p^2\over 4\pi}$ for the area of the ambient polygon. If it is partitioned into a large number of convex polygons with integer sides, then almost all of
these polygons have a very small area. This is possible only if these small area polygons
have
an even number of sides $2k$, two
very small angles at two opposite vertices, and
the rest of the angles that are almost equal to $\pi$. Looking from a different perspective,
for almost all vertices $v$ inside the ambient polygon, the angle $2\pi$ at $v$ is
divided by the incident edges into two angles almost equal to $\pi$ and one or several
very small angles. 

We can use these observations to demonstrate
that if the number of polygons in the partition is large, then there exist two vertices $A$ and $C$ in the plane graph $G$ formed by the polygons in the partition that are connected by a large number $N$ of vertex-disjoint paths, so that all edges of these paths are almost parallel to each other and to the straight line $(AC)$. (All angles between incident edges on each of these paths are almost equal to $\pi$.) The number $N$ would behave as $\vert E\vert^{1\over p^2}$, where $E$ denotes the number of all edges of $G$. All angles
at all vertices along these paths are close to $\pi$. If we knew that at least one angle along
each of these paths is less than $\pi-\alpha$ or greater than $\pi+\alpha$ for some $\alpha>0$,
then we would know that the area of all polygons of the subdivision adjacent to each of at least $N$ vertices with the outer angle with the absolute value greater than $\alpha$
is at least $\sim \alpha$, and the total area of the polygon formed by the two outer
paths from $A$ to $C$ is at least $\sim \alpha N$. 

All paths connecting $A$ and $C$ have the same
integer length $l$ slightly larger than $|AC|$.
The key idea is to use $c=l-|AC|$ to derive a lower bound of the form $c_1(p)\sqrt{c}$ for $\alpha$ in the previous paragraph that leads to the lower bound $c_2(p)\sqrt{c}N$ for the area of the polygon between the outer paths (see Lemma 6 and Corollary 6). Indeed, if all angles along the path were almost equal to $\pi$, then 
$l$ would be almost equal to $|AC|$, so it is not surprising that such an estimate must exist.

At first glance, this estimate is useless, as we do not have any control over $c$ which potentially can be arbitrarily small. But we combine this estimate with the observation
that all $N$ paths from $A$ to $C$ are $const(p)\sqrt{c}$-close to the segment $[AC]$ (Lemma 7).
Indeed, otherwise one of the vertices, $D$, on one of the paths was at a distance $>const(p) \sqrt{c}$ for an appropriate $const(p)$
from $[AC]$. If so, the sum of the distances $|AD|+|DC|$ that cannot exceed $l=|AC|+c$ would 
be greater than $|AC|+c$. Thus, the polygon between the outer paths among $N$ paths connecting $A$ and $C$ is contained in the
rectangle with side lengths $|AC|<l<p$ and $2const(p)\sqrt{c}$, and, therefore, its area is at most $c_3(p)\sqrt{c}$. Comparing this upper bound for the area that does not involve $N$ and the lower bound $c_2(p)\sqrt{c}N$ we obtain the desired upper bound for $N$ in terms of $p$.

\subsection{Branch points in stationary geodesic nets on closed Riemannian surfaces: reduction to a local version}

In a nutshell, Theorem 2 reduces to a version of the problem about polygon partitions at the beginning of this section when the polygons of the partition are flat, but there is a small positive or negative curvature living at the vertices of the partition so that the total absolute curvature is very small. Then one can proceed more or less like
it was discussed in the previous subsection for flat partitions.

Here are some details: After a rescaling, we can assume that the (sectional) curvature of the surface is in the interval $[-1,1]$. Let $d$ be an upper bound for the diameter of the surface, and $v$ is a positive lower bound for its area. If the Euler characteristic of the surface is not $0$, the Gauss-Bonnet theorem implies that $v\geq 2\pi\chi(M)\geq 2\pi$, where $\chi(M)$ is the Euler characteristic of the surface, so we can simply take $v=2\pi$. A modern version of Cheeger's lower bound
for the injectivity radius of Riemannian manifolds and the inequality
relating the convexity radius with the injectivity radius imply that the convexity radius, $conv(M)$ of $M$ is not less than ${v\over 4\exp(d)}$.
For $r\leq {conv(M)\over 2}$ we also have lower bound
${8\over\pi}r^2$ for the area of metric discs of radius $r$ ([CCLW]). The Bishop comparison
theorem implies that the areas of all balls of radius $R$
in $M$ will be less than the volumes of the corresponding
balls in the hyperbolic space. In particular, these volumes will be less than $\pi\exp(d)$. Therefore,
the lower bound ${v\over 4\exp(d)}$ for $conv(M)$ is strictly less than $1$.

Given a closed geodesic net of length $l$ in $M$ consider $r={1\over 2000}\min\{1, conv(M), {1\over l}\}$. For each $a\in M$ use the coarea formula to find $r_*\in ({r\over 2}, r)$ such that the geodesic
net intersect the geodesic circle of radius $r_*$ centered at $a$ transversally and not at vertices, and the number of points 
of intersection counted with integer multiplicities equal to the weight of the corresponding edges of the geodesic net is at most $p= {2l\over r}$.
Denote the resulting geodesic net with the imbalance $\leq p$ in the ball of radius $r_*$ centered at $a$ by $T$.

Later we are going to prove Proposition 2
(which is the local version of Theorem 2) that  asserts that for $l\geq c(M)=\max\{1,{1\over conv(M)}\}$,
the number of branch points of $T$ is at most $(350l)^{(180l)^4}$. For $l\leq c(M)$ the number of the branch points is at most $(350c(M))^{(180c(M))^4}$.
(We are going to explain the proof of this key result in the next subsection.)

Once this result is established we can prove the global estimate using the following standard ideas (originally
due to Gromov) from the comparison geometry. Pack in $M$
the maximal possible number of disjoint balls of radius ${r\over 2}$. Then (1) the concentric balls of radius $r$
cover all $M$; (2) the number of these balls does not exceed the ratio of the volume of $M$ to ${8\over \pi}r^2$ yielding the upper bound ${\pi^2\over 8}{\exp(d)\over r^2}$. 

Thus, the total number of the branch points of the net does not exceed ${\pi^2\over 8}{\exp(d)\over r^2}$ times
$(350l)^{(180l)^4}$ for $l\geq c(M)$, and $(350c(M))^{(180c(M))^4}$, if $l<c(M)$. An easy
calculation implies that these products do not exceed $\exp(d)(375l)^{(180l)^4}$ and $\exp(d)(375c(M))^{(180c(M))^4}$. If $l\geq \exp(d)$, then
the first of these upper bounds can be replaced by $(400l)^{(180l)^4}$. Also, one can replace the lower bound $l\geq c(M)$ by the larger lower bound $l\geq {\exp(d)\over \min\{{v\over 4},1\}}$. Similarly, if
$l\leq \tilde c(M)={\exp(d)\over \min\{{v\over 4}, 1\}}$,
the number of branch points is less than $(400\tilde c(M))^{(180\tilde c(M))^4}$, and we obtain Theorem 2.
%


\subsection{Geodesic nets on surfaces: proof of the local version}
Let $T$ be a geodesic net on $M$ of length $\leq l$ contained in a convex
metric disk with unbalanced vertices on the boundary
of this ball and imbalance $\leq p$. For all sufficiently large $l$ $p=4000\ l^2$.
Proceeding as it will be described in section 3.3, we are going to obtain a dual subdivison of a polygon of perimeter $p$ into flat convex polygons. The number of these polygons will be equal to 
the number of the branch points of the net in the disk. However, the ambient polygon will not be flat. It will have the curvature (=the angular defects) at its vertices. The vertices will correspond to domains in the complement of the geodesic net to the metric disk. The value of the angular defect at a vertex will be equal to the total curvature in the corresponding domain. As we assumed that the absolute value of the curvature of the original
surface does not exceed $1$, the absolute values of the angular defects at the vertices will not exceed the areas of the corresponding domains in the metric disk. Their sum will not exceed the area of disk that can be majorized by $\pi\exp(r)\leq 4r^2$. (The last inequality
holds as we consider only $r\leq {1\over 2000}$.)

We will need to establish an analogue of Theorem 1 for polyhedral disks with small curvature living at vertices. This will use the same
ideas as the proof of Theorem 1 but there will be certain technical complications. First, we reduce the general case to the case of a polyhedral surface $N\subset T$, where all faces are ``thin" polygons.
Its $1$-skeleton is the union of many broken lines connecting a pair of vertices 
with the same number $p$ of edges of length $1$ with angles between adjacent vertices that are close to
$\pi$. These paths are not vertex- or even edge-disjoint,
but we are able to find a large number $n$ of vertex-disjoint paths connecting $A$ and $B$ of length $l\leq p$. 

The analogue of Lemma 6 for polyhedral surfaces
(namely, Corollary 8)
asserts that the sum of all angles at inner vertices of each of these paths is at least
$\sqrt{2c\over p}-\mu$, where
$\mu$ is the total angular defect of the polyhedral surface $N$ with ``thin" polygonal faces. The analogue 
of Lemma 7 for polyhedral surfaces (Lemma 9) implies that all these $n$ paths are $\max\{\sqrt{2pc+{p^2\mu^2\over 3}}, {p\mu\over \sqrt{2}}\}$-close to a minimal geodesic connecting $A$ and $B$. (Corollary 9 yields an upper bound for the area of neighborhoods of the minimal geodesic between $A$ and $B$.) This means that we also
need to estimate $\mu$ for the thin polyhedral surfaces.
The idea is that the thin polyhedral surfaces correspond to a part of the original geodesic net in a thin strip of the metric disk of radius $r_*$ in $M$.
This idea leads to an upper bound for $\mu$ of the form $const\ l^2\sqrt{c}$ (Lemma 10) which is much more than term $\sqrt{2c\over p}$ in Corollary 8. This makes the lower bound in Corollary 8 negative. On the first glance it might seem that this bound cannot be applied in our situation, but it can be made useful 
by partitioning the thin surface with $n$ vertex-disjoint paths into $20000\ l^2\sqrt{p}$ polyhedral surfaces with at least $\left\lfloor{n\over 20000 l^2\sqrt{p}}\right\rfloor$ paths. At least one of these surfaces has total absolute curvature $\leq\sqrt{c\over 2p}$, which is sufficient for an application of Corollary 8 that together with Lemma 9 and its Corollary 9 
yields lower bound ${n\over l^2\sqrt{p}}$
for the area of the surface with the smallest total curvature which is sufficient for our purposes.

\subsection{ Plan of the rest of the paper}
In the next section we introduce (metric) dual complexes for geodesic nets in the plane and in convex metric discs on surfaces. In section 4 we prove that the dual complex of a geodesic net with many edges contains a ``thin" subcomplex, where all edges nearly parallel.
Moreover, a pair of vertices in this subcomplex is connected by many vertex-disjoint paths of the same length $l\leq p$. In section 5.1 we prove Theorem 1 modulo lemmata 6 and 7. In section 5.2 we prove that given such a path between $A$ and $B$ one can estimate below the sum of the absolute values of the turning angles at the inner vertices in terms of $p$ and $c=l-|AB|$. First, we prove
this result in the plane case (Lemma 6, Corollary 6),
and then for polyhedral surfaces (Corollaries 7 and 8).
In section 5.3 we prove Lemma 7 that asserts that in the plane case the distance from the broken line to the straight line $(AB)$ does not exceed
$\sqrt{2cp}$. Then we generalize it to polyhedral surfaces (Lemmata 8 and 9). In Corollary 9
we prove an upper bound for the area of the part of a polyhedral surface sufficiently close to the geodesic $(AB)$, where an upper bound for the distance is provided by Lemma 9. We start section 6 from Lemma 10 that gives an upper bound for the total absolute
curvature of the ``thin" polyhedral surface dual to a geodesic net in a small convex disk in $M$ in terms of $\sqrt{c}$,
where $c$ is calculated for two ``almost straight" paths forming the boundary of the considered surface. (The proof uses a lemma about the length of orthogonal projections of broken geodesics on surfaces that is proven in section 7.) Finally, section 6 ends with a proof of Proposition 2 providing a local version of Theorem 2. (As we saw in section 2.2, Proposition 2 implies Theorem 2.)

%


\section{Dual complexes of nets}
\subsection{Extensions of geodesic nets on the plane}
Given a geodesic net $N$, take its vertex $v$. As was noticed in section 3.4 of [NP], if $v$ is unbalanced, one can make $v$ balanced by attaching to the geodesic net $2$
new rays if $i(v)\in (0,1)$, one new ray (in the direction of $I(v)$), if $i(v)=1$, and 
$\lfloor i(v)\rfloor+1$ new rays,
if $i(v)>1$.
We can perform this operation with all vertices, adding newly created vertices at the intersections of newly added rays with already existing edges and other newly added rays. It is easy to see that the total number of vertices in
the new object is bounded in terms of the imbalance of the original net and the number
of vertices in the original net. Take the intersection of the newly created geometric object with a circle $C$ of a very large radius. More precisely, we choose $C$ so that (a) the net is contained in the interior of $C$; and (b) all points of intersection of the newly added rays are also contained in the interior of $C$.
The points of intersection are precisely
the points of intersection of each newly created ray with $C$. Remove
the unbounded parts of the rays outside of the circle. As a result, we will obtain
a geodesic net. Its unbalanced points will be precisely its points on $C$.
The imbalance vector at each of these points will have the norm $1$, and the total number
of these points will be equal to the number $p$
of newly added rays. Note, that $p<I+2U\leq 2(I+U)$. (The term $2U$ appears in the case when all non-zero values of $i(v)$ are very small positive, but we will need two new rays to balance this vertex. If $i(v)\geq 1$, we need at most $i(v)+1$ new rays to balance $v$.)
All unbalanced vertices are located on a circle, and all balanced vertices are in the disc bounded by this circle. The resulting net will be called
a {\it convex extension} of $N$. When we extended the rays, we create new balanced vertices that appear as the points of intersection of the rays with the edges of the graph. The number of these points of intersection does not exceed $pE$, where $E\leq 3V-6$ denotes the number of distinct edges of the original net, and $V$
the number of its vertices. 
(The inequality $E\leq 3V-6$ holds for all planar graphs.) We also add new vertices that are the points of intersection of pairs of newly added rays. The number of such vertices does not exceed
$\frac{p^2-p}{2}$.

Each edge of this net is (a) an edge of the original net; (b) or, a segment of an edge of the original net; (c) or, a segment in one of the new rays. The edges in (b) appear when an edge of the original net is cut by one or more of the rays; the edges in (c) appear because the rays can intersect each other.
We assign weight $1$ to all edges of type (c). Each edge of types (a), (b) corresponds to an edge of the original net, and will inherit its weight.
Let us summarize the main features of this construction:
\begin{theorem}
    Given a plane geodesic net $Net$ with $U$ unbalanced vertices, $B$
    balanced vertices, imbalance $I$, define extended imbalance $p$ of $Net$  as $\Sigma_{\{v\in U| i(v)>1\}} \lfloor i(v)\rfloor+U+|\{v\in U |i(v)\in (0,1)\}|$. Then there exists a plane  geodesic net $Net'$ with $U'=p<I+2U\leq 2(I+U)$ unbalanced vertices that contains $Net$ as its subnet, such that:
    \par\noindent
    (1) The number $B'$ of balanced vertices of $Net'$ is at least $B$, $B'< B+(3(U+B)-6)(I+2U)+\frac{(I+2U)(I+2U-1)}{2}$;
    \par\noindent
    (2) For each unbalanced vertex $v$ of $Net'$ the imbalance at $v$ is equal to $1$. Therefore the imbalance of $Net'$ is equal to the number $U'=p$ of unbalanced vertices of $Net'$;
    \par\noindent
    (3) All unbalanced vertices are located on a circle in the ambient plane; all balanced vertices are located in the open disc bounded by this circle.
    \par\noindent
    (4) If the weight of an edge of $Net'$ is not equal to $1$, it is equal to the weight of some edge of $Net$.
\end{theorem}

\subsection{Dual complexes for geodesic nets in the plane.}
Consider a geodesic net $N$ in $\R^2$.
Then, for any balanced vertex $v$ the condition for $N$ to be a geodesic net is: 
$\sum_{e\ni v} w(e)\frac{e}{\|e\|} = 0$. (Here $w(e)$ is a positive integer weight of $e$, and each edge $e$ is directed from $v$.)
For each edge $e$ 
consider the {\it unit}  vector $e^*$ that is orthogonal to $e$ such $e\wedge e^*>0$.
Mark the initial  end of $e^*$ with the name of the face of $N$ adjacent to $e$ and located clockwise of $e$, and the terminal end with the name of the face located counterclockwise of $e$. Note that $e^*$ can be obtained as the result of rotating $\frac{e}{\Vert e\Vert}$ counterclockwise by $\frac{\pi}{2}$. Therefore, the sum of $w(e)e^*$ over all edges $e$ incident to a vertex of $N$ is also equal to $0$.

Now for each balanced vertex $v$ order all incident edges
in a counterclockwise order $e_i$, $i=1,\ldots , deg\ v$, and connect the correspondent vectors $w(e_i)e^*_i$ by gluing the terminal vertex of $e^*_i$ to the initial vertex of $e^*_{i+1}$.
As $\Sigma_{i=1}^{deg\ v}w(e_i)e^*_i=0$, the terminal vertex of $e^*_{deg\ v}$ will coincide with
the initial vertex of $e^*_1$, and we will obtain a convex polygon where each side has integer length. 

We are going to have one polygon for each vertex $v$.
Note that when a vertex of polygon is regarded as the terminal vertex of one of the two
incident edges of the polygon or initial vertex of another incident edge, this vertex is marked by the name
of the same face of $N$ located between the edges of $N$ corresponding to the two considered edges of the polygon.

Now we will take the disjoint union of polygons corresponding to all balanced vertices of $N$, and take a quotient map that glues them along edges as follows:
Note that each edge $e$ is incident to two vertices. If both these vertices are balanced, the corresponding edge $e^*$ will appear in two polygons that correspond to the vertices of $e$ and will yield sides of the same length in both these polygons. For each such edge $e$ of $N$ we will identify two its copies in different polygons matching vertices
marked by the name of the same face. If only one end vertex of $e$ is balanced, then
the corresponding edge will not be identified with anything.

As the result we will obtain a metric $2$-dimensional complex $K$. Its faces are convex polygons with edges of integer length and correspond to vertices of $N$. Its edges correspond to edges of $N$ as explained above. Each vertex $w$ of $K$ corresponds to a face of $N$ in the following sense: all edges of $K$ incident to $w$ are marked by
the name of the same face $F$ at their vertices corresponding to $w$. These edges were
obtained from edges of $N$ at the boundary of $F$, and appear in the same cyclic order
as the corresponding edges $e_1,\dots , e_k$ of $N$ in $\partial F$. If all vertices of
$F$ are balanced, then we will see one edge of $K$ incident to $w=w_F$ for each edge of $N$
in the boundary of $F$. Assume that all vertices of $F$ are balanced.
The angles at $w_F$ between consecutive edges of $K$ incident to $w$ are equal to $\pi$ minus the angles between the corresponding edges $e_i, e_{i+1}$ of $F$. As the result the sum of all angles of $K$ at $w$ is equal to $k\pi-(k-2)\pi=2\pi$, and the complex $K$ is flat at $w$ (see Figure~\ref{fig:dual-flat}).

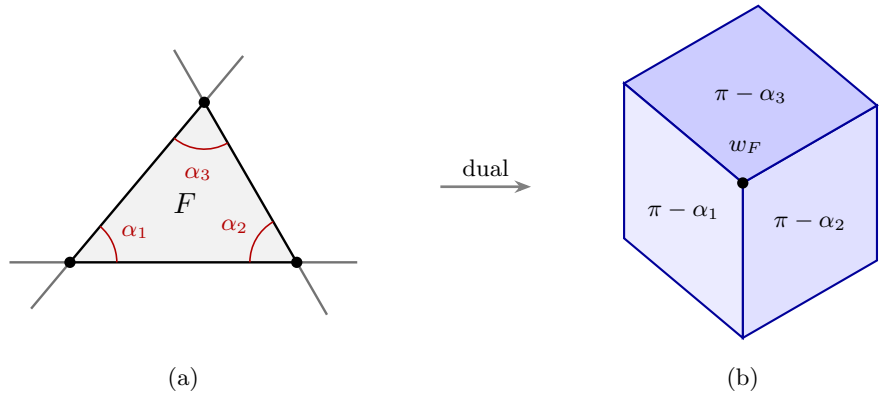
\begin{figure}[ht]
\centering
\resizebox{\ifdim\width>\textwidth \textwidth\else\width\fi}{!}{%
\begin{tikzpicture}[
   >={Stealth[length=2.2mm]},
   edge/.style={line width=0.9pt, black},
   stub/.style={line width=0.9pt, black!55},
   cell/.style={draw=blue!60!black, line width=0.8pt},
   ang/.style={line width=0.6pt, red!70!black},
   dot/.style={circle, fill=black, inner sep=1.5pt},
   scale=1.0]

\begin{scope}
  \coordinate (V1) at (0.000,0.000);
  \coordinate (V2) at (3.000,0.000);
  \coordinate (V3) at (1.777,2.117);
  \fill[black!5] (V1)--(V2)--(V3)--cycle;
  \draw[edge] (V1)--(V2)--(V3)--cycle;
  \draw[stub] (V1) -- (-0.514,-0.613);
  \draw[stub] (V1) -- (-0.800,0.000);
  \draw[stub] (V2) -- (3.800,0.000);
  \draw[stub] (V2) -- (3.400,-0.693);
  \draw[stub] (V3) -- (1.377,2.810);
  \draw[stub] (V3) -- (2.291,2.730);
  \draw[ang] (V1) ++(0.0:0.62) arc (0.0:50.0:0.62);
  \node[red!70!black,font=\footnotesize] at (0.861,0.401) {$\alpha_{1}$};
  \draw[ang] (V2) ++(120.0:0.62) arc (120.0:180.0:0.62);
  \node[red!70!black,font=\footnotesize] at (2.177,0.475) {$\alpha_{2}$};
  \draw[ang] (V3) ++(230.0:0.62) arc (230.0:300.0:0.62);
  \node[red!70!black,font=\footnotesize] at (1.694,1.171) {$\alpha_{3}$};
  \foreach \p in {(V1),(V2),(V3)} \node[dot] at \p {};
  \node at (1.52,0.78) {$F$};
  \node[font=\footnotesize] at (1.5,-1.55) {(a)};
\end{scope}

\draw[->, line width=1pt, gray] (4.9,1.0) -- (6.1,1.0) node[midway,above,font=\footnotesize,black] {dual};

\begin{scope}[shift={(8.900,1.050)}]
  \coordinate (w) at (0,0);
  \filldraw[cell,fill=blue!12] (0.000,0.000) --   (0.000,-2.050) --   (1.775,-1.025) --   (1.775,1.025) -- cycle;
  \node[font=\footnotesize] at (330.0:1.025) {$\pi-\alpha_{2}$};
  \filldraw[cell,fill=blue!20] (0.000,0.000) --   (1.775,1.025) --   (0.205,2.343) --   (-1.570,1.318) -- cycle;
  \node[font=\footnotesize] at (85.0:1.176) {$\pi-\alpha_{3}$};
  \filldraw[cell,fill=blue!8] (0.000,0.000) --   (-1.570,1.318) --   (-1.570,-0.732) --   (0.000,-2.050) -- cycle;
  \node[font=\footnotesize] at (205.0:0.866) {$\pi-\alpha_{1}$};
  \node[dot] at (w) {};
  \node[font=\footnotesize] at (85.0:0.48) {$w_F$};
\end{scope}
\node[font=\footnotesize] at (8.90,-1.55) {(b)};
\end{tikzpicture}
}
\caption{A face $F$ of the net and the flat vertex $w_F$ of the dual complex it gives rise to; drawn for a triangular face, whose vertices are then $4$-valent, so each of the three cells at $w_F$ is a unit rhombus.}
\label{fig:dual-flat}
\end{figure}

The boundary of $K$ will consist of edges that were not identified with anything. We will call $K$ the {\it dual complex} of $N$.

\begin{remark}[An example: Parsch's net]
F.~Parsch ([P], [NP, \S 3.4]) constructed an explicit planar geodesic net with $U=4$ unbalanced vertices at the corners of a square and $B=16$ balanced vertices, all edges of weight $1$. The net and its dual are shown in Figure~\ref{fig:parsch}. Each balanced vertex of the net on the left corresponds to a single convex cell of the tiling on the right; each face of the net corresponds to a vertex of the dual; each edge of the net corresponds to an edge of the dual of the same integer length (here, length~$1$). The dual net has been rotated by $90$ degrees.
\end{remark}

\begin{figure}[ht]
\centering
\definecolor{xdxdff}{rgb}{0.49,0.49,1}
\definecolor{uuuuuu}{rgb}{0.267,0.267,0.267}
\definecolor{ududff}{rgb}{0.302,0.302,1}
\definecolor{qqqqff}{rgb}{0,0,1}
\begin{tikzpicture}[scale=1.1,line cap=round,line join=round,line width=1pt]
\draw (-5.5,8)-- (-6.1340,5.6340);
\draw (-8.5,5)-- (-6.1340,5.6340);
\draw (-6.1340,5.6340)-- (-6.0876,5.5876);
\draw (-5.5,8)-- (-4.8660,5.6340);
\draw (-4.8660,5.6340)-- (-4.9124,5.5876);
\draw (-4.9124,5.5876)-- (-2.5,5);
\draw (-4.8660,5.6340)-- (-2.5,5);
\draw (-5.5,8)-- (-4.9124,5.5876);
\draw (-5.5,5.7307)-- (-4.9124,5.5876);
\draw (-5.5,5.7307)-- (-4.9780,5.5220);
\draw (-5.5,5.7307)-- (-6.0220,5.5220);
\draw (-6.0876,5.5876)-- (-5.5,5.7307);
\draw (-4.9124,5.5876)-- (-4.9780,5.5220);
\draw (-6.0876,5.5876)-- (-6.0220,5.5220);
\draw (-6.0220,5.5220)-- (-6.2307,5);
\draw (-6.0876,5.5876)-- (-6.2307,5);
\draw (-4.9124,5.5876)-- (-4.7693,5);
\draw (-4.9780,5.5220)-- (-4.7693,5);
\draw (-4.7693,5)-- (-4.9780,4.4780);
\draw (-4.9780,4.4780)-- (-5.5,4.2693);
\draw (-5.5,4.2693)-- (-4.9124,4.4124);
\draw (-4.9124,4.4124)-- (-4.7693,5);
\draw (-4.9780,4.4780)-- (-4.9124,4.4124);
\draw (-4.9124,4.4124)-- (-4.8660,4.3660);
\draw (-6.2307,5)-- (-6.0220,4.4780);
\draw (-6.0220,4.4780)-- (-5.5,4.2693);
\draw (-5.5,4.2693)-- (-6.0876,4.4124);
\draw (-6.0876,4.4124)-- (-6.2307,5);
\draw (-6.0220,4.4780)-- (-6.0876,4.4124);
\draw (-6.0876,4.4124)-- (-6.1340,4.3660);
\draw (-6.2307,5)-- (-8.5,5);
\draw (-5.5,4.2693)-- (-5.5,2);
\draw (-5.5,2)-- (-4.8660,4.3660);
\draw (-4.8660,4.3660)-- (-2.5,5);
\draw (-8.5,5)-- (-6.1340,4.3660);
\draw (-6.1340,4.3660)-- (-5.5,2);
\draw (-5.5,8)-- (-5.5,5.7307);
\draw (-6.0876,5.5876)-- (-5.5,8);
\draw (-6.0876,5.5876)-- (-8.5,5);
\draw (-4.7693,5)-- (-2.5,5);
\draw (-4.9124,4.4124)-- (-5.5,2);
\draw (-4.9124,4.4124)-- (-2.5,5);
\draw (-6.0876,4.4124)-- (-8.5,5);
\draw (-6.0876,4.4124)-- (-5.5,2);
\begin{scope}[shift={(-2,0)}]
    \draw (3,3)-- (4,3);
    \draw (3,3)-- (3.2412,3.9705);
    \draw (4,3)-- (3.7588,3.9705);
    \draw (3.2412,3.9705)-- (3.5,4.9364);
    \draw (3.5,4.9364)-- (3.7588,3.9705);
    \draw (2.5341,4.6776)-- (3.2412,3.9705);
    \draw (2.5341,4.6776)-- (3.5,4.9364);
    \draw (3.5,4.9364)-- (4.4659,4.6776);
    \draw (4.4659,4.6776)-- (3.7588,3.9705);
    \draw (5.4364,4.4364)-- (5.4364,5.4364);
    \draw (1.5636,5.4364)-- (1.5636,4.4364);
    \draw (2.5341,4.6776)-- (1.5636,4.4364);
    \draw (4.4659,4.6776)-- (5.4364,4.4364);
    \draw (2.5341,5.1952)-- (3.5,4.9364);
    \draw (3.5,4.9364)-- (4.4659,5.1952);
    \draw (2.5341,5.1952)-- (1.5636,5.4364);
    \draw (4.4659,5.1952)-- (5.4364,5.4364);
    \draw (3.2412,5.9023)-- (2.5341,5.1952);
    \draw (3.7588,5.9023)-- (4.4659,5.1952);
    \draw (3.2412,5.9023)-- (3.5,4.9364);
    \draw (3.5,4.9364)-- (3.7588,5.9023);
    \draw (3.2412,5.9023)-- (3,6.8728);
    \draw (3.7588,5.9023)-- (4,6.8728);
    \draw (3,6.8728)-- (4,6.8728);
    \draw (4,6.8728)-- (4.9705,7.1140);
    \draw (4.9705,7.1140)-- (5.6776,6.4069);
    \draw (5.6776,6.4069)-- (5.4364,5.4364);
    \draw (1.3224,6.4069)-- (2.0295,7.1140);
    \draw (2.0295,2.7588)-- (1.3224,3.4659);
    \draw (5.6776,3.4659)-- (4.9705,2.7588);
    \draw (2.0295,7.1140)-- (3,6.8728);
    \draw (1.3224,6.4069)-- (1.5636,5.4364);
    \draw (1.5636,4.4364)-- (1.3224,3.4659);
    \draw (2.0295,2.7588)-- (3,3);
    \draw (4,3)-- (4.9705,2.7588);
    \draw (5.6776,3.4659)-- (5.4364,4.4364);
    \draw (2.0295,7.1140)-- (1.0636,7.3728);
    \draw (1.0636,7.3728)-- (1.3224,6.4069);
    \draw (1.3224,3.4659)-- (1.0636,2.5);
    \draw (1.0636,2.5)-- (2.0295,2.7588);
    \draw (4.9705,2.7588)-- (5.9364,2.5);
    \draw (5.9364,2.5)-- (5.6776,3.4659);
    \draw (5.6776,6.4069)-- (5.9364,7.3728);
    \draw (5.9364,7.3728)-- (4.9705,7.1140);
\end{scope}
\begin{scriptsize}
\foreach \p in {(-5.5,8),(-2.5,5),(-5.5,2),(-8.5,5)}
  \fill[qqqqff] \p circle (1.5pt);
\foreach \p in {(-6.0876,5.5876),(-4.9124,5.5876),(-6.0220,5.5220),(-4.9780,5.5220),
  (-6.0876,4.4124),(-6.0220,4.4780),(-4.9780,4.4780),(-4.9124,4.4124)}
  \fill[qqqqff] \p circle (0.8pt);
\foreach \p in {(-6.1340,5.6340),(-4.8660,5.6340),(-6.1340,4.3660),(-4.8660,4.3660),
  (-5.5,5.7307),(-6.2307,5),(-4.7693,5),(-5.5,4.2693)}
  \fill[qqqqff] \p circle (0.8pt);
\begin{scope}[shift={(-2,0)}]
    \foreach \p in {(3,3),(4,3),(3.5,4.9364),(3.2412,3.9705),(3.7588,3.9705),
    (4.4659,4.6776),(2.5341,4.6776),(5.4364,4.4364),(5.4364,5.4364),
    (1.5636,5.4364),(1.5636,4.4364),(2.5341,5.1952),(4.4659,5.1952),
    (3.2412,5.9023),(3.7588,5.9023),(3,6.8728),(4,6.8728),
    (4.9705,7.1140),(5.6776,6.4069),(1.3224,6.4069),(2.0295,7.1140),
    (2.0295,2.7588),(1.3224,3.4659),(5.6776,3.4659),(4.9705,2.7588),
    (1.0636,7.3728),(1.0636,2.5),(5.9364,2.5),(5.9364,7.3728)}
    \fill[qqqqff] \p circle (1.6pt);
\end{scope}
\end{scriptsize}
\end{tikzpicture}
\caption{Parsch's example: a planar geodesic net with four unbalanced and sixteen balanced vertices (left) and its dual complex (right); vertex positions are exaggerated slightly for clarity.}
\label{fig:parsch}
\end{figure}
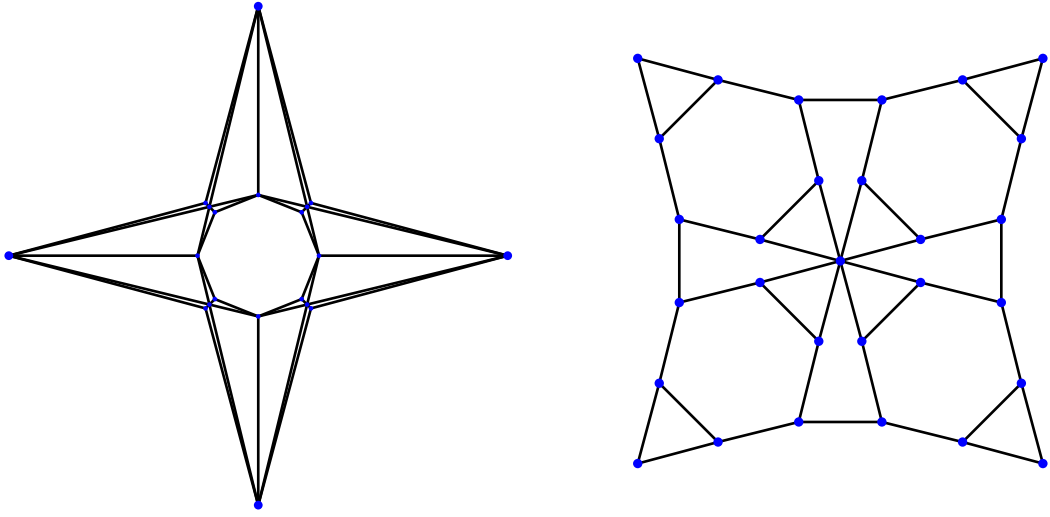

Now apply this construction to a convex extension a geodesic net. All the vertices with the
exception of vertices of degree $1$ at the end of the added edges (on the added rays) are balanced.
All interior vertices of the resulting complex $K$ will correspond to bounded faces, and $K$ will
be flat at these vertices. The number of edges in the boundary of $K$ is equal
to the number of vertices of degree $1$.  
Only edges dual to edges on rays leading to the vertices of degree $1$ will belong to the boundary of $K$. The cyclic order will be the same
as the order of the corresponding vertices on $C$, and the angles of the polygon (or polygons) at the boundary will be positive and less than $\pi$. Indeed, each angle is equal $\pi$ minus
the angle between two consecutive rays, and these angles will be all positive as the rays do not intersect beyond the circle $C$. 

As the sum of the imbalance vectors at all vertices
of each geodesic net is zero, the sum of these vectors is equal to $0$. The same will be true for the perpendicular vectors that form the boundary of $K$.
So, the boundary of $K$ will be a closed
broken line.

The unit tangent vectors to the rays at the points of intersection with $C$ can be extended to an injective unit vector field $f$ on $C$, that is, a unit vector field with monotonously changing angle on each arc of $C$ 
between consecutive vertices. This vector field is isotopic to the normal vector field to $C$, and, therefore, has the same rotation angle $2\pi$. 
Moreover, since the angles between consecutive imbalance vectors are positive, the sum of any proper subset of the set of these angles will be less that $2\pi$. 

This means that the boundary of $K$ will
be a simple convex closed curve. Now the combinatorial Gauss-Bonnet formula implies that the flat PL-manifold with boundary $K$ has Euler characteristic $1$, and, therefore, is homeomorphic to the $2$-disc. Now it is easy to see that the flat simplicial complex $K$ can be embedded  to the Euclidean plane.


Figure~\ref{fig:parsch-ext} shows this for Parsch's example: each of the four concavities along the sides of the inner dual (the same dual as in Figure~\ref{fig:parsch}, right) is filled by a chain of new unit-sided cells, producing a $20$-gon boundary.

\begin{figure}[ht]
\centering
\definecolor{qqqqff}{rgb}{0,0,1}
\begin{tikzpicture}[scale=1.1,line cap=round,line join=round,line width=1pt]
\draw (3,3)-- (4,3);
\draw (3,3)-- (3.2412,3.9705);
\draw (4,3)-- (3.7588,3.9705);
\draw (3.2412,3.9705)-- (3.5,4.9364);
\draw (3.5,4.9364)-- (3.7588,3.9705);
\draw (2.5341,4.6776)-- (3.2412,3.9705);
\draw (2.5341,4.6776)-- (3.5,4.9364);
\draw (3.5,4.9364)-- (4.4659,4.6776);
\draw (4.4659,4.6776)-- (3.7588,3.9705);
\draw (5.4364,4.4364)-- (5.4364,5.4364);
\draw (1.5636,5.4364)-- (1.5636,4.4364);
\draw (2.5341,4.6776)-- (1.5636,4.4364);
\draw (4.4659,4.6776)-- (5.4364,4.4364);
\draw (2.5341,5.1952)-- (3.5,4.9364);
\draw (3.5,4.9364)-- (4.4659,5.1952);
\draw (2.5341,5.1952)-- (1.5636,5.4364);
\draw (4.4659,5.1952)-- (5.4364,5.4364);
\draw (3.2412,5.9023)-- (2.5341,5.1952);
\draw (3.7588,5.9023)-- (4.4659,5.1952);
\draw (3.2412,5.9023)-- (3.5,4.9364);
\draw (3.5,4.9364)-- (3.7588,5.9023);
\draw (3.2412,5.9023)-- (3,6.8728);
\draw (3.7588,5.9023)-- (4,6.8728);
\draw (3,6.8728)-- (4,6.8728);
\draw (4,6.8728)-- (4.9705,7.1140);
\draw (4.9705,7.1140)-- (5.6776,6.4069);
\draw (5.6776,6.4069)-- (5.4364,5.4364);
\draw (1.3224,6.4069)-- (2.0295,7.1140);
\draw (2.0295,2.7588)-- (1.3224,3.4659);
\draw (5.6776,3.4659)-- (4.9705,2.7588);
\draw (2.0295,7.1140)-- (3,6.8728);
\draw (1.3224,6.4069)-- (1.5636,5.4364);
\draw (1.5636,4.4364)-- (1.3224,3.4659);
\draw (2.0295,2.7588)-- (3,3);
\draw (4,3)-- (4.9705,2.7588);
\draw (5.6776,3.4659)-- (5.4364,4.4364);
\draw (2.0295,7.1140)-- (1.0636,7.3728);
\draw (1.0636,7.3728)-- (1.3224,6.4069);
\draw (1.3224,3.4659)-- (1.0636,2.5);
\draw (1.0636,2.5)-- (2.0295,2.7588);
\draw (4.9705,2.7588)-- (5.9364,2.5);
\draw (5.9364,2.5)-- (5.6776,3.4659);
\draw (5.6776,6.4069)-- (5.9364,7.3728);
\draw (5.9364,7.3728)-- (4.9705,7.1140);
\draw (1.0636,7.3728)-- (0.8048,6.4069);
\draw (0.8048,6.4069)-- (0.5636,5.4364);
\draw (0.5636,5.4364)-- (0.5636,4.4364);
\draw (0.5636,4.4364)-- (0.8048,3.4659);
\draw (0.8048,3.4659)-- (1.0636,2.5);
\draw (1.0636,7.3728)-- (2.0295,7.6316);
\draw (2.0295,7.6316)-- (3,7.8728);
\draw (3,7.8728)-- (4,7.8728);
\draw (4,7.8728)-- (4.9705,7.6316);
\draw (4.9705,7.6316)-- (5.9364,7.3728);
\draw (5.9364,7.3728)-- (6.1952,6.4069);
\draw (6.1952,6.4069)-- (6.4364,5.4364);
\draw (6.4364,5.4364)-- (6.4364,4.4364);
\draw (6.4364,4.4364)-- (6.1952,3.4659);
\draw (6.1952,3.4659)-- (5.9364,2.5);
\draw (5.9364,2.5)-- (4.9705,2.2412);
\draw (4.9705,2.2412)-- (4,2);
\draw (4,2)-- (3,2);
\draw (3,2)-- (2.0295,2.2412);
\draw (2.0295,2.2412)-- (1.0636,2.5);
\begin{scriptsize}
\foreach \p in {(3,3),(4,3),(3.5,4.9364),(3.2412,3.9705),(3.7588,3.9705),
  (4.4659,4.6776),(2.5341,4.6776),(5.4364,4.4364),(5.4364,5.4364),
  (1.5636,5.4364),(1.5636,4.4364),(2.5341,5.1952),(4.4659,5.1952),
  (3.2412,5.9023),(3.7588,5.9023),(3,6.8728),(4,6.8728),
  (4.9705,7.1140),(5.6776,6.4069),(1.3224,6.4069),(2.0295,7.1140),
  (2.0295,2.7588),(1.3224,3.4659),(5.6776,3.4659),(4.9705,2.7588),
  (1.0636,7.3728),(1.0636,2.5),(5.9364,2.5),(5.9364,7.3728)}
  \fill[qqqqff] \p circle (1.6pt);
\foreach \p in {(0.8048,6.4069),(0.5636,5.4364),(0.5636,4.4364),(0.8048,3.4659),
  (2.0295,7.6316),(3,7.8728),(4,7.8728),(4.9705,7.6316),
  (6.1952,6.4069),(6.4364,5.4364),(6.4364,4.4364),(6.1952,3.4659),
  (2.0295,2.2412),(3,2),(4,2),(4.9705,2.2412)}
  \fill[qqqqff] \p circle (1.6pt);
\end{scriptsize}
\end{tikzpicture}
\caption{The dual of a convex extension of Parsch's net: the sixteen inner cells of Figure~\ref{fig:parsch}, with the four concavities filled so that the boundary becomes convex.}
\label{fig:parsch-ext}
\end{figure}
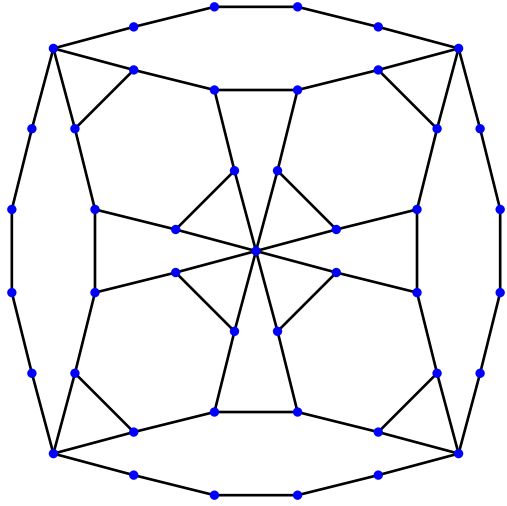


%

Thus, we obtain:
\begin{theorem}
    A dual complex of a convex extension $N'$ of a geodesic net $N$ is a flat metric disc
    bounded by a convex polygon that can be embedded in $\R^2$. The length of each edge of this polygon is equal to $1$, and the number
    of edges is equal to the extended imbalance $p$ of the net (defined in the text of Theorem 4).

    This polygon is tiled by smaller convex polygons, where all sides of all polygons have integer lengths.
    Each of these polygons  corresponds to a balanced vertex of $N'$, each vertex to a face of $N'$, each edge of these polygons corresponds to an edge of $N'$, and its length is equal to the weight of this edge in $N'$.
\end{theorem}

Recall, that $p<I+2U\leq 2(I+U)$, where $I$ is the imbalance of the net, and $U$ denotes here the number of unbalanced vertices of the net.

Before moving further, note that this construction is very similar to a construction known since XIX century called the Maxwell-Cremona correspondence.
The Maxwell-Cremona correspondence can be applied to a planar drawing of a graph and arbitrary
real weights $w_j$ assigned to all edges $E_j$ of the graph that satisfy the following stationary condition: For each vertex $v_i$, the sum of $w_jE_j$ is equal to $0$, where $E_j$ runs over the set of all edges incident to $v_i$ and regarded as vectors directed from $v_i$.
Here $E_j=v_k-v_i$ for an adjacent vertex $v_k$, and one does {\it not} divide by the length of $E_j$. Then one draws the dual graph where each edge is replaced by the perpendicular vector of length equal to the length of  $w_jE_j$, and obtained from ${\rm sign}(w_j)\ E_j$ by the counterclockwise rotation by the angle $90^o$. It is easy to see that the stationarity in our sense corresponds
to the choice of weights $w_j=\frac{1}{\rm length \ E_j}=\frac{1}{\Vert v_i-v_k\Vert}.$  The fact that this setting is non-linear, i.e.
the weights depend on the positions of the vertices makes the most of the Maxwell-Cremona theory unusable in our situation.
In particular, we do not know if given a convex polygon with sides of length $1$
and its partition into convex polygons  with integer sides we can 
always find a corresponding plane net.
\subsection{Dual complexes for nets on surfaces.}

In section 2.2 we explained how one can start from a closed geodesic net on a surface $M$ and cut out a geodesic net $N$ contained in a convex metric disk where all unbalanced vertices are on the boundary and have just one incident edge
with total imbalance $p$. Here we are going to
describe the construction of the dual complex for this case.

As before we consider all balanced vertices $v$. We consider all incident edges $e$, and look
at their tangent vectors $\tilde e$ of length $w(e)$ in the tangent space. These vectors sum to the zero vector. The same will be true for vectors $e^*$ obtained by rotating $\tilde e$
counterclockwise in the tangent plane at $v$. (We
assume that all the tangent planes are oriented using the same orientation of the ambient convex disk.) Now we can build a convex polygon
using the vectors $e^*$. We do this for all
balanced vertices and glue these polygons along all pairs of identical edges. As before, each edge appears twice in the polygons that correspond to both its endpoints, and we similarly orient
the edges so that we glue vertices corresponding to the same face in $M$ incident
to $e$. Therefore, the vertices will correspond
to faces of the original net in $M$.
However, now the sums of the angles of faces of the dual complex meeting at a vertex of the dual complex corresponding to a face $F$
of the original net are, in general, not equal to $2\pi$. Instead, this sum of angles is equal to the sum of the outer angles of $F$.
Therefore, the Gauss-Bonnet theorem implies that the angular defect of $v$ is equal to the total curvature of $F\subset M$ (see Figure~\ref{fig:surface-dual}). Another difference
is that now the (integer) lengths of boundary edges that correspond to the edges leading to unbalanced vertices are not necessarily equal to $1$. (These lengths are equal to the integer weights of the corresponding edges in the original net, and in the case of surfaces we did not make them to be equal to one.)

\begin{figure}[ht]
\centering
\resizebox{\ifdim\width>\textwidth \textwidth\else\width\fi}{!}{%
\begin{tikzpicture}[
   >={Stealth[length=2.2mm]},
   edge/.style={line width=0.9pt, black},
   stub/.style={line width=0.9pt, black!55},
   cell/.style={draw=blue!60!black, line width=0.8pt},
   ang/.style={line width=0.6pt, red!70!black},
   dot/.style={circle, fill=black, inner sep=1.5pt},
   scale=1.0]

\begin{scope}
  \coordinate (V1) at (0.000,0.000);
  \coordinate (V2) at (3.000,0.000);
  \coordinate (V3) at (1.650,2.150);
  \fill[orange!12] (0.000,0.000) .. controls (1.245,-0.192) and (1.747,-0.133) .. (3.000,0.000) .. controls (2.532,0.958) and (2.368,1.361) .. (1.650,2.150) .. controls (0.808,1.384) and (0.548,0.998) .. (0.000,0.000) -- cycle;
  \draw[edge] (0.000,0.000) .. controls (1.245,-0.192) and (1.747,-0.133) .. (3.000,0.000) .. controls (2.532,0.958) and (2.368,1.361) .. (1.650,2.150) .. controls (0.808,1.384) and (0.548,0.998) .. (0.000,0.000);
  \draw[stub] (V1) -- ++(171.25:0.80);
  \draw[stub] (V1) -- ++(241.25:0.80);
  \draw[stub] (V2) -- ++(296.06:0.80);
  \draw[stub] (V2) -- ++(366.06:0.80);
  \draw[stub] (V3) -- ++(402.31:0.80);
  \draw[stub] (V3) -- ++(492.31:0.80);
  \draw[ang] (V1) ++(-8.75:0.62) arc (-8.75:61.25:0.62);
  \node[red!70!black,font=\footnotesize] at (0.852,0.420) {$\alpha_{1}$};
  \draw[ang] (V2) ++(116.06:0.62) arc (116.06:186.06:0.62);
  \node[red!70!black,font=\footnotesize] at (2.169,0.460) {$\alpha_{2}$};
  \draw[ang] (V3) ++(222.31:0.62) arc (222.31:312.31:0.62);
  \node[red!70!black,font=\footnotesize] at (1.605,1.201) {$\alpha_{3}$};
  \foreach \p in {(V1),(V2),(V3)} \node[dot] at \p {};
  \node at (1.55,0.80) {$F$};
  \node[font=\footnotesize] at (1.5,-1.55) {(a)};
\end{scope}

\draw[->, line width=1pt, gray] (5.0,1.0) -- (6.2,1.0) node[midway,above,font=\footnotesize,black] {dual};

\begin{scope}[shift={(9.100,1.050)}]
  \coordinate (w) at (0,0);
  \filldraw[cell,fill=blue!12] (0.000,0.000) --   (2.050,0.000) --   (1.349,1.926) --   (-0.701,1.926) -- cycle;
  \node[font=\footnotesize] at (55.0:1.176) {$\pi-\alpha_{2}$};
  \filldraw[cell,fill=blue!20] (0.000,0.000) --   (-0.701,1.926) --   (-2.628,1.225) --   (-1.926,-0.701) -- cycle;
  \node[font=\footnotesize] at (155.0:1.450) {$\pi-\alpha_{3}$};
  \filldraw[cell,fill=blue!8] (0.000,0.000) --   (-1.926,-0.701) --   (-0.609,-2.272) --   (1.318,-1.570) -- cycle;
  \node[font=\footnotesize] at (255.0:1.176) {$\pi-\alpha_{1}$};
  \draw[ang] (w) ++(310.0:1.48) arc (310.0:360.0:1.48);
  \node[red!70!black,font=\footnotesize] at (335.0:2.09) {defect};
  \node[dot] at (w) {};
  \node[font=\footnotesize] at (335.0:0.52) {$w_F$};
\end{scope}
\node[font=\footnotesize] at (9.10,-1.55) {(b)};
\end{tikzpicture}
}
\caption{A geodesic face $F$ on a curved surface and the corresponding dual vertex $w_F$, whose angular defect equals the total curvature of $F$.}
\label{fig:surface-dual}
\end{figure}
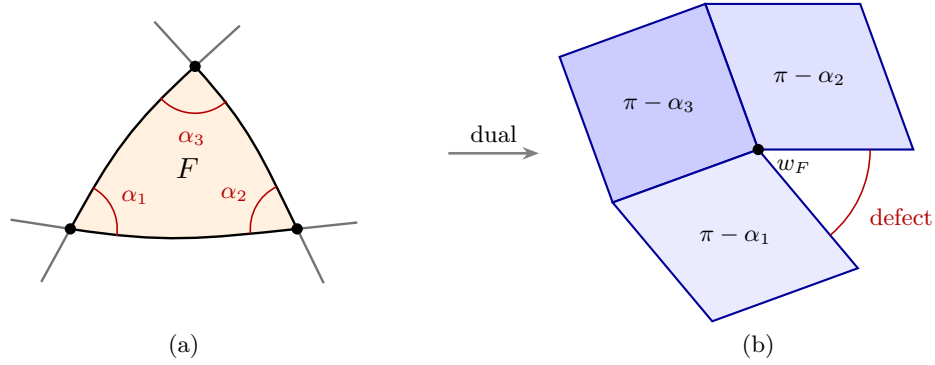

\begin{proposition}
   The length of the boundary of the dual complex of $N$ is greater than or equal to the perimeter of each of its convex cells.
\end{proposition}

\begin{proof}
We will start from the geodesic net $N$ in a convex metric ball $B$ in $M$ of radius $r_*\leq {1\over 2}conv(M)$ constructed in section 2.2.
Consider the Euclidean plane and cut out a disc $D$ of radius $\rho$ such that the length of its boundary is equal to the length of the boundary of $B$.
One can choose geodesic normal coordinates
on $B$ and use them to smooth out
the Riemannian metric on $B\cup (\R^2\setminus D)$ on the collar $A$ of $\partial(\R^2\setminus D)$ in $\R^2\setminus D$ of radius $\rho$
so that all concentric metric circles centered at the center of $B$ are convex, and the metric on $\R^2\setminus (D\cup A)$ remains flat.  Now we can extend all geodesics (edges) in $N$ intersecting $\partial B$ to the newly added part $\R^2\setminus D$ with the smoothed-out  metric.
It is easy to see that they will cross the annulus $A$ and become straight lines in the flat part of $\R^2\setminus D$. 
We obtain an extension $N'$ of $N$ where geodesic rays extended to infinity have the same weights as the corresponding edges of $N$.
Consider the number of points of intersection of $N'$ with the circle of a very large radius centered at any vertex of $N$ counted with multiplicities of the corresponding edges $N'$. This number will be equal $p$,
the number of vertices of $N$ on $\partial B$
counted with the multiplicities of incident edges.

The monotonicity formula for $N'$ implies that for each vertex $v$ of $N$ regarded
as a vertex of $N'$ the ratio ${{\rm length}(N'\cap B_v(r))\over r}$ is an increasing
function of $r$. (Here $B_v(r)$ denotes the metric ball of radius $r$ centered at $v$ in $B\cup (\R^2\setminus D)$. When $r\longrightarrow 0$, this ratio approaches the degree of $v$, when $r\longrightarrow\infty$, this ratio approaches $p$. This implies that $p\geq {\rm deg}(v)$.

But $p$ is the length of the boundary of the dual complex of $N$, and ${\rm deg}(v)$ is the length
of the boundary of the convex polygon in the dual complex corresponding to the vertex $v$ of $N$.

\end{proof}

\section{Subdividing a polygon with sides of integer length}


In this section, we demonstrate that Theorem 1 reduces to its
``thin" strip version, when all edges are nearly parallel, and all angles of all polygons are either very small or almost $\pi$.






Consider a family of polygons $Pol_n$, $n=1,2,\ldots $ in $\R^2$ with sides of integer length and perimeters $\leq p$. Assume that there exist subdivisions of each $Pol_n$ into $N_n$ convex polygons
with side of integer length, and $N_n\longrightarrow\infty$.
Consider the corresponding plane graphs $G_n$ formed by all vertices and edges of $Pol_n$ and all polygons of the subdivision. When we construct $G_n$, we use new vertices of degree $2$ to subdivide each integer side
of length $l$ into $l$ edges of length $1$.

Quantities that are bounded by a function of $p$ (but not of $n$) will be called {\it bounded}. For example,
observe that the isoperimetric inequality implies that the areas, $A_n$, of the  polygon are bounded (by $\frac{p^2}{4\pi}$).

We will be distinguishing between {\it boundary} vertices of $G_n$, namely the vertices of the original polygon on the boundary
of $Pol_n$, and the rest of vertices that are called {\it inner} vertices.

%

A polygon of fixed perimeter cannot 
contain a convex polygon with a larger perimeter. Therefore:
\begin{lemma}
    The polygons in the subdivision have perimeters $\leq p$.
\end{lemma}
As each face of $G_n$ corresponds to a vertex of the original net, we obtain the following corollary:
\begin{corollary}
    The degree of each balanced vertex in a plane geodesic net counted with multiplicities equal to the integer weights of incident edges does not exceed $p$.
\end{corollary}

Of course, this corollary also follows immediately from the {\it monotonicity formula} applied to the (convex extension of the original) geodesic net
with attached $p$ rays. This formula says
that for each point $x$ on the stationary infinite geodesic net $N$ the ratio
$\frac{length\ (N\bigcap B(x,r))}{r}$ is an
increasing function of $r$. When $x$ is a vertex, and $r$ is very small, this ratio is approximately equal to the degree of $v$, and when $r$ is large this ratio approaches $p$ (see the proof of Proposition 1.)


\begin{definition}
    An $\varepsilon$-polygon is a convex polygon with sides that have integer lengths
    with all angles being either greater than $\pi - \varepsilon$ or smaller than $\varepsilon$. We will assume that $\varepsilon$ is smaller than $\dfrac{\pi}{6}$ for convenience reasons; in practice the chosen $\varepsilon$ will be very small.
\end{definition}

\begin{lemma}
    For any fixed $\varepsilon<\frac{\pi}{6}$ the number of non-$\varepsilon$-polygons in the subdivision of $Pol_n$ does not exceed 
    $\frac{p^2}{2\varepsilon}$.
\end{lemma}

\begin{proof}
    Indeed, any non-$\varepsilon$-polygon has area greater than the triangle with two sides of length $1$ with the angle $\epsilon$ (or $\pi-\epsilon$) between them.
    Then the number $N_{\text{non-}\varepsilon}$ of non-$\varepsilon$-polygons satisfies $N_{\text{non-}\varepsilon} \sin(\varepsilon)/2 \le A$. Therefore, $N_{\text{non-}\varepsilon} \le \dfrac{2A}{\sin(\varepsilon)}\le
    \dfrac{p^2}{2\pi\ \sin(\varepsilon)}< \dfrac{3p^2}{2\pi\varepsilon}<\dfrac{p^2}{2\varepsilon}$.
\end{proof}

\par\noindent

\begin{definition}
    A vertex $v$ of the 
    $G_n$  is called $\varepsilon$\textit{-vertex} if
    (a) 
    all angles between consecutive edges incident to this vertex and measured in the counterclockwise direction are either greater than $\pi - \varepsilon$, 
    or smaller than $\varepsilon$; (b) Exactly two of these angles are greater than $\pi-\varepsilon$ (see Figure~\ref{fig:eps-polygon}). 
\end{definition}

\begin{figure}[ht]
\centering
\resizebox{\ifdim\width>\textwidth \textwidth\else\width\fi}{!}{%
\begin{tikzpicture}[
   >={Stealth[length=2.2mm]},
   edge/.style={line width=0.9pt, black},
   poly/.style={fill=blue!8, draw=blue!65!black, line width=1.0pt},
   ang/.style={line width=0.6pt, red!70!black},
   dot/.style={circle, fill=black, inner sep=1.3pt},
   scale=1.0]

\begin{scope}
  \filldraw[poly] (0,0) -- (1.5,0.20) -- (3,0.28) -- (4.5,0.20) -- (6,0)
                       -- (4.5,-0.20) -- (3,-0.28) -- (1.5,-0.20) -- cycle;
  \foreach \p in {(0,0),(1.5,0.20),(3,0.28),(4.5,0.20),(6,0),(4.5,-0.20),(3,-0.28),(1.5,-0.20)}
     \node[dot] at \p {};
  \node[red!70!black,font=\footnotesize,left] at (-0.12,0) {$<\varepsilon$};
  \node[red!70!black,font=\footnotesize,above right] at (6.05,0.02) {$<\varepsilon$};
  \node[red!70!black,font=\footnotesize,above] at (3,0.36) {$>\pi-\varepsilon$};
  \node[red!70!black,font=\footnotesize,below] at (1.5,-0.30) {$>\pi-\varepsilon$};
  \node[font=\footnotesize] at (3,-1.15) {(a)};
\end{scope}

\begin{scope}[shift={(11.4,0)}]
  \coordinate (v) at (0,0);
  \foreach \a in {-7,0,7} \draw[edge] (v) -- (\a:2.2);
  \foreach \a in {173,180} \draw[edge] (v) -- (\a:2.2);
  \draw[ang] (v) ++(7:1.0) arc (7:173:1.0);
  \node[red!70!black,font=\footnotesize,above] at (90:1.12) {$>\pi-\varepsilon$};
  \draw[ang] (v) ++(180:0.8) arc (180:353:0.8);
  \node[red!70!black,font=\footnotesize] at (270:0.45) {$>\pi-\varepsilon$};
  \node[red!70!black,font=\footnotesize,right] at (2.3,0.14) {$<\varepsilon$};
  \node[red!70!black,font=\footnotesize,left] at (-2.3,0.10) {$<\varepsilon$};
  \node[dot] at (v) {};
  \node[below right] at (v) {$v$};
  \node[font=\footnotesize] at (0,-1.15) {(b)};
\end{scope}

\end{tikzpicture}
}
\caption{(a) an $\varepsilon$-polygon; (b) an $\varepsilon$-vertex.}
\label{fig:eps-polygon}
\end{figure}
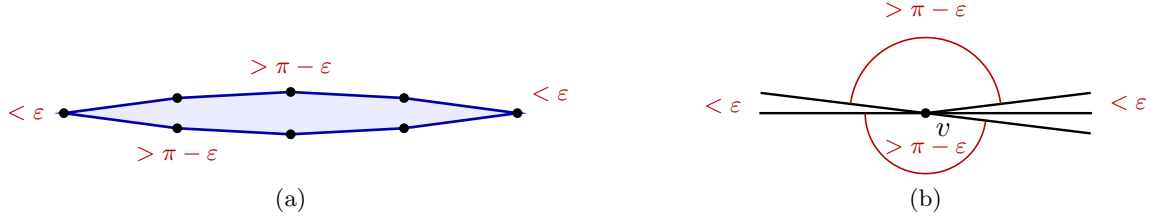

\par\noindent
{\bf Observation 1:} Observe, that each of the angles formed by a pair of consecutive edges in this definition is
an angle in one of the polygons of the partition. Therefore, all these angles are strictly less than $\pi$.
\par\noindent
{\bf Observation 2:} Let $A$ be an $\epsilon$-vertex. Assume that the angles between all pairs of edges incident to $A$ are defined so that
their values are in the interval $[0,\pi]$. Then these values will be
in the set 
$(0, 2\epsilon]\bigcup (\pi-\epsilon,\pi)$. 
(Again, here we calculate each angle clockwise or counterclockwise to ensure that
the value is in $[0,\pi]$.)
Indeed,
we have two ``large angles" between $\pi-\epsilon$ and $\pi$. Therefore, the total sum
of all small angles is at most $2\epsilon$. Each angle between two edges incident to $A$
will be either a sum of values of a collection of small angles, or a sum of values of one large angle and, possibly, several small angles.
In the first case, it will be at most $2\epsilon$, in the second greater than $\pi-\epsilon$.

\begin{lemma}
    There is a bounded number of non-$\varepsilon$-vertices in $G_n$. More precisely, for $\varepsilon<\frac{\pi}{6}$ it is bounded by $\frac{p^2}{2\epsilon}+p<\frac{p^2}{\epsilon}$.
\end{lemma}

\begin{proof}
    The number of boundary vertices does not exceed $p$. Consider inner non-$\varepsilon$-vertices $v$. If one of the incident angles to $v$ is greater than $\epsilon$ or is less than $\pi-\varepsilon$, then the incident triangle with two sides of length $1$ has the area at least ${\sin\varepsilon\over 2}\geq {3\over 2\pi}\varepsilon$. These triangles for different vertices intersect only when they coincide, and they can coincide only for at most three vertices.
    If all non-$\varepsilon$-vertices were like that, their number would be bounded by $3{{p^2\over 4\pi}\over {3\varepsilon\over 2\pi}}={p^2\over 2\varepsilon}.$
    If a $v$ vertex satisfies (a), but not (b), then a half-plane incident to $v$ would be divided
    by edges into triangular domains with angles $\leq\varepsilon\leq{\pi\over 6}$. It is easy to see that these triangles cannot be faces but are parts of polygons with more than three sides. The union of these triangles in a half-plane has area at least $1.5$. This union does not intersect
    similar unions for other vertices
    that satisfy (a) but not (b). It is easy to see that it does not intersect triangles with angles $>\varepsilon$ or $<\pi-\varepsilon$ incident to non-$\varepsilon$ vertices that were considered above. Therefore,
    non-$\varepsilon$ vertices of this type waste too much area, and
    the number of non-$\varepsilon$ vertices is the maximal, when such vertices do not arise at all.

\end{proof}

 \begin{definition}
    An $\varepsilon$-path is a path in $G_n$ such that all of its inner vertices are $\varepsilon$-vertices, and the angle between any two consecutive edges is greater than $\pi - \varepsilon$. (Here we define the angles between edges so that
    they are always in the interval $[0,\pi]$.) If both endpoints
    of this path are non-$\varepsilon$ vertices, we call it a maximal $\varepsilon$-path.
 \end{definition}
    
    Note, that the definition of $\varepsilon$-vertices implies that each $\varepsilon$-path can be extended to a maximal $\varepsilon$-path. In particular, each edge of $G_n$ is
    on at least one maximal $\varepsilon$-path.

\begin{lemma}
    Assume that $\varepsilon \leq \frac{\pi}{p}$. 
    Then 
    the length $l$ of each $\epsilon$-path is less than $p$.
    
\end{lemma}

\begin{proof}
    Let $e_i$, $i=1,2,\ldots, N$, be the unit vectors on the edges in a $\varepsilon$-path in the order in which they appear. (Here we represent an edge of length $l>1$ as the sum of $l$ unit vectors.) Consider vectors $b_i=e_1+\ldots +e_i$.

    
    For each $k$ project vectors $e_1,\ldots, e_k$ to $b_k$. The lengths of these projections will be equal
    to $\cos\angle(e_i, b_k)$. Observe that $\Vert b_k\Vert< {p\over 2}$. In order to prove an upper
    bound on $N$, we need to find $k$ such that $\Vert b_k\Vert\geq {p\over 2}$. In this case we will be able to conclude that $N<k$. Thus, we need a lower bound for $\Vert b_k\Vert$. In order to make $\Vert b_k\Vert$ as small as possible, we need to keep
    the angles between $e_i$ and $b_k$ as large as possible. Obviously, we need to have angles between $e_i$ and $e_{i+1}$ to have maximal possible values,
    namely, ${\pi\over p}$. In this case, 
    $\Vert b_k\Vert=\Sigma_{j=1}^k\cos{\pi j\over p}={\sin {k\pi\over 2p}\over \sin {\pi\over 2p}}$.
    If $k>{2\over 3}p$, then the last
    expression is greater than or equal to
    ${\sin {2p\pi\over 3(2p)}\over \sin{\pi\over 2p}}\geq {\sqrt{3}\over {2\pi\over 2p}}={p\sqrt{3}\over \pi}>{p\over 2}$. Thus, $N\leq {2\over 3}p<p$.
\end{proof}

\begin{corollary}
    Assume that $\varepsilon\leq \frac{\pi}{p}$. Then the number of maximal $\varepsilon$-paths is greater then $\frac{\vert V_n\vert}{p}
    -\frac{p}{\varepsilon}$,
    where $\vert V_n\vert$ denotes the number of vertices of $G_n$.
    Also, the number of maximal $\varepsilon$-paths is at least $\frac{\vert E_n\vert}{p}$, where $\vert E_n\vert$ denotes the number of edges of $G_n$.
\end{corollary}
\begin{proof}
 Indeed, for each $\varepsilon$-vertex $v_2$ there are two adjacent vertices $v_1$ and $v_3$ such that $v_1v_2v_3$ is a $\varepsilon$-path. This path can be extended in both directions as $\varepsilon$-path until it will stop
 at a non-$\varepsilon$ vertex. Note, that such a path cannot revisit already visited vertices. On the other hand, the sum of the numbers of all $\varepsilon$-vertices on all maximal $\varepsilon$-paths does not exceed the product of the number of all $\varepsilon$-paths 
 and the maximal number of vertices in such a path. According to Lemma 4 the number of $\varepsilon$-vertices in a $\varepsilon$-path is at most $p-1$.
 Note that Lemma 3 implies that the number of all $\varepsilon$-vertices is at least $\vert V_n\vert-\frac{p^2}{\varepsilon}$, which implies the first assertion.

 The second assertion follows from the fact that each edge is on a maximal $\varepsilon$-path.
 \end{proof}
 
\begin{corollary}
    The maximal degree of a $G_n$ is unbounded. Moreover,
    if $\varepsilon\leq \frac{\pi}{p}$, and the number of vertices $\vert V_n\vert\geq \frac{p^2}{\varepsilon}$, then there exists
    a 
    vertex of $G_n$ (possibly on the boundary 
    of $Pol_n$) such that its degree is greater than
    $(\frac{2\varepsilon}{p^3}(\vert V_n\vert-\frac{p^2}{\varepsilon}))^{\frac{1}{p}}$.
    Also, if $\varepsilon\leq \frac{\pi}{p}$, then there
    exists a 
    vertex $A$ of degree $K$ greater than
    $(\frac{2\epsilon\vert E_n\vert}{p^3})^{\frac{1}{p}}$.
\end{corollary}
\begin{proof}
Denote the maximal degree of a 
vertex of $G_n$ by $M_n$, the maximal length of
a maximal $\varepsilon$-path by $l_n$, and the number of all non-$\varepsilon$ vertices by $v_n$. 
The number of all maximal $\varepsilon$-paths does not exceed $\frac{1}{2}v_nM_n^{l_n}$.
Since $l_n<p$ (Lemma 4), and $v_n<\frac{p^2}{\varepsilon}$ (Lemma 3), the number
of all maximal $\varepsilon$-paths is at most $\frac{1}{2}\left(\frac{p^2}{\varepsilon}\right) M_n^p$. Juxtaposing
this with the two lower bounds of the number of all maximal $\varepsilon$-paths in Corollary 3, we
obtain both assertions of Corollary 4. 
\end{proof}

\begin{corollary} If $\varepsilon\leq {\pi\over p}$, then there exist
at least ${\varepsilon K\over p^2}$ $\varepsilon$-paths between some 
vertex $A$ of degree $K\geq$ 
 $(\frac{2\epsilon\vert E_n\vert}{p^3})^{\frac{1}{p}}$ (as in Corollary 4) and another 
 vertex $B$ (which is a non-$\varepsilon$ vertex). The initial edges of these paths that are incident to $A$ are all distinct.
\end{corollary}

\begin{proof}
Take a 
vertex $A$ satisfying the lower bound on its degree in terms of $|E_n|$ as in Corollary 4. Start along each of the $K$ edges incident to $A$. When a path reaches a 
$\varepsilon$-vertex $v$, we extend it by an edge $e$, so that the path
remains a $\varepsilon$-path. When a path hits a non-$\varepsilon$ vertex, we stop it. The lower bound in the Corollary is the ratio of $K$ to the upper bound for the number of non-$\varepsilon$ vertices in Lemma 3.
\end{proof}

In the proof of the above corollary, we can change the construction
as follows. Consider the paths from $A$ to $B$ and reexamine their construction. We can assume that the paths were constructed one by one. Now we assume that if a constructed path $\tau$ reaches a $\varepsilon$-vertex $v$ on a previously constructed path $\sigma$,
we just stop $\tau$ and say that $\sigma$ and $\tau$ merge at $v$.

Assume that at most $T=({\varepsilon K\over p^2})^{1\over p}$ paths merge at any $\varepsilon$-vertex $v$. This means that after
traveling the first segment, we are still left with ${{\varepsilon K\over p^2}\over {\varepsilon K\over p^2}^{1\over p}}={\varepsilon K\over p^2}^{p-1\over p}$ paths, after traveling two segments,
we are going to have at least ${\varepsilon K\over p^2}^{p-2\over p}$ paths, etc. As the length of maximal $\varepsilon$-paths is at most $p-1$, in the end we will still be left with at least $({\varepsilon K\over p^2})^{1\over p}$ vertex-disjoint paths meeting at $B$.
Recall that $K\geq ({2\varepsilon\vert E_n\vert\over p^3})^{1\over p}$.
We arrive to the following lemma:

\begin{lemma} Assume that $\varepsilon\leq{\pi\over p}$.
Let $R$ denote $\varepsilon^{{1\over p}+{1\over p^2}}\vert E_n\vert^{1\over p^2}{2^{1\over p^2}}p^{-{2p+3\over p^2}}$. Let $A$ be a 
vertex of degree $K\geq ({2\varepsilon|E_n|\over p^3})^{1\over p}$, as in Corollary 5. There exists a vertex $C$ connected with $A$ by more than $R$ vertex-disjoint $\varepsilon$-paths of length less than $p$ (Figure~\ref{fig:eps-paths}).
\end{lemma}

\begin{figure}[ht]
\centering
\resizebox{\ifdim\width>\textwidth \textwidth\else\width\fi}{!}{%
\begin{tikzpicture}[
   >={Stealth[length=2.2mm]},
   pth/.style={line width=0.9pt, blue!70!black},
   opth/.style={line width=1.2pt, blue!45!black},
   chord/.style={line width=0.9pt, dashed, black},
   edot/.style={circle, fill=black, inner sep=1.7pt},
   scale=1.0]
\draw[opth] (0.000,0.000) -- (0.550,0.225) -- (1.100,0.177) -- (1.650,0.417) -- (2.200,0.265) -- (2.750,0.435) -- (3.300,0.265) -- (3.850,0.435) -- (4.400,0.265) -- (4.950,0.435) -- (5.500,0.265) -- (6.050,0.435) -- (6.600,0.265) -- (7.150,0.417) -- (7.700,0.177) -- (8.250,0.225) -- (8.800,0.000);
\draw[pth] (0.000,0.000) -- (0.550,0.365) -- (1.100,0.440) -- (1.650,0.750) -- (2.200,0.615) -- (2.750,0.785) -- (3.300,0.615) -- (3.850,0.785) -- (4.400,0.615) -- (4.950,0.785) -- (5.500,0.615) -- (6.050,0.785) -- (6.600,0.615) -- (7.150,0.750) -- (7.700,0.440) -- (8.250,0.365) -- (8.800,0.000);
\draw[pth] (0.000,0.000) -- (0.550,0.505) -- (1.100,0.703) -- (1.650,1.083) -- (2.200,0.965) -- (2.750,1.135) -- (3.300,0.965) -- (3.850,1.135) -- (4.400,0.965) -- (4.950,1.135) -- (5.500,0.965) -- (6.050,1.135) -- (6.600,0.965) -- (7.150,1.083) -- (7.700,0.703) -- (8.250,0.505) -- (8.800,0.000);
\draw[pth] (0.000,0.000) -- (0.550,0.645) -- (1.100,0.965) -- (1.650,1.415) -- (2.200,1.315) -- (2.750,1.485) -- (3.300,1.315) -- (3.850,1.485) -- (4.400,1.315) -- (4.950,1.485) -- (5.500,1.315) -- (6.050,1.485) -- (6.600,1.315) -- (7.150,1.415) -- (7.700,0.965) -- (8.250,0.645) -- (8.800,0.000);
\draw[opth] (0.000,0.000) -- (0.550,0.785) -- (1.100,1.228) -- (1.650,1.747) -- (2.200,1.665) -- (2.750,1.835) -- (3.300,1.665) -- (3.850,1.835) -- (4.400,1.665) -- (4.950,1.835) -- (5.500,1.665) -- (6.050,1.835) -- (6.600,1.665) -- (7.150,1.747) -- (7.700,1.228) -- (8.250,0.785) -- (8.800,0.000);
\fill[blue!70!black] (0.550,0.225) circle (0.9pt);
\fill[blue!70!black] (1.100,0.177) circle (0.9pt);
\fill[blue!70!black] (1.650,0.417) circle (0.9pt);
\fill[blue!70!black] (2.200,0.265) circle (0.9pt);
\fill[blue!70!black] (2.750,0.435) circle (0.9pt);
\fill[blue!70!black] (3.300,0.265) circle (0.9pt);
\fill[blue!70!black] (3.850,0.435) circle (0.9pt);
\fill[blue!70!black] (4.400,0.265) circle (0.9pt);
\fill[blue!70!black] (4.950,0.435) circle (0.9pt);
\fill[blue!70!black] (5.500,0.265) circle (0.9pt);
\fill[blue!70!black] (6.050,0.435) circle (0.9pt);
\fill[blue!70!black] (6.600,0.265) circle (0.9pt);
\fill[blue!70!black] (7.150,0.417) circle (0.9pt);
\fill[blue!70!black] (7.700,0.177) circle (0.9pt);
\fill[blue!70!black] (8.250,0.225) circle (0.9pt);
\fill[blue!70!black] (0.550,0.365) circle (0.9pt);
\fill[blue!70!black] (1.100,0.440) circle (0.9pt);
\fill[blue!70!black] (1.650,0.750) circle (0.9pt);
\fill[blue!70!black] (2.200,0.615) circle (0.9pt);
\fill[blue!70!black] (2.750,0.785) circle (0.9pt);
\fill[blue!70!black] (3.300,0.615) circle (0.9pt);
\fill[blue!70!black] (3.850,0.785) circle (0.9pt);
\fill[blue!70!black] (4.400,0.615) circle (0.9pt);
\fill[blue!70!black] (4.950,0.785) circle (0.9pt);
\fill[blue!70!black] (5.500,0.615) circle (0.9pt);
\fill[blue!70!black] (6.050,0.785) circle (0.9pt);
\fill[blue!70!black] (6.600,0.615) circle (0.9pt);
\fill[blue!70!black] (7.150,0.750) circle (0.9pt);
\fill[blue!70!black] (7.700,0.440) circle (0.9pt);
\fill[blue!70!black] (8.250,0.365) circle (0.9pt);
\fill[blue!70!black] (0.550,0.505) circle (0.9pt);
\fill[blue!70!black] (1.100,0.703) circle (0.9pt);
\fill[blue!70!black] (1.650,1.083) circle (0.9pt);
\fill[blue!70!black] (2.200,0.965) circle (0.9pt);
\fill[blue!70!black] (2.750,1.135) circle (0.9pt);
\fill[blue!70!black] (3.300,0.965) circle (0.9pt);
\fill[blue!70!black] (3.850,1.135) circle (0.9pt);
\fill[blue!70!black] (4.400,0.965) circle (0.9pt);
\fill[blue!70!black] (4.950,1.135) circle (0.9pt);
\fill[blue!70!black] (5.500,0.965) circle (0.9pt);
\fill[blue!70!black] (6.050,1.135) circle (0.9pt);
\fill[blue!70!black] (6.600,0.965) circle (0.9pt);
\fill[blue!70!black] (7.150,1.083) circle (0.9pt);
\fill[blue!70!black] (7.700,0.703) circle (0.9pt);
\fill[blue!70!black] (8.250,0.505) circle (0.9pt);
\fill[blue!70!black] (0.550,0.645) circle (0.9pt);
\fill[blue!70!black] (1.100,0.965) circle (0.9pt);
\fill[blue!70!black] (1.650,1.415) circle (0.9pt);
\fill[blue!70!black] (2.200,1.315) circle (0.9pt);
\fill[blue!70!black] (2.750,1.485) circle (0.9pt);
\fill[blue!70!black] (3.300,1.315) circle (0.9pt);
\fill[blue!70!black] (3.850,1.485) circle (0.9pt);
\fill[blue!70!black] (4.400,1.315) circle (0.9pt);
\fill[blue!70!black] (4.950,1.485) circle (0.9pt);
\fill[blue!70!black] (5.500,1.315) circle (0.9pt);
\fill[blue!70!black] (6.050,1.485) circle (0.9pt);
\fill[blue!70!black] (6.600,1.315) circle (0.9pt);
\fill[blue!70!black] (7.150,1.415) circle (0.9pt);
\fill[blue!70!black] (7.700,0.965) circle (0.9pt);
\fill[blue!70!black] (8.250,0.645) circle (0.9pt);
\fill[blue!70!black] (0.550,0.785) circle (0.9pt);
\fill[blue!70!black] (1.100,1.228) circle (0.9pt);
\fill[blue!70!black] (1.650,1.747) circle (0.9pt);
\fill[blue!70!black] (2.200,1.665) circle (0.9pt);
\fill[blue!70!black] (2.750,1.835) circle (0.9pt);
\fill[blue!70!black] (3.300,1.665) circle (0.9pt);
\fill[blue!70!black] (3.850,1.835) circle (0.9pt);
\fill[blue!70!black] (4.400,1.665) circle (0.9pt);
\fill[blue!70!black] (4.950,1.835) circle (0.9pt);
\fill[blue!70!black] (5.500,1.665) circle (0.9pt);
\fill[blue!70!black] (6.050,1.835) circle (0.9pt);
\fill[blue!70!black] (6.600,1.665) circle (0.9pt);
\fill[blue!70!black] (7.150,1.747) circle (0.9pt);
\fill[blue!70!black] (7.700,1.228) circle (0.9pt);
\fill[blue!70!black] (8.250,0.785) circle (0.9pt);
\draw[chord] (0,0) -- (8.800,0);
\node[edot] at (0,0) {};
\node[edot] at (8.800,0) {};
\node[left]  at (0,0) {$A$};
\node[right] at (8.800,0) {$B$};
\node[blue!45!black,font=\footnotesize] at (4.400,2.050) {$\gamma_n$};
\node[blue!45!black,font=\footnotesize] at (6.336,0.175) {$\gamma_1$};
\node[below,font=\footnotesize] at (2.904,-0.10) {$|AB|$};
\end{tikzpicture}
}
\caption{Many vertex-disjoint $\varepsilon$-paths from $A$ to $B$, with the partition cells between two consecutive paths; the paths are unit-edge broken lines (not necessarily convex) and only the vertical scale is exaggerated.}
\label{fig:eps-paths}
\end{figure}
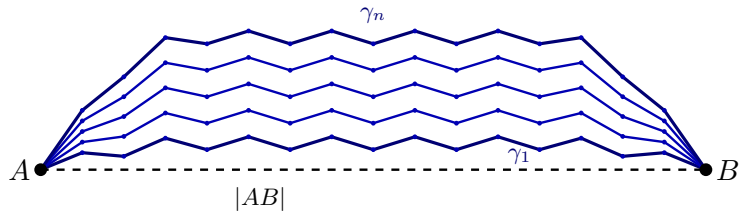

{\bf Remark 1.} Note that $R>{1\over 4}\varepsilon^{p+1\over p^2}\vert E_n\vert^{1\over p^2}$. If we choose $\varepsilon={1\over p}$, then $\varepsilon^{{1\over p}+{1\over p^2}}=p^{-{p+1\over p^2}}\geq 3^{-{4\over 9}}>0.61$ for all $p\geq 3$. (Here $p\geq 3$ always holds, as any polygon with integer side lengths has at least three sides, and, therefore, perimeter at least $3$.) Thus, $R>0.15|E_n|^{1\over p^2}$.

{\bf Remark 2.} Assume that instead of flat polygons $Pol_n$ subdivided into convex polygons with integer sides, we have polyhedral surfaces $S_n$ homeomorphic to the disk that were obtained as dual complexes
of geodesic nets in convex discs in $M$
as explained in sections 2.2 and 3.3.
An analogue of Lemma 1 for considered polyhedral surfaces was stated as Proposition 1.

Assume that the total absolute curvature $\mu_n$ of $S_n$ does not exceed ${1\over 4}$. Then the conclusion of Lemma 2 will held for $S_n$. 
The Fiala-Huber isoperimetric inequality generalized for polyhedral inequalities by Yu. Burago and V. Zalgaller ([BZ]) asserts
that the area $A_n$ of $S_n$ does not exceed ${p_n^2\over 4(\pi-{\mu_n\over 2})}$, which is still stronger than the inequality $A\leq {p^2\over 12}$ used in the proof of Lemma 2 as long as ${\mu_n\over 2}\leq 0.14$.
Therefore, Lemma 2 also holds for considered
surfaces. 

For Lemma 3 we assume that $\mu_n\leq {1\over 10}$. Moreover, we strengthen the assumption about $\varepsilon$.
Instead of assuming that $\varepsilon\leq {\pi\over 6}$,
we are going to assume that $\epsilon\leq 1/3$. As the
result ${\sin \varepsilon\over \varepsilon}\geq 0.98$, and the Fiala-Huber inequality implies that the area $A_n\geq {p^2\over 12.36}$. Now the same proof as the proof of Lemma 3 yields the same inequality. In application in the next sections we will be always assuming that $\varepsilon\leq {1\over p}\leq {1\over 3}$.

In Lemma 4 the angles between $e_j$ and $b_k$ are not determined anymore only by angles between consecutive edges $e_i$. The Gauss-Bonnet theorem applied to the polygon formed by edges $e_1,\ldots , e_k,b_k$ implies
that, in principle, these angles can become larger by a summand that does not exceed $\mu$. Now the lower
bound for $\Vert b_k\Vert$, $\Sigma_{j=1}^k \cos({j\pi\over p}+\mu)=
{\sin{k\pi\over 2p}\over \sin{\pi\over 2p}}\cos \mu$.
As $\mu\leq 0.25$, if $N\ge {2\over 3}p$, then $\Vert b_k\Vert\geq {\sin {\pi\over 3}\over {\pi\over 2p}}\cos 0.25>0.5p$, and again $N<{2\over 3}p<p$.

The proofs of all other results
in this section will be true as they are without any additional changes (besides the stronger assumption that $\varepsilon\leq {1\over 3}$).

Thus, all the results in this section will be true for polyhedral surfaces with total absolute curvature $\leq {1\over 10}$ radian that are homeomorphic to a disk, and $\varepsilon\leq {1\over 3}$ radian, provided that the polyhedral surface was constructed as the dual
complex of a geodesic net in a small convex ball in a closed Riemannian surface as described in section 3.3.



\section{ Proof of Theorem 1, and some useful lemmata for the proof of Theorem 2}
\subsection{Proof of Theorem 1 modulo lemmata 6 and 7}
The setting for these lemmata is two vertices $A$ and $B$ connected by a large number of vertex-disjoint consecutive
$\varepsilon$-paths $\gamma_1,\ldots, \gamma_{\lceil R\rceil}$ for a small $\varepsilon$ as in the previous lemma.
These paths have the same integer length $l$ greater than $|AB|$. The central idea is to use the parameter $c=l-|AB|$ that potentially can be uncontrollably close to $0$, and to prove that: (1) The distances
from all vertices of all paths to the line $(AB)$ do not exceed
$f(p)\sqrt{c}$ for $f(p)=\sqrt{2p}$ (Lemma 7 below);
(2) For each path $\gamma_i$
from $A$ to $B$ the sum $\Sigma_j|\pi-\alpha_j|$ of absolute values of its outer angles is at least $g(p)\sqrt{c}$, where $g(p)={1\over\sqrt{p}}$ (Lemma 6 and Corollary 6 below).

Note that if we consider paths $\gamma_{i-1},\gamma_{i+1}$ and $\gamma_i$ between them,
for each inner vertex $q$ of $\gamma_i$ the
triangle formed by halves of the edges
adjacent to $q$ is contained in the domain between $\gamma_{i-1}$ and $\gamma_{i+1}$
and has area ${1\over 8}\sin (|\pi-\alpha_j|)\geq {1\over 16}|\pi-\alpha_j|$,
where $\alpha_j$ denotes the angle at $q$. These triangles are disjoint. Applying the mentioned Corollary 6, we see that the sum of their
areas is at least ${1\over 16}{\sqrt{c}\over\sqrt{p}}$.
Therefore, we see that the area of the domain between $\gamma_{i-1}$ and $\gamma_{i+1}$
is at least ${1\over 16}{\sqrt{c}\over\sqrt{p}}$. As the result,
the area of the polygon between $\gamma_1$ and
$\gamma_{\left\lceil R\right\rceil}$ is at least
${1\over 16}\left\lfloor{\left\lceil R\right\rceil\over 2}\right\rfloor{\sqrt{c}\over \sqrt{p}}$.
On the other hand, 
this polygon is contained in the rectangle with one side parallel to $(AB)$ of length $|AB|<l<p$ and the other side $\leq 2f(p)\sqrt{c}=2\sqrt{2}\sqrt{p}\sqrt{c}$ (Figure~\ref{fig:area-argument}).
Combining these two estimates, we see that $R\leq 16+64{f(p)\over g(p)}p=64\sqrt{2}p^2+16<91p^2+16$. Combining this inequality with the inequality $R>0.15|E_n|^{1\over p^2}$ established in Remark 1 at the end of the previous section (for the choice $\epsilon={1\over p}$), we see that
$$\vert V_n\vert\le\vert E_n\vert<\left({91p^2+16\over 0.15}\right)^{p^2}<(607p^2+107)^{p^2}\leq (625p^2)^{p^2}=(25p)^{2p^2},$$
where the second to last inequality holds since $p\geq 3$, and, therefore, $107\leq 18p^2$. This completes the proof of Theorem 3 and, therefore, of Theorem 1 and Corollary 1.

\begin{figure}[ht]
\centering
\resizebox{\ifdim\width>\textwidth \textwidth\else\width\fi}{!}{%
\begin{tikzpicture}[
   >={Stealth[length=2.2mm]},
   pth/.style={line width=1.0pt, blue!70!black},
   nb/.style={line width=0.9pt, blue!35!black, opacity=0.7},
   chord/.style={line width=0.9pt, dashed, black},
   tri/.style={fill=red!25, draw=red!70!black, line width=0.6pt},
   rect/.style={line width=0.8pt, dashed, red!70!black},
   edot/.style={circle, fill=black, inner sep=1.7pt},
   dot/.style={circle, fill=blue!70!black, inner sep=1.2pt},
   scale=1.0]
\begin{scope}
\draw[nb] (0.000,0.000) -- (0.620,0.440) -- (1.240,0.078) -- (1.860,0.688) -- (2.480,0.190) -- (3.100,0.710) -- (3.720,0.190) -- (4.340,0.688) -- (4.960,0.078) -- (5.580,0.440) -- (6.200,0.000);
\draw[nb] (0.000,0.000) -- (0.620,0.840) -- (1.240,0.827) -- (1.860,1.637) -- (2.480,1.190) -- (3.100,1.710) -- (3.720,1.190) -- (4.340,1.637) -- (4.960,0.827) -- (5.580,0.840) -- (6.200,0.000);
\draw[pth] (0.000,0.000) -- (0.620,0.640) -- (1.240,0.452) -- (1.860,1.163) -- (2.480,0.690) -- (3.100,1.210) -- (3.720,0.690) -- (4.340,1.163) -- (4.960,0.452) -- (5.580,0.640) -- (6.200,0.000);
\fill[tri] (1.860,1.163) -- (1.550,0.808) -- (2.170,0.926) -- cycle;
\node[dot] at (1.860,1.163) {};
\fill[tri] (2.480,0.690) -- (2.170,0.926) -- (2.790,0.950) -- cycle;
\node[dot] at (2.480,0.690) {};
\fill[tri] (3.100,1.210) -- (2.790,0.950) -- (3.410,0.950) -- cycle;
\node[dot] at (3.100,1.210) {};
\fill[tri] (3.720,0.690) -- (3.410,0.950) -- (4.030,0.926) -- cycle;
\node[dot] at (3.720,0.690) {};
\fill[tri] (4.340,1.163) -- (4.030,0.926) -- (4.650,0.808) -- cycle;
\node[dot] at (4.340,1.163) {};
\node[blue!70!black,font=\footnotesize,left] at (0,0.380) {$\gamma_i$};
\node[blue!35!black,font=\footnotesize,left] at (0,0.880) {$\gamma_{i+1}$};
\node[blue!35!black,font=\footnotesize,left] at (0,-0.120) {$\gamma_{i-1}$};
\node[font=\footnotesize] at (3.100,-0.75) {(a)};
\end{scope}
\begin{scope}[shift={(8.400,0)}]
\draw[rect] (-0.18,-0.30) rectangle (7.740,1.380);
\draw[pth,opacity=0.85] (0.000,0.000) -- (0.420,0.127) -- (0.840,0.080) -- (1.260,0.226) -- (1.680,0.125) -- (2.100,0.235) -- (2.520,0.125) -- (2.940,0.235) -- (3.360,0.125) -- (3.780,0.235) -- (4.200,0.125) -- (4.620,0.235) -- (5.040,0.125) -- (5.460,0.235) -- (5.880,0.125) -- (6.300,0.226) -- (6.720,0.080) -- (7.140,0.127) -- (7.560,0.000);
\draw[pth,opacity=0.85] (0.000,0.000) -- (0.420,0.199) -- (0.840,0.215) -- (1.260,0.397) -- (1.680,0.305) -- (2.100,0.415) -- (2.520,0.305) -- (2.940,0.415) -- (3.360,0.305) -- (3.780,0.415) -- (4.200,0.305) -- (4.620,0.415) -- (5.040,0.305) -- (5.460,0.415) -- (5.880,0.305) -- (6.300,0.397) -- (6.720,0.215) -- (7.140,0.199) -- (7.560,0.000);
\draw[pth,opacity=0.85] (0.000,0.000) -- (0.420,0.271) -- (0.840,0.350) -- (1.260,0.568) -- (1.680,0.485) -- (2.100,0.595) -- (2.520,0.485) -- (2.940,0.595) -- (3.360,0.485) -- (3.780,0.595) -- (4.200,0.485) -- (4.620,0.595) -- (5.040,0.485) -- (5.460,0.595) -- (5.880,0.485) -- (6.300,0.568) -- (6.720,0.350) -- (7.140,0.271) -- (7.560,0.000);
\draw[pth,opacity=0.85] (0.000,0.000) -- (0.420,0.343) -- (0.840,0.485) -- (1.260,0.739) -- (1.680,0.665) -- (2.100,0.775) -- (2.520,0.665) -- (2.940,0.775) -- (3.360,0.665) -- (3.780,0.775) -- (4.200,0.665) -- (4.620,0.775) -- (5.040,0.665) -- (5.460,0.775) -- (5.880,0.665) -- (6.300,0.739) -- (6.720,0.485) -- (7.140,0.343) -- (7.560,0.000);
\draw[pth,opacity=0.85] (0.000,0.000) -- (0.420,0.415) -- (0.840,0.620) -- (1.260,0.910) -- (1.680,0.845) -- (2.100,0.955) -- (2.520,0.845) -- (2.940,0.955) -- (3.360,0.845) -- (3.780,0.955) -- (4.200,0.845) -- (4.620,0.955) -- (5.040,0.845) -- (5.460,0.955) -- (5.880,0.845) -- (6.300,0.910) -- (6.720,0.620) -- (7.140,0.415) -- (7.560,0.000);
\draw[pth,opacity=0.85] (0.000,0.000) -- (0.420,0.487) -- (0.840,0.755) -- (1.260,1.081) -- (1.680,1.025) -- (2.100,1.135) -- (2.520,1.025) -- (2.940,1.135) -- (3.360,1.025) -- (3.780,1.135) -- (4.200,1.025) -- (4.620,1.135) -- (5.040,1.025) -- (5.460,1.135) -- (5.880,1.025) -- (6.300,1.081) -- (6.720,0.755) -- (7.140,0.487) -- (7.560,0.000);
\draw[chord] (0,0) -- (7.560,0);
\node[edot] at (0,0) {};
\node[edot] at (7.560,0) {};
\node[left] at (0,0) {$A$};
\node[right] at (7.560,0) {$B$};
\draw[<->, red!70!black, line width=0.7pt] (8.060,-0.30) -- (8.060,1.380);
\node[red!70!black,font=\footnotesize,right] at (8.110,0.540) {$2\sqrt{2pc}$};
\draw[<->, red!70!black, line width=0.7pt] (-0.18,-0.72) -- (7.740,-0.72);
\node[red!70!black,font=\footnotesize,below] at (3.780,-0.70) {$|AB|<l\le p$};
\node[font=\footnotesize] at (3.780,-1.25) {(b)};
\end{scope}
\end{tikzpicture}
}
\caption{The two area estimates for the region between consecutive $\varepsilon$-paths: (a) the disjoint triangles at the inner vertices of $\gamma_i$, meeting at edge midpoints (turning exaggerated); (b) the whole family confined to a thin rectangle.}
\label{fig:area-argument}
\end{figure}
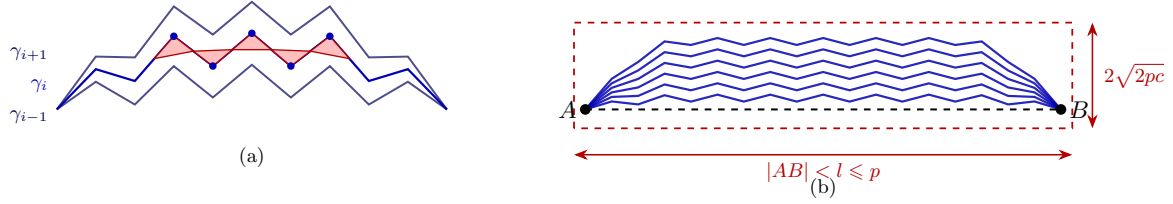

\subsection{Lemma 6 and its generalization
for polyhedral surfaces}

\begin{lemma} Consider a broken line in
the plane connecting points $A$ and $B$. Assume that this broken line is a graph of a function on $[AB]$. Denote the number of segments by $N$, their lengths by $s_1,\ldots, s_N$, the length of the curve
(equal to $\Sigma_{i=1}^Ns_i$) by $l$, and $l-|AB|$ by $c$. Further, denote the angles
between the $i$th and the $(i+1)$st segments by $\alpha_i$, $i=1,\ldots, N-1$.
Then there exists $i$ such that $\alpha_i$
satisfies the inequality $\vert\alpha_i-\pi\vert>{1\over N-1}{\sqrt {2c\over l}}$. Denote the outer angles $\pi-\alpha_i$ at $N-1$ inner vertices of 
the broken line by $\theta_i$, $i=1, \ldots , N-1$. (Note that $\theta_i$ can be negative, if $\alpha_i>\pi$).
Then $\Sigma_{i=1}^{N-1}\vert \theta_i\vert\geq \sqrt{2c\over l}$ (see Figure~\ref{fig:lemma6}).

\end{lemma}

\begin{figure}[ht]
\centering
\resizebox{\ifdim\width>\textwidth \textwidth\else\width\fi}{!}{%
\begin{tikzpicture}[
   >={Stealth[length=2.2mm]},
   bl/.style={line width=1.0pt, blue!70!black},
   chord/.style={line width=0.9pt, dashed, black},
   ext/.style={line width=0.6pt, dotted, gray!70!black},
   dot/.style={circle, fill=blue!70!black, inner sep=1.3pt},
   edot/.style={circle, fill=black, inner sep=1.6pt},
   ang/.style={line width=0.6pt, red!70!black},
   scale=1.15]

\coordinate (A)  at (0,0);
\coordinate (P1) at (1.25,0.52);
\coordinate (P2) at (2.55,0.86);
\coordinate (P3) at (3.95,0.60);
\coordinate (P4) at (5.35,0.92);
\coordinate (P5) at (6.55,0.44);
\coordinate (B)  at (8,0);

\draw[chord] (A)--(B);
\node[below,font=\footnotesize] at (4,-0.06) {$|AB|=l-c$};

\draw[bl] (A)--(P1)--(P2)--(P3)--(P4)--(P5)--(B);

\node[above left,font=\footnotesize]  at ($(A)!0.5!(P1)$)  {$s_1$};
\node[above,font=\footnotesize]        at ($(P1)!0.5!(P2)$) {$s_2$};
\node[below,font=\footnotesize]        at ($(P4)!0.5!(P5)$) {$s_{N-1}$};
\node[above right,font=\footnotesize]  at ($(P5)!0.5!(B)$)  {$s_N$};

\draw[ext] (P2) -- ($(P2)+(14.65:1.25)$);
\draw[ang] (P2) ++(14.65:0.92) arc (14.65:-10.55:0.92);
\node[red!70!black,font=\footnotesize] at ($(P2)+(2:1.45)$) {$\theta_2$};

\draw[ext] (P3) -- ($(P3)+(-10.55:1.2)$);
\draw[ang] (P3) ++(-10.55:0.86) arc (-10.55:12.75:0.86);
\node[red!70!black,font=\footnotesize] at ($(P3)+(1:1.4)$) {$\theta_3$};

\draw[ang] (A) ++(0:0.95) arc (0:22.6:0.95);
\node[red!70!black,font=\footnotesize] at (1.18,0.19) {$\beta$};
\draw[ang] (B) ++(180:0.95) arc (180:159.7:0.95);
\node[red!70!black,font=\footnotesize] at (6.86,0.20) {$\gamma$};

\node[edot] at (A) {};
\node[edot] at (B) {};
\foreach \p in {(P1),(P2),(P3),(P4),(P5)} \node[dot] at \p {};
\node[below left]  at (A) {$A$};
\node[below right] at (B) {$B$};

\end{tikzpicture}
}
\caption{The setting of Lemma 6: a broken line over the chord $[AB]$ with segment lengths $s_i$ and signed turning angles $\theta_i=\pi-\alpha_i$.}
\label{fig:lemma6}
\end{figure}
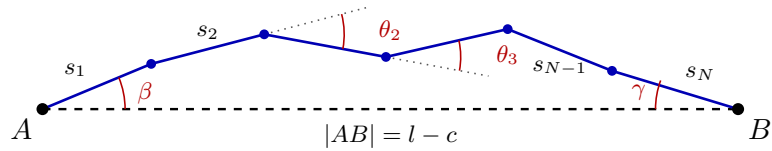

\begin{proof}
We need to prove that 
$\Sigma_i^{N-1}\vert\theta_i\vert\geq\sqrt{2c\over l}$.
As an immediate corollary, for some $i$, $\vert \pi-\alpha_i\vert>{1\over N-1}\sqrt {2c\over l}$.
Note that $\beta+\gamma=\Sigma_{i=1}^{N-1}\theta_i$.
Projecting the broken line on the line $(AB)$
we see that $s_1\cos\beta+s_2\cos(\beta-\theta_1)+s_3\cos(\beta-\theta_1-\theta_2)+\ldots +s_N\cos(\beta-\theta_1-\ldots -\theta_{N-1})=|AB|=l-c$. Observe that all angles in this formula do not
exceed 
$\Sigma_i\vert\theta_i\vert$. 
Therefore, $c=2s_1\sin^2{\beta\over 2}+2s_2\sin^2{\beta-\theta_1\over 2}+\ldots +2s_N\sin^2{\beta-\theta_1-\ldots -\theta_{N-1}\over 2}\leq {1\over 2}(s_1\beta^2+s_2(\beta-\theta_1)^2+\ldots s_N(\beta-\theta_1-\ldots -\theta_{N-1})^2)\leq {l\over 2}(\Sigma_i\vert\theta_i\vert)^2$. 

\end{proof}

\begin{corollary}
Assume that in the previous lemma the lengths
of all segments of the broken line are equal to $1$ and, therefore, $N=l$. Then
$\Sigma_{i=1}^{l-1}|\pi-\alpha_i|\geq\sqrt{2c\over l}$,
and, therefore, $\max_i|\pi-\alpha_i|\geq \sqrt{2c\over l^3}$.
\end{corollary}

It is easy to see that the previous lemma
holds for polyhedral surfaces where the curvature at all inner vertices is non-positive.

\begin{corollary}

Let $\Sigma$ be a polyhedral surface homeomorphic to the disc such that the curvature at all
inner vertices of $\Sigma$ is non positive.
Consider a simplicial path $\gamma$ in
$\Sigma$ connecting boundary vertices $A$ and $B$.
Denote the number of segments by $N$, their lengths by $s_1,\ldots, s_N$, the length of the curve
(equal to $\Sigma_{i=1}^Ns_i$) by $l$, and $l-dist_\Sigma(A,B)$ by $c$. Further, denote the angles
between the $i$th and the $(i+1)$st segments by $\alpha_i$, $i=1,\ldots, N-1$.
Assume that for all $i$ $|\pi-\alpha_i|\leq {1\over N}$,
the total absolute curvature of $\Sigma$ does not exceed $0.1$.
Then there exists $i$ such that $\alpha_i$
satisfies the inequality $\vert\alpha_i-\pi\vert>{1\over N-1}{\sqrt {2c\over l}}$. Denote the outer angles $\pi-\alpha_i$ at $N-1$ inner vertices of 
the broken line by $\theta_i$, $i=1, \ldots , N-1$. (Note that $\theta_i$ can be negative, if $\alpha_i>\pi$).
Then $\Sigma_{i=1}^{N-1}\vert \theta_i\vert\geq \sqrt{2c\over l}$.

\end{corollary}

\par\noindent{\bf Remark.} 
If $\Sigma$ is the Euclidean plane, the constraint on angles $\alpha_i$
implies that each perpendicular to the straight line segment $[AB]$ intersects
$\gamma$ at one point. 

\begin{proof}
First, we will glue flat polygons to the boundary of $\Sigma$ to obtain a larger
domain $\tilde\Sigma$ such that all vertices
on its boundary have angles $\leq \pi$.
Thus, $\tilde\Sigma$ will be a CAT(0) space.
The length $d$ of the minimal geodesic $\tau$ between $A$ and $B$ in $\tilde\Sigma$ will be less than or equal to $dist_\Sigma(A,B)$. Therefore $\tilde c$ calculated for $\tilde \Sigma$ in the same way $c$ was calculated for $\Sigma$ will be greater than or equal to $c$. Reshetnyak's
majorization theorem implies that there
exists a convex curve $\tilde\gamma$ in the Euclidean plane
with the same number $N$ of straight line segments as in $\gamma$, the same lengths of segments as in $\gamma$, the angles
equal to $\pi-|\pi-\alpha_i|$ for all $i$,
and the distance between the endpoints
of $\gamma$ in the plane is less than or equal to
the distance between $A$ in $B$ in $\tilde\Sigma$,
and, therefore, the distance between $A$ and $B$ in $\Sigma$. The assumption about the angles in the corollary implies that the angles
$\alpha_i\in ({\pi\over 2}, {3\pi\over 2})$.
Therefore, $\tilde\gamma$ does not ``backtrack" and is the graph of a function on the segments connecting its endpoints $\tilde A$, $\tilde B$.
Now we are going to apply Lemma 6 to
the plane curve $\tilde\gamma$ that has the same
absolute values of the deviations of its angles from $\pi$. Denoting the endpoints of $\tilde\gamma$ by $\tilde A$, $\tilde B$, we see that $\Sigma_i|\pi-\alpha_i|\geq \sqrt{2(length(\tilde\gamma)-dist_{R^2}(\tilde A,\tilde B))\over length(\tilde\gamma)}\geq \sqrt{2\tilde c\over l}$.
As $\tilde c\geq c$, the corollary follows.

\end{proof}

In the next corollary we already do not assume that $\Sigma$ is non-positively curved,
but we will need to subtract an upper bound
for the total positive curvature from our estimates.

\begin{corollary}

Let $\Sigma$ be a polyhedral surface homeomorphic to the disc. Let $\mu$ be the total positive curvature, that is, the sum of angular defects $2\pi - \beta_v$ over all
inner vertices $v$ of $\Sigma$, where the angular defect (curvature) is positive.
Consider a simplicial path $\gamma$ in
$\Sigma$ connecting boundary vertices $A$ and $B$. 
Denote the number of segments by $N$, their lengths by $s_1,\ldots, s_N$, the length of the curve
(equal to $\Sigma_{i=1}^Ns_i$) by $l$, and $l-dist_\Sigma(A,B)$ by $c$. Further, denote the angles
between the $i$th and the $(i+1)$st segments by $\alpha_i$, $i=1,\ldots, N-1$.
Assume that for all $i$ $|\pi-\alpha_i|\leq {1\over N}$,
and the total absolute curvature of $\Sigma$ does not exceed $0.1$. 
Then there exists $i$ such that $\alpha_i$
satisfies the inequality $\vert\alpha_i-\pi\vert>{1\over N-1}({\sqrt {2c\over l}}-\mu)$. Denote the outer angles $\pi-\alpha_i$ at $N-1$ inner vertices of 
the broken line by $\theta_i$, $i=1, \ldots , N-1$. (Note that $\theta_i$ can be negative, if $\alpha_i>\pi$.) 
Then $\Sigma_{i=1}^{N-1}\vert \theta_i\vert\geq \sqrt{2c\over l}-\mu$.

\end{corollary}

\begin{proof}
We will be proving the lower bound for the sum of the angles $\theta_i$. The lower bound for $\max_i|\pi-\alpha_i|$ will be its immediate corollary.

Consider a minimal geodesic $g$ between $A$ and $B$. The curve $\gamma$ can intersect the geodesic $g$ at several points. The domain
of $\Sigma$ between $g$ and $\Sigma$ splits into several discs each bounded by the union of an arc of $g$ and and arc of $\gamma$. The interiors of these discs do not intersect, but the discs are pairwise joined at the points of intersection of $g$ and $\gamma$.
It is clear that the desired lower bound for the sum of absolute values of outer angles of $\gamma$ would follow the analogous statement for each of the individual discs. Let $D$ be one of these discs and $\mu_D$ the total positive curvature in all inner vertices of $D$.
The boundary of $D$ consists of an arc $\gamma_D$ of $\gamma$ and an arc $g_D$ of $g$. Parametrize both $g_D$ and $\gamma_D$
proportionally to the arclength and connect
the corresponding points by minimizing geodesics. If there are inner vertices of positive curvature, none of these geodesics can pass through such vertices. As an almost immediate corollary for each vertex $v$ of positive curvature in ${\rm Int}(D)$, there exist a pair of minimizing geodesics in $D$ between a point $x_v\in \gamma_D$ and a point $y_v\in g_D$. We are going to remove the interior of each such geodesic digon, and identify the points on both geodesics forming the boundary of the digon that are at the same distance from $x$. 

As the result, we will get rid of all inner vertices with positive curvature in $D$. The resulting domain, $D'$ will be homeomorphic to the disc. The arc $\gamma_{D'}$
will have the same length, but might acquire
new angles at vertices $x_v$. The sum of the outer angles at these vectors can be majorized by $\mu_D$, and this estimate, $\mu_D$, needs to be subtracted from the lower bound $\sqrt{2c(D')\over l(D')}$ for
$\Sigma_i|\theta_i|$. As we observed, $l(D')=l(D)$ (as the length if $\gamma_D$ did not change). However, $g_{D'}$ acquired angles $<\pi$ at points $y_v$ , and the length of the shortest geodesic between the common endpoints of $g_{D'}$ and $\gamma_{D'}$ can be less than the length of $g_D$. As the result, $c(D')\ge c(D)$, and $\sqrt{2c(D)\over l(D)}-\mu_D$ will be the desired lower bound. Summing these estimates
over all disks $D$, we obtain the lower bound in the corollary.

\end{proof}

\subsection{Lemma 7 and its generalizations for polyhedral surfaces}

\begin{lemma}
Consider a path $\gamma$ of length $l<p$ in the plane between $A=\gamma(0)$ and $B=\gamma(1)$ (Figure~\ref{fig:lemma7}).
Denote $l-|AB|$ by $c$.
Then the distance between each point of $\gamma$ and the straight line $(AB)$ does not exceed $\sqrt{2pc}$.
\end{lemma}

\begin{figure}[ht]
\centering
\resizebox{\ifdim\width>\textwidth \textwidth\else\width\fi}{!}{%
\begin{tikzpicture}[
   >={Stealth[length=2.2mm]},
   pth/.style={line width=1.0pt, blue!70!black},
   chord/.style={line width=1.0pt, black},
   aux/.style={line width=0.7pt, dashed, red!70!black},
   ell/.style={line width=0.7pt, dotted, gray!75!black},
   dot/.style={circle, fill=blue!70!black, inner sep=1.1pt},
   edot/.style={circle, fill=black, inner sep=1.7pt},
   scale=1.0]

\coordinate (A) at (0,0);
\coordinate (B) at (7,0);
\coordinate (X) at (3.5,1.02);
\coordinate (Y) at (3.5,0);

\draw[ell] (3.5,0) ellipse (3.75 and 1.35);

\draw[chord] (A)--(B);
\draw[pth] (0,0)--(1.0,0.46)--(1.9,0.34)--(2.7,0.72)--(3.5,1.02)
           --(4.4,0.66)--(5.3,0.80)--(6.1,0.34)--(7,0);
\node[blue!70!black,font=\footnotesize] at (1.55,0.66) {$\gamma$};

\draw[aux] (X)--(Y);
\draw[aux] (A)--(X);
\draw[aux] (X)--(B);
\node[red!70!black,font=\footnotesize,right=1pt] at ($(X)!0.5!(Y)$) {$h$};
\node[red!70!black,font=\footnotesize,fill=white,inner sep=1pt] at (1.55,0.42) {$|AX|$};
\node[red!70!black,font=\footnotesize,fill=white,inner sep=1pt] at (5.45,0.42) {$|XB|$};
\draw[line width=0.6pt,red!70!black] (3.5,0.16) -- (3.66,0.16) -- (3.66,0);
\draw[<->,line width=0.6pt] (0,-0.72)--(3.5,-0.72);
\node[font=\footnotesize,below] at (1.75,-0.70) {$a=|AY|$};
\draw[<->,line width=0.6pt] (3.5,-0.72)--(7,-0.72);
\node[font=\footnotesize,below] at (5.25,-0.70) {$b=|YB|$};

\node[dot] at (X) {};
\node[edot] at (A) {};
\node[edot] at (B) {};
\node[above] at (X) {$X$};
\node[below left] at (A) {$A$};
\node[below right] at (B) {$B$};
\node[font=\footnotesize,anchor=north east] at (3.44,-0.04) {$Y$};

\end{tikzpicture}
}
\caption{The setting of Lemma 7: a point $X$ of $\gamma$, its orthogonal projection $Y$ onto $(AB)$, and the ellipse with foci $A,B$ that confines $\gamma$.}
\label{fig:lemma7}
\end{figure}
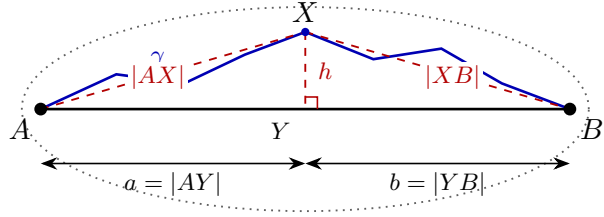

\begin{proof} Let $X=\gamma(t)$ for some $t\in [0,1]$. Denote the projection of $X$ to $(AB)$ by $Y$, and the distance from $X$ to $(AB)$ by $h$.

First, assume that the orthogonal
projection of $X$ 
to the line $(AB)$ is between $A$ and $B$.
Denote $|AY|$ by $a$, $|YB|$ by $b$, and $|XY|$ by $h$. It is obvious that $h\leq{l\over 2}\leq {p\over 2}<p$. Note that $|AX|+|XB|\leq l=|AB|+c=a+b+c$.
Using the Pythagorean theorem, $\sqrt{a^2+h^2}-a+\sqrt{b^2+h^2}-b\leq c$. This inequality can be rewritten as 
$$h^2({1\over \sqrt{h^2+a^2}+a}+{1\over \sqrt{b^2+h^2}+b})\leq c.\ \ \ \ (1)$$
We can replace the left-hand side of inequality (1) with a smaller value ${2h^2\over \sqrt{|AB|^2+p^2}+|AB|}> {2h^2\over 3p}>{h^2\over 2p}$. Now we see that $h< \sqrt{2pc}$.

Now assume that the projection $Y$ from $X$ to $(AB)$ is outside the segment $[AB]$. The Pythagorean theorem implies that the sum $|AX|+|XB|$ is at least $h+\sqrt{
h^2+|AB|^2}$ and $c\geq |AX|+|XB|-|AB|\geq h+\sqrt{h^2+|AB|^2}-|AB|=h+{h^2\over \sqrt{h^2+|AB|^2}+h}\geq {2h^2\over \sqrt{|AB|^2+p^2}+p}>{2h^2\over 3p}>{h^2\over 2p}$. Again, $h<\sqrt{2pc}$.

\end{proof}

\par\noindent

Note that Lemma 7 holds for Alexandrov spaces
of non-positive curvature (instead of the plane).

\begin{lemma}
Consider a path $\gamma$ of length $\leq p$ in an Alexandrov space of curvature $\leq 0$ between $A=\gamma(0)$ and $B=\gamma(1)$.
Denote $l-dist(A,B)$ by $c$.
Then the distance between each point of $\gamma$ and the geodesic $[AB]$ does not exceed $\sqrt{2pc}$.
\end{lemma}

\begin{proof}

Let $X$ be $\gamma(t)$ for some $t$. Denote the closest to $X$ point of the geodesic $[AB]$ by $Y$, and the distance from $X$ to $Y$ by $h$. Note that
$dist(A,X)+dist(X,B)\leq l$. Therefore, $(dist(A,X)-dist(A,Y))+(dist(X,B)-dist(Y,B))\leq l-|AB|=c$. 
Consider two cases:
\par\noindent
{\bf Case 1.} $Y$ is in $(AB)$ (that is, not one of the endpoints). Note that the angles $XYA$ and $XYB$
are greater than or equal to ${\pi\over 2}$ (or, one can find a point on $AB$ closer to $X$ than $Y$).
Consider the plane triangles $X'Y'A'$ and $X'Y'B'$, with $|Y'A'|=|YA|$, $|X'Y'|=|XY|$,
$|Y'B'|=|YB|$,
$\angle A'Y'X'=\angle AYX$, $\angle B'Y'X'=\angle BYX$.

The Alexandrov-Toponogov comparison theorem implies
that $|AX|\geq |A'X'|$, $|BX|\geq |B'X'|$. On the other hand, as $\angle A'Y'X'\geq {\pi\over 2}$,
$|A'X'|\geq\sqrt{h^2+|AY|^2}$, and, similarly, 
$|B'X'|\geq \sqrt{h^2+|BY|^2}$.

Now $$c=l-dist(A,B)\geq |AX|+|XB|-|AY|-|YB|\geq {h^2\over |YB|+\sqrt{|YB|^2+h^2}}+{h^2\over |YA|+\sqrt{|AY|^2+h^2}}.\ \ \ \ (2)$$

As $h\leq {l\over 2}\leq {p\over 2}$, we can replace
the right hand side in $(2)$ by a smaller value
${2h^2\over |AB|+\sqrt{p^2+|AB|^2}}>{2h^2\over 3p}\geq {h^2\over 2p}$. So, $2pc\geq h^2$. 
\par\noindent
{Case 2.} Assume that $Y$ coincides with 
$A$ or $B$. Assume that $A$ is the closest to $X$
point of $[AB]$. Therefore, $\angle XAB\geq{\pi\over 2}$. Comparing the triangle $XAB$ with the plane triangle $X'A'B'$ with $|X'A'|=|XA|$, $|A'B'|=|AB|$
and $\angle XAB=\angle X'A'B'$, we see that $|XB|\geq |X'B'|\geq \sqrt{|AB|^2+|XA|^2}$. Denote $dist(X,[AB])=|XA|$ by $h$. Now $c\geq |XA|+|XB|-|AB|\geq h+\sqrt{h^2+|AB|^2}-|AB|=h+
{h^2\over \sqrt{h^2+|AB|^2}+h}\geq
{2h^2\over \sqrt{h^2+|AB|^2}+h}\geq {2h^2\over 3p}\geq {h^2\over 2p}$, and the lemma follows.

\end{proof}

In the next lemma we prove a similar inequality for all polyhedral surfaces glued from flat polygons, so that the curvature lives at vertices. For each vertex $v$ let $A(v)$ denote the sum of the all incident angles of polygons that have $v$ as one of the vertices.
We call $\vert A(v)-2\pi\vert$ the {\it absolute angular defect} of the surface at $v$, and the sum of the absolute angular defects over all $v$ is called {\it the total absolute curvature}, or the {\it total absolute defect}.

\begin{lemma}
Let $\Sigma$ be a polyhedral surface (that is, a surface glued from flat polygons). 
Consider a path $\gamma$ of length $l\leq p$ in $\Sigma$ between $A=\gamma(0)$ and $B=\gamma(1)$. Assume that the total absolute
curvature of the geodesic triangle $ABC$ on $\Sigma$, where $C=\gamma(t)$ is
an arbitrary point of $\gamma$, does not exceed $\phi\le {\pi\over 2}$.
Denote $l-dist(A,B)$ by $c$.
Then the distance between each point of $\gamma$ and the geodesic $[AB]$ does not exceed $x=\max\{\sqrt{{3\over 2}pc+{p^2\phi^2\over 2}}, p\phi\}$.

\end{lemma}

\begin{proof}
We start as in the proof of the previous lemma. We consider
$X=\gamma(t)$, its closest point $Y$ on $[AB]$, and
two cases, when $Y\in (AB)$, and when $Y=A$ or $B$.
In the first case we consider triangles $AXY$ and
$BXY$ but this time we cannot use the Alexandrov-Toponogov comparison theorem  for $K\leq 0$ as these triangles can contain vertices $v$ with positive curvature, ($A(v)<2\pi$). Consider, for example,
the triangle $AXY$. Our goal is to find a good lower bound
for the length of $[AX]$ in terms of $|AY|$ and $\vert XY\vert$. (Recall that $\angle AYX\geq{\pi\over 2}$.)

Consider geodesics between $X$ and a variable point
$D\in [AB]$. At some points $D$ these geodesics change continuously
with $D$. However, at each point of discontinuity $X$ can be connected with $D$ by more than one geodesics.
For each such point consider the disc $\Omega_D$ formed by the two outer geodesics from $X$ to $D$.
For each $D$ with more than one minimizing geodesic from $X$ to $D$ we glue one of two outer geodesics from
$X$ to $D$ to the other by identifying pairs of points at the same distance from $X$. Thus, we will completely
eliminate the interior of the disc $\Omega_D$. It is easy to see that each such disc should contain one or more points of positive curvature, and each point of
positive curvature is contained in one of these discs.
Thus, the result will be a space $S_{AXB}$ of non-positive curvature. Removal of the disc $\Omega_D$ decreases the angle at $D$ between $AD$ and $DB$ by
the angle $\alpha_D$ of the digon $\Omega_D$ at $D$. This angle
does not exceed the sum of angular defects at all
vertices inside $\Omega_D$. So the image of the geodesic $(AB)$ under the quotient map $q$ to $S_{AXB}$ will be not a geodesic but a broken geodesic $q((AB))$ with angles $\pi-\alpha_{D_i}$ at the vertices. Yet all its segments will have the same length as the corresponding segments of $(AB)$.
It is easy to see that the distances between $q(X)\in S$ and each point on $q((AB))$ will be the same as in the original geodesic triangle $AXB$. Finally, the angle $AYX$
in $S_{AXB}$ remain the same as
in the original triangle. In particular, the angle $AYX$ in $S_{AXB}$ is greater than or equal to ${\pi\over 2}$.

To find a lower bound for $|AX|=|q((AX))|$ we first minorize $|q((AX))|$ by
the side length in the hinge with side lengths $dist_S(A,Y)$, $dist_S(X,Y)$, and the angle $\angle_S AYX\geq {\pi\over 2}$ in the plane using the Alexandrov-Toponogov theorem. So, $dist_S(A,X)\geq
\sqrt{dist_S^2(X,Y)+dist_S^2(A,Y)}$.

Now we need a lower bound for $dist_S(A,Y)$. We can use Reshetnyak's majorization theorem (or several times the hinge version of Alexandrov-Toponogov theorem) to conclude that $dist_S(A,Y)$ can be estimated below by the distance
in the plane between the endpoints of the plane broken
line with the segments of the same length as the corresponding segments of $q((AY))\subset q((AB))$ and the angles between two consecutive segments at $D_i$
equal to $\pi-\alpha_{D_i}$. Note that $\Sigma_i\alpha_{D_i}\leq \phi_+(AXY)$, where $\phi_+(AXY)$ denotes the sum of all
positive angular defects at the vertices in the triangle
$AXY$. Now
elementary trigonometry yields the lower bound $\vert AY\vert \cos\phi_+(AXY)$. (See a very similar argument at the end of the proof of Lemma 6.) As it is the square of $dist_S(A,Y)$ that enters the hinge estimate, we square this bound:
$$dist_S^2(A,Y)\geq \vert AY\vert^2\cos^2\phi_+(AXY)=\vert AY\vert^2(1-\sin^2\phi_+(AXY))\geq \vert AY\vert^2(1-\phi^2).$$
We can get a similar lower bound for $dist_S(B,Y)$, and denoting $\vert XY\vert $ by $h$, we would like to proceed similarly to the inequality (2), and write
$$c=l-|AB|\geq {h^2-|AY|^2\phi^2\over 3p}+{h^2-|BY|^2\phi^2\over 3p}\geq {2h^2\over 3p}-{p^2\phi^2\over 3p}.\ \ \ (**)$$
However, as in (2), here we replaced smaller denominators by $3p$ which requires the validity
of the inequalities $h^2\geq |AY|^2\phi^2$ and $h^2\geq |BY|^2\phi^2$.
If both these inequalities hold, then inequality (**) is valid, and it immediately implies that
$h\leq \sqrt{{3\over 2}pc+{p^2\phi^2\over 2}}$. If either of these two inequalities is not valid, then $h^2\leq p^2\phi^2$ (as $\vert AY\vert, \vert BY\vert\leq dist(A,B)\leq p$), and $h\leq p\phi$.

It remains to consider the case when $Y$ coincides with $A$ or $B$; say $Y=A$ (the case $Y=B$ is symmetric). Then $h=|XY|=dist(X,A)$ is the distance from $X$ to $[AB]$, and, since $A$ is the closest point of $[AB]$ to $X$, we have $\angle XAB\geq{\pi\over 2}$. If $h\leq p\phi$, then $h\leq x$ and we are done; so assume $h> p\phi$.

As in the first case, we cut out the discs containing all points of positive curvature in the triangle $AXB$ to obtain a space $S=S_{AXB}$ of non-positive curvature. The quotient map $q$ preserves the distances from $q(X)$ to all points of $q((AB))$; in particular $dist_S(A,X)=|AX|=h$ and $dist_S(X,B)=|XB|$, and the angle at $A$ is preserved, so $\angle_S XAB\geq{\pi\over 2}$. Applying the hinge version of the Alexandrov--Toponogov theorem in $S$ at $A$, and bounding $dist_S(A,B)$ from below by Reshetnyak's majorization theorem exactly as above (so that $dist_S^2(A,B)\geq |AB|^2\cos^2\phi_+(AXB)\geq |AB|^2(1-\phi^2)$), we obtain
$$|XB|=dist_S(X,B)\geq\sqrt{dist_S^2(A,X)+dist_S^2(A,B)}\geq\sqrt{h^2+|AB|^2(1-\phi^2)}.$$
Since $h> p\phi\geq |AB|\phi$, the numerator below is positive, and, replacing the denominator by $3p$ as in $(**)$,
$$|XB|-|AB|\geq{h^2-|AB|^2\phi^2\over\sqrt{h^2+|AB|^2}+|AB|}\geq{h^2-|AB|^2\phi^2\over 3p}\geq{h^2-p^2\phi^2\over 3p}.$$
Therefore
$$c=l-|AB|\geq |AX|+|XB|-|AB|\geq h+{h^2-p^2\phi^2\over 3p},$$
so that $3pc\geq 3ph+h^2-p^2\phi^2\geq 2h^2-p^2\phi^2$ (as $h\leq p$), that is, $h\leq \sqrt{{3\over 2}pc+{p^2\phi^2\over 2}}\leq x$. This completes the proof.
\end{proof}

    

\begin{corollary}
Let $\Sigma$ be a polyhedral surface homeomorphic to the $2$-disk. Let $A$, $B$ be two vertices on the boundary of $\Sigma$ connected by a polyhedral path $\gamma$ of length $l\leq p$ in $\Sigma$, where $p\geq 1$. Let $\theta_i$ denote (positive or negative) turning (outer) angles at the inner vertices of $\gamma$. Assume that all $\theta_i$ satisfy $\vert\theta_i\vert\leq {1\over p}$, and, hence, $\Sigma_i\vert\theta_i\vert\leq 1$ (radian).

Assume that the total absolute curvature $\mu$ of $\Sigma$ (that is, the sum of the absolute
values of all angular defects) is less than $0.1$ radian. Further, denote $l-dist_\Sigma(A,B)$ by $c$. Let $\alpha$ be a shortest
geodesic between $A$ and $B$ in $\Sigma$. 
Then the area of the domain $D$ of $\Sigma$ between $\alpha$ and 
$\gamma$ does not exceed
$\max\{7.26p^{3\over 2}\sqrt{c}+1.61\mu p^2, 2.38\mu p^2\}$.


\end{corollary}

\begin{proof}




\newcommand{\dist}{\text{dist}}
\newcommand{\area}{\text{Area}}




Note that $D$ consists of one or several domains $D_i$ homeomorphic to the $2$-disk
each bounded by an arc $\alpha_i$ of $\alpha$ and an arc $\gamma_i$ of $\gamma$.
The boundary of $D_i$
is the union of two arcs, $\gamma$ and $\alpha$, meeting only at their common endpoints $A_i$ and $B_i$. Moreover, $\alpha_i$ is the minimizing geodesic between $A_i$ and $B_i$ in $D_i$. 
Applying
Lemma 9 to $D_i$, we see that the distances from all points of $\gamma_i$ to
$\alpha_i$ do not exceed $x=\max\{\sqrt{{3\over 2}pc+{p^2\mu^2\over 2}}, \mu p\}$.



We are going to modify $D=\cup_i D_i$
so that all singular inner vertices of $\Sigma$ in the modified $D$ will have negative curvature. For each $i=1,2,\ldots$ we are going to proceed as follows. First, order all an inner vertices $a$ of $\Sigma$ in $D_i$ of positive curvature in the increasing order with respect to the distance from $a$ to $\gamma_i$. Then, proceeding in the order just established, connect each $a$ with the nearest point $y$ on $\gamma_i$ by a geodesic segment $g$, cut along $g$, and glue in an isosceles triangle with vertex $a$, the angle at $a$ equal to the angular defect of $a$, and the side length equal to the length of $g$. The point $a$ will now become flat, the area will become only larger, and the maximum distance from a point of $\gamma_i^{\rm new}$ (with the insertions of the third side of the inserted triangle) to $\alpha_i$ 
will remain the same. So, this distance still will not exceed $x=\max\{\sqrt{{3\over 2}pc+{p^2\mu^2\over 2}}, \mu p\}$.

After performing this procedure we can assume that the angular defects at all
vertices of $\Sigma$ inside $D_i$ are non-negative. 

Now consider all inner vertices $a$ of $\Sigma$ in $D$ with negative curvatures. This set also includes the inner vertices of $\Sigma$ that are vertices of $\alpha$. For each $i$ we similarly order the set of these vertices $a$ in increasing order with respect to the distance from $a$ to $\gamma$ (starting from the nearest points to $\gamma$). Then, proceeding in this order, for each of them we perform the following operation.
Since the minimal geodesics
are unique, the minimal geodesics from points $z$ in $\alpha_i$ to $\gamma_i$ depend continuously on $z$, and one of them passes through $a$. Therefore, the distance from $a$ to $\gamma_i$ does not exceed $x$. Let $\mu_a$ denote the absolute value of the angular defect at $a$. Consider a shortest geodesic $g$ from $a$ to $\gamma_i$. 
Remove the triangle $A$ bounded by two geodesics incident to $a$ forming angle $\mu_a$ with each other, such that $g$ is the bisector of this angle. The third side of triangle $A$ is an arc
of $\gamma_i$. Note that the third side
of the triangle might be not a straight line segment but a broken line.
The distances from $a$ to
other points on the arc of $\gamma$ in
the boundary of $A$ (=the third side of the triangle) do not exceed $1.09x$, and, therefore, the area of the removed triangle does not exceed $1.2{\mu_ax^2\over 2}$. To see this, let $z$ be a point on the third side of the triangle. Consider the triangle with one side being the bisector $b$ of the angle of the triangle $A$ at $a$, and the other side being the straight line segment $az$. The angle between these sides in $A$ does not exceed half of the angle at $a$, that is ${\mu\over 2}\leq 0.05$ radian. The triangle is flat, and its
angle at the vertex $w$ of $b$ that is different from $a$ does not exceed ${\pi\over 2}+1$, where $1$ is the upper bound for the turning angle of $\gamma$, and ${\pi\over 2}$ is the value of the angle in the case, if 
$\gamma$ were a straight line (without
vertices) on an open arc including  $w$ and $z$. Now the law of sines implies the upper bound $1.09 x$ for $|az|$.

Glue points on one side of (already removed) $A$ incident to $a$ to the points on the other side incident to $a$ that are at the same distances to $a$. Of course, the areas of the triangles $A$ that we just removed will need to be later added to our upper bound for the area of the resulting flat surface. Our calculation implies that the total area of these triangles is
at most $1.2{\mu x^2\over 2}=0.6\mu x^2\leq\max\{0.9\mu pc+0.3\mu^2p^2, 0.6\mu^2p^2\}$.

Now we finally find ourselves in the situation when the resulting surface, $D$, is flat. Note that it is contained in the $x$-neighborhood of
$\alpha$. Therefore, the area of $D$ does not exceed the area of the $x$-neighborhood of $\alpha$ in the plane that
does not exceed $2x{\rm length}(\alpha)+\pi x^2$. 

If $x=\sqrt{{3\over 2}pc+{p^2\mu^2\over 2}}$,
then $2x{\rm length}(\alpha)+\pi x^2
\leq 2p\sqrt{{3\over 2}pc+{p^2\mu^2\over 2}}+\pi({3\over 2}pc+{p^2\mu^2\over 2})\leq \sqrt{6}p^{3\over 2}\sqrt{c}+\sqrt{2}\mu p^2+{3\over 2}\pi pc+{\pi\over 2}\mu^2p^2$. Now we add back $0.9\mu pc+0.3\mu^2p^2$.
Taking into account that $\mu\leq 0.1$ and, therefore $\mu^2\leq 0.1\mu$, and also that $c<p$ (and, therefore, $pc<p^{3\over 2}\sqrt{c}$), we can replace the resulting upper bound by the following simpler looking upper bound: ${\rm Area}(D)\leq 7.26p^{3\over 2}\sqrt{c}+1.61\mu p^2$.

If $x=p\mu$, then $2x{\rm length}(\alpha)+\pi x^2\leq 2p^2\mu+\pi\mu^2 p^2$. Adding back $0.6\mu^2p^2$, and using the inequality $\mu^2\leq 0.1\mu$ we see that ${\rm Area}(D)\leq 2.38\mu p^2$.
\end{proof} 

%




\section{Geodesic nets on surfaces}

Assume that a geodesic net in a small convex metric ball in $M$ was constructed from a closed geodesic net in $M$ by applying the coarea formula as we sketched in section 2.2. Consider
the dual complex of this geodesic net constructed as explained in section 3.3. This dual complex is a polyhedral surface.
Then we proceed as in section 4 and find a subset $D$ of this
polyhedral surface that is also a polyhedral surface homeomorphic to a disk, where all angles are ${1\over p}$-close to $0$ or to $\pi$, exactly two angles at each inner vertex are close to $\pi$, and exactly one
angle at each boundary vertex is close to $\pi$. Moreover, each face is a convex polygon with an even number of sides with exactly
two angles being acute, and the distance between the vertices with acute angles along the boundary of the polygon is equal to the half of the perimeter of the face. We assume that the boundary of $D$ consists of two broken lines of the same length connecting two vertices $A$ and $B$; $D$ is a subset of a larger polyhedral surface $D'$ that has the same properties as $D$; all vertices on the boundary of $D$ other than $A$ and $B$ are inner vertices of $D'$ (Figure~\ref{fig:D-in-Dprime}). Finally, we assume that at each vertex $v\in\partial D\setminus\{A, B\}$ we have a stronger condition. The angle of $D$ at $v$ is $\delta$-close to $\pi$; all angles of faces in $D$ and $D'\setminus D$ adjacent to $v$ are $\delta$-close to $0$ or $\pi$, and there is exactly one face adjacent to $v$
on each side of the boundary of $D$ such that the angle is $\delta$-close to $\pi$. It is very important here that $\delta$ can be an uncontrollably small positive number, much smaller than ${1\over p}$. The purpose of Lemma 10
is to first give an upper bound for the total absolute
curvature $\mu$ of $D$ that depends linearly on $\delta$, and then another linear upper bound for $\mu$ in terms of $\sqrt{c}$, where $c$ is the difference between the integer length of either of two halves of the boundary of $D$ connecting $A$ and $B$ and the distance between 
$A$ and $B$ in $D'$. Finally, one obtains an upper bound for the area of $D$ by plugging this upper bound for $\mu$ in terms of $c$ to Corollary 9.

\begin{figure}[ht]
\centering
\resizebox{\ifdim\width>\textwidth \textwidth\else\width\fi}{!}{%
\begin{tikzpicture}[
   >={Stealth[length=2.2mm]},
   pth/.style={line width=0.9pt, blue!70!black},
   bd/.style={line width=1.3pt, blue!75!black},
   bdp/.style={line width=1.3pt, blue!45!black},
   edot/.style={circle, fill=black, inner sep=1.7pt},
   scale=1.0, yscale=1.55]
\fill[blue!6] (0.000,0.000) -- (0.500,0.158) -- (1.000,0.095) -- (1.500,0.279) -- (2.000,0.150) -- (2.500,0.290) -- (3.000,0.150) -- (3.500,0.290) -- (4.000,0.150) -- (4.500,0.290) -- (5.000,0.150) -- (5.500,0.290) -- (6.000,0.150) -- (6.500,0.290) -- (7.000,0.150) -- (7.500,0.279) -- (8.000,0.095) -- (8.500,0.158) -- (9.000,0.000) -- (9.000,0.000) -- (8.500,0.638) -- (8.000,0.995) -- (7.500,1.419) -- (7.000,1.350) -- (6.500,1.490) -- (6.000,1.350) -- (5.500,1.490) -- (5.000,1.350) -- (4.500,1.490) -- (4.000,1.350) -- (3.500,1.490) -- (3.000,1.350) -- (2.500,1.490) -- (2.000,1.350) -- (1.500,1.419) -- (1.000,0.995) -- (0.500,0.638) -- (0.000,0.000) -- cycle;
\fill[blue!16] (0.000,0.000) -- (0.500,0.278) -- (1.000,0.320) -- (1.500,0.564) -- (2.000,0.450) -- (2.500,0.590) -- (3.000,0.450) -- (3.500,0.590) -- (4.000,0.450) -- (4.500,0.590) -- (5.000,0.450) -- (5.500,0.590) -- (6.000,0.450) -- (6.500,0.590) -- (7.000,0.450) -- (7.500,0.564) -- (8.000,0.320) -- (8.500,0.278) -- (9.000,0.000) -- (9.000,0.000) -- (8.500,0.518) -- (8.000,0.770) -- (7.500,1.134) -- (7.000,1.050) -- (6.500,1.190) -- (6.000,1.050) -- (5.500,1.190) -- (5.000,1.050) -- (4.500,1.190) -- (4.000,1.050) -- (3.500,1.190) -- (3.000,1.050) -- (2.500,1.190) -- (2.000,1.050) -- (1.500,1.134) -- (1.000,0.770) -- (0.500,0.518) -- (0.000,0.000) -- cycle;
\draw[bdp] (0.000,0.000) -- (0.500,0.158) -- (1.000,0.095) -- (1.500,0.279) -- (2.000,0.150) -- (2.500,0.290) -- (3.000,0.150) -- (3.500,0.290) -- (4.000,0.150) -- (4.500,0.290) -- (5.000,0.150) -- (5.500,0.290) -- (6.000,0.150) -- (6.500,0.290) -- (7.000,0.150) -- (7.500,0.279) -- (8.000,0.095) -- (8.500,0.158) -- (9.000,0.000);
\draw[bd] (0.000,0.000) -- (0.500,0.278) -- (1.000,0.320) -- (1.500,0.564) -- (2.000,0.450) -- (2.500,0.590) -- (3.000,0.450) -- (3.500,0.590) -- (4.000,0.450) -- (4.500,0.590) -- (5.000,0.450) -- (5.500,0.590) -- (6.000,0.450) -- (6.500,0.590) -- (7.000,0.450) -- (7.500,0.564) -- (8.000,0.320) -- (8.500,0.278) -- (9.000,0.000);
\draw[pth] (0.000,0.000) -- (0.500,0.398) -- (1.000,0.545) -- (1.500,0.849) -- (2.000,0.750) -- (2.500,0.890) -- (3.000,0.750) -- (3.500,0.890) -- (4.000,0.750) -- (4.500,0.890) -- (5.000,0.750) -- (5.500,0.890) -- (6.000,0.750) -- (6.500,0.890) -- (7.000,0.750) -- (7.500,0.849) -- (8.000,0.545) -- (8.500,0.398) -- (9.000,0.000);
\draw[bd] (0.000,0.000) -- (0.500,0.518) -- (1.000,0.770) -- (1.500,1.134) -- (2.000,1.050) -- (2.500,1.190) -- (3.000,1.050) -- (3.500,1.190) -- (4.000,1.050) -- (4.500,1.190) -- (5.000,1.050) -- (5.500,1.190) -- (6.000,1.050) -- (6.500,1.190) -- (7.000,1.050) -- (7.500,1.134) -- (8.000,0.770) -- (8.500,0.518) -- (9.000,0.000);
\draw[bdp] (0.000,0.000) -- (0.500,0.638) -- (1.000,0.995) -- (1.500,1.419) -- (2.000,1.350) -- (2.500,1.490) -- (3.000,1.350) -- (3.500,1.490) -- (4.000,1.350) -- (4.500,1.490) -- (5.000,1.350) -- (5.500,1.490) -- (6.000,1.350) -- (6.500,1.490) -- (7.000,1.350) -- (7.500,1.419) -- (8.000,0.995) -- (8.500,0.638) -- (9.000,0.000);
\node[edot] at (0,0) {};
\node[edot] at (9.000,0) {};
\node[left]  at (0,0) {$A$};
\node[right] at (9.000,0) {$B$};
\node[blue!45!black,font=\footnotesize,above,fill=white,inner sep=0.5pt] at (4.500,1.560) {$\gamma_n$};
\node[blue!75!black,font=\footnotesize,fill=white,inner sep=0.5pt] at (2.500,1.335) {$\gamma_{n-1}$};
\node[blue!75!black,font=\footnotesize,fill=white,inner sep=0.5pt] at (6.500,0.470) {$\gamma_2$};
\node[blue!45!black,font=\footnotesize,fill=white,inner sep=0.5pt] at (4.500,0.150) {$\gamma_1$};
\node[font=\footnotesize,fill=white,inner sep=0.5pt] at (6.000,0.900) {$D$};
\node[font=\footnotesize,fill=white,inner sep=0.5pt] at (7.250,1.220) {$D'$};
\end{tikzpicture}
}
\caption{The nested thin surfaces $D\subset D'$, bounded by $\gamma_2,\gamma_{n-1}$ and by $\gamma_1,\gamma_n$; all edges are unit segments and only the vertical scale is exaggerated.}
\label{fig:D-in-Dprime}
\end{figure}
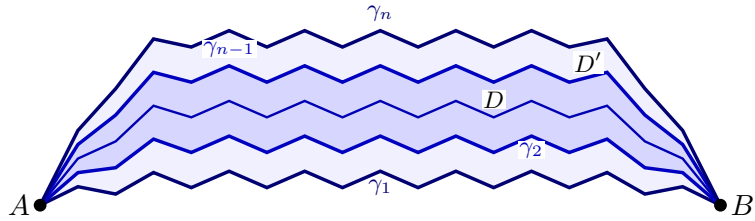

\begin{lemma}
Assume that the polyhedral surface $\Sigma$
with all sides of length $1$ is obtained as above starting from a closed Riemannian surface $M$ with $\vert K\vert\leq 1$ and a subnet $T$ of a closed geodesic net $Net$ of length $\leq l$ in $M$ that is contained in a metric ball $B_a(r_*)$ in $M$ of radius $r_*$ centered at a point $a\in M$ so that $T$ has at most $p={2 l\over r}$ unbalanced points, and all these points are on $\partial B_a(r_*)$.
Here $r_*\in ({r\over 2}, r)$, where $r={1\over 2000}\min\{1, conv(M), {1\over l}\}$, and $conv(M)$ denotes the convexity radius of $M$.


Consider a polyhedral subset $D$ of $\Sigma$ that consists of closed faces, edges, and vertices of $\Sigma$ and is homeomorphic to the $2$-disc. Assume that the boundary of $D$ consists
of two vertices $A$ and $B$ connected by two vertex-disjoint edge paths $P_1$ and $P_2$ with $k\leq p$ edges of length $1$ each. 

Assume that the angles
$\alpha_i$ in all faces of $\Sigma$ at all boundary vertices of $D$ other than $A$ or $B$ are $\delta$-close to either $0$ or $\pi$ , $(i=1,\ldots , k)$, and exactly one of these angles in $D$ as well as exactly one of the faces outside of $D$ are $\delta$-close to $\pi$  for each of the considered vertices. Further, assume that for each vertex $v$ of $\partial D$ other than $A$ and $B$ the angle $\beta_v$ between edges of the boundary of $D$ incident to $v$ is $\delta$-close to $\pi$. (This angle is equal to the sum
of all angles in the faces of $D$ that are incident to $v$.)
Assume that $\delta\leq {1\over 8p}$.
Finally, assume that the angles at $A$ and $B$ are acute.

Let $\mu$ denote the total
absolute angular defect (that is, the total absolute curvature) of the domain $D$. By definition, it is the sum over all inner vertices $v$ of $D$ of the absolute values of the difference of the angle at $v$ and $2\pi$.

(1) Then $\mu$ satisfies
the following inequality:
$\mu\leq  960r_*^2p^2\delta < 1000r^2p^2\delta=4000l^2\delta$.



Assume, in addition, that all faces of the polyhedral surface are even-sided polygons with two opposite
vertices with acute angles $\leq {1\over p}$ connected with two arcs of the same integer length $\leq k$. Assume that all but two angles in each face are in the interval $[\pi-{1\over p},\pi]$. Finally, assume that
for each interior vertex exactly two angles incident to this vertex are ${1\over p}$-close to $\pi$. 
Then:

(2) For each pair of vertices $v_1, v_2$
of $D$ and each pair of paths $p_1, p_2$ in the $1$-skeleton of $D$ from $v_1$ to $v_2$, the lengths of $p_1$ and $p_2$ are equal, providing that the distance to $v_2$ monotonously decreases along both $p_1$ and $p_2$.

(3) Let $D'\supset D$ be a polyhedral subset of $\Sigma$ satisfying exactly the same conditions as $D$ for the same vertices $A$, $B$ and, in addition, containing all faces of $\Sigma$ incident to all vertices $v\in\partial D\setminus\{A, B\}$ outside of $D$. 
Assume that $c$ denotes the difference between
$p$ and the length of a minimal geodesic between $A$ and $B$ in $D'$. Then $2\sqrt{2}\sqrt{c}\geq\delta$.
This estimate does not require the upper bound  for $\delta\leq {1\over 8p}$ in the text of the theorem.


(4) Let $\tilde l=\max\left\{l, 1, {1\over {\rm conv}(M)}\right\}$.
Then the area of  $D$ is less than 
$(100\tilde l)^6 c^{1\over 2}$.


\end{lemma}
\begin{proof}
(1) The absolute value of the curvature of the original Riemannian surface $M$ is bounded by $1$. By construction, each vertex $v$ in the interior of the polyhedral surface $D$ corresponds to a face $F_v$
of the original surface, and the angular defect is the integral of the curvature over $F_v\subset M$. Therefore, its absolute value
does not exceed the area of $F_v$. The total absolute curvature of $D$ does not exceed the total area of all faces of the original geodesic net in $M$ corresponding to all the vertices inside $D$. Denote the union of all faces $F_v$ , $v\in\ int(D)$, by $O$. We have
$\mu\leq Area(O)$. Our goal is to find an upper bound for $Area(O)$.

Recall, that the original net is contained in
the ball $B_a(r_*)$ centered at a point $a\in M$ of radius $r_*\in ({r\over2}, r)$, for $r={1\over 2000}\min\{1, conv(M), {1\over l}\}$.

Observe that, as $pr=2l$, for
 $l\geq \max\{1, {1\over conv(M)}\}$, $r={1\over 2000l}$, and $p=4000l^2$.

Our strategy would be to choose certain geodesic $\gamma$, and consider Fermi coordinates (a.k.a. geodesic normal coordinates) adapted to $\gamma$.
We going to prove that
$O$ will be contained in the strip $S$ between
two geodesics orthogonal to $\gamma$ and passing through points on $\gamma$ at the distance $W\leq Const_1\ p^2r_*\delta+Const_2pr_* Q$ from each other, where a constant $Q$ is known to be greater than or equal to the area of the set of all points
in $B_a(r_*)$ with orthogonal projections to $\gamma$ in the same interval of length $W$. (The constants $Const_1$ and $Const_2$ will be chosen later.)
Easy comparison argument implies that the area of the strip in $B_a(r_*)$ with the projection on $\gamma$ in the interval of length $W$ is the maximal possible when the disk $B_a(r_*)$ is hyperbolic, $\gamma$ passes through its center, and the center of the disk is also the center of the interval of length $W$ on $\gamma$. For $r_*\leq 0.5$ the area of the strip does not exceed $2W\sinh(r_*)<3Wr_*=
3Const_1\ p^2\delta r_*^2+3Const_2r_*^2p Q$.
As long as $r_*^2\leq {1\over 6p\ Const_2}$, and
$Q=6Const_1 p^2\delta r_*^2$, the area of the strip $S$ does not exceed $Q$, and the assumption in the definition of $W$ is satisfied. So, we obtain the desired upper bound 
$6Const_1 p^2\delta r_*^2$ 
for the area of $S$ and, therefore, for the area of $O$. As we already observed, the same expression will automatically be an upper bound for the total absolute angular defect $\mu$ of $D$.

Consider the boundary of $D$ and its ``preimage" in $M$.
The boundary of $D$ has two vertices $A$ and $B$ with acute angles ($\leq \delta$). There could be many edges adjacent to $A$ (or $B$). These vertices
and edges inside $D$ correspond to broken geodesic lines $l_A$ and $l_B$ in $B_a(r_*)\subset M$. The sum of the outer
angles at all vertices of $l_A$ (correspondingly, $l_B$)
is equal to the angle $\alpha$ at $A$ (correspondingly,
the angle $\beta$ at $B$). The sum of $\alpha$ and $\beta$
does not exceed the sum of $\mu$ and all the outer angles at vertices of $\partial D$ other than $A$ and $B$. That is,
the sum of these angles at does not exceed $2(p-1)\delta+\mu$.
These lines are parts of the boundary of $O$ and are concave to the domain $O\subset M$. 

Points $A$ and $B$ are connected in $\partial D$ by two broken lines. The angles at all vertices of each of these lines are $\delta$-close to $\pi$. Each of these vertices $v$ corresponds to a geodesic polygon in $M$ (that we previously called $F_v$). However, these polygons will not be in $O$ as we consider now the vertices on the boundary of $D$. Yet,
each of them will be adjacent to $O$ and will share a part of its boundary with $\partial O$. This part is a connected arc that we will call $\alpha_v$. In fact, the union of these two arcs $\alpha_v$ will contain $\partial O\setminus (l_A\cup l_B)$. Each of these arcs shares exactly one segment with $l_A$, namely the first one, and exactly one with $l_B$, namely the last one. The angles $\alpha_{vi}$ between consecutive edges in $D$ incident to $v$ will be $\delta$-close to either $0$ or $\pi$. Exactly one of these angles, $\alpha_{vi_0}$, will be close to $\pi$. The angles between the corresponding edges in the boundary
of $F_v$ will be $\pi-\alpha_{vi}$; exactly one of them will be $\leq \delta$, the rest of them will be close to $\pi$.  The sum of
these angles $\alpha_{vi}$ that are $\delta$-close to $0$ and the only angle $\pi-\alpha_{vi_0}$ that is $\delta$ close to $0$
will be less than $\delta$.

\begin{figure}[ht]
\centering
\resizebox{\ifdim\width>\textwidth \textwidth\else\width\fi}{!}{%
\begin{tikzpicture}[
   netl/.style={line width=0.7pt, black!80},
   env/.style={line width=1.6pt, blue!55!black},
   arc/.style={line width=1.6pt, red!65!black},
   arcb/.style={line width=1.1pt, red!45!black},
   othr/.style={line width=1.1pt, black!40},
   dual/.style={line width=0.8pt, black},
   dashd/.style={line width=0.8pt, black!40, dash pattern=on 3pt off 2pt},
   stub/.style={line width=0.6pt, black!60},
   angm/.style={line width=0.6pt, red!70!black},
   dot/.style={circle, fill=black, inner sep=1.0pt},
   odot/.style={circle, draw=black!60, fill=white, line width=0.4pt, inner sep=0.7pt},
   scale=1.0]

\begin{scope}
\clip (-3.120,-5.120) rectangle (3.008,4.240);

\fill[blue!17] (-0.733,-1.372) -- (-0.959,-0.258) -- (-0.936,-1.717) -- cycle;
\fill[blue!17] (-0.544,-2.298) -- (-0.733,-1.372) -- (-0.936,-1.717) -- (-0.914,-3.166) -- cycle;
\fill[blue!17] (-0.389,-3.063) -- (-0.544,-2.298) -- (-0.914,-3.166) -- (-0.889,-4.772) -- cycle;
\fill[blue!17] (-0.969,0.417) -- (-1.249,1.170) -- (-0.959,-0.258) -- cycle;
\fill[blue!17] (-0.470,-0.927) -- (-0.969,0.417) -- (-0.959,-0.258) -- (-0.733,-1.372) -- cycle;
\fill[blue!17] (-0.233,-1.566) -- (-0.470,-0.927) -- (-0.733,-1.372) -- (-0.544,-2.298) -- cycle;
\fill[blue!17] (-0.075,-1.991) -- (-0.233,-1.566) -- (-0.544,-2.298) -- (-0.389,-3.063) -- cycle;
\fill[blue!17] (-1.420,2.011) -- (-1.796,2.641) -- (-1.249,1.170) -- cycle;
\fill[blue!17] (-0.983,1.278) -- (-1.420,2.011) -- (-1.249,1.170) -- (-0.969,0.417) -- cycle;
\fill[blue!17] (-0.071,-0.252) -- (-0.983,1.278) -- (-0.969,0.417) -- (-0.470,-0.927) -- cycle;
\fill[blue!17] (0.161,-0.641) -- (-0.071,-0.252) -- (-0.470,-0.927) -- (-0.233,-1.566) -- cycle;
\fill[blue!17] (0.268,-0.820) -- (0.161,-0.641) -- (-0.233,-1.566) -- (-0.075,-1.991) -- cycle;
\fill[blue!17] (-0.071,-0.252) -- (0.161,-0.641) -- (0.671,0.557) -- (0.931,1.442) -- cycle;
\fill[blue!17] (0.161,-0.641) -- (0.268,-0.820) -- (0.671,0.557) -- cycle;
\fill[blue!17] (0.931,1.442) -- (0.671,0.557) -- (1.352,2.155) -- cycle;

\fill[red!8] (-2.390,4.240) -- (-2.749,4.240) -- (-1.796,2.641) -- cycle;
\fill[red!8] (-1.873,4.240) -- (-2.390,4.240) -- (-1.796,2.641) -- (-1.420,2.011) -- cycle;
\fill[red!8] (-1.028,4.240) -- (-1.873,4.240) -- (-1.420,2.011) -- (-0.983,1.278) -- cycle;
\fill[red!18] (1.750,4.240) -- (-1.028,4.240) -- (-0.983,1.278) -- (-0.071,-0.252) -- (0.931,1.442) -- cycle;
\fill[red!8] (2.240,4.240) -- (1.750,4.240) -- (0.931,1.442) -- (1.352,2.155) -- cycle;
\fill[red!8] (2.585,4.240) -- (2.240,4.240) -- (1.352,2.155) -- cycle;
\fill[black!6] (-2.948,-5.120) -- (-1.746,-5.120) -- (-0.914,-3.166) -- (-0.936,-1.717) -- cycle;
\fill[black!6] (-1.746,-5.120) -- (-0.991,-5.120) -- (-0.889,-4.772) -- (-0.914,-3.166) -- cycle;
\fill[black!6] (-0.991,-5.120) -- (-0.884,-5.120) -- (-0.889,-4.772) -- cycle;
\fill[black!6] (-0.884,-5.120) -- (0.030,-5.120) -- (-0.389,-3.063) -- (-0.889,-4.772) -- cycle;
\fill[black!6] (0.030,-5.120) -- (1.088,-5.120) -- (-0.075,-1.991) -- (-0.389,-3.063) -- cycle;
\fill[black!6] (1.088,-5.120) -- (2.833,-5.120) -- (0.268,-0.820) -- (-0.075,-1.991) -- cycle;

\draw[netl] (2.833,-5.120) -- (-2.749,4.240);
\draw[netl] (1.088,-5.120) -- (-2.390,4.240);
\draw[netl] (0.030,-5.120) -- (-1.873,4.240);
\draw[netl] (-0.884,-5.120) -- (-1.028,4.240);
\draw[netl] (-0.991,-5.120) -- (1.750,4.240);
\draw[netl] (-1.746,-5.120) -- (2.240,4.240);
\draw[netl] (-2.948,-5.120) -- (2.585,4.240);

\draw[env] (-1.796,2.641) -- (-1.249,1.170) -- (-0.959,-0.258) -- (-0.936,-1.717);
\draw[env] (0.268,-0.820) -- (0.671,0.557) -- (1.352,2.155);
\draw[othr] (-0.936,-1.717) -- (-0.914,-3.166);
\draw[arcb] (-1.420,2.011) -- (-1.796,2.641);
\draw[othr] (-0.914,-3.166) -- (-0.889,-4.772);
\draw[arcb] (-0.983,1.278) -- (-1.420,2.011);
\draw[arc] (0.931,1.442) -- (-0.071,-0.252) -- (-0.983,1.278);
\draw[othr] (-0.889,-4.772) -- (-0.389,-3.063);
\draw[arcb] (1.352,2.155) -- (0.931,1.442);
\draw[othr] (-0.389,-3.063) -- (-0.075,-1.991);
\draw[othr] (-0.075,-1.991) -- (0.268,-0.820);

\draw[angm] (-0.071,-0.252) ++(59.4:0.55) arc (59.4:120.8:0.55);
\node[red!70!black,font=\footnotesize,above] at (-0.071,0.348) {$\leq\delta$};
\end{scope}

\node[font=\footnotesize] at (-2.749,4.460) {$e_{1}$};
\node[font=\footnotesize] at (-2.390,4.800) {$e_{2}$};
\node[font=\footnotesize] at (-1.873,4.460) {$e_{3}$};
\node[font=\footnotesize] at (-1.028,4.460) {$e_{4}$};
\node[font=\footnotesize] at (1.750,4.460) {$e_{5}$};
\node[font=\footnotesize] at (2.240,4.460) {$e_{6}$};
\node[font=\footnotesize] at (2.585,4.800) {$e_{7}$};
\node[font=\scriptsize] at (-2.570,5.160) {$F_{v_{1}}$};
\node[font=\scriptsize] at (-2.132,5.500) {$F_{v_{2}}$};
\node[font=\scriptsize] at (-1.451,5.160) {$F_{v_{3}}$};
\node[font=\scriptsize] at (0.361,5.160) {$F_{v_{4}}$};
\node[font=\scriptsize] at (1.995,5.160) {$F_{v_{5}}$};
\node[font=\scriptsize] at (2.412,5.500) {$F_{v_{6}}$};

\draw[line width=0.45pt,black!55] (0.950,0.595) -- (0.430,0.595);
\node[red!65!black,font=\footnotesize,right] at (1.040,0.595) {$\alpha_{v_{4}}$};
\draw[line width=0.45pt,black!55] (-1.785,1.906) -- (-1.523,1.906);
\node[blue!55!black,font=\footnotesize,left] at (-1.875,1.906) {$l_A$};
\draw[line width=0.45pt,black!55] (0.732,-0.131) -- (0.470,-0.131);
\node[blue!55!black,font=\footnotesize,right] at (0.822,-0.131) {$l_B$};
\draw[line width=0.45pt,black!55] (-2.149,-3.325) -- (-0.684,-3.325);
\node[black,font=\footnotesize,left] at (-2.239,-3.325) {$O$};
\node[black!45,font=\scriptsize] at (0,-5.360) {the other broken line};
\node[font=\footnotesize] at (0,-5.840) {(a)};

\draw[dashd] (4.678,-0.293) -- (5.866,-0.995) -- (7.136,-1.536) -- (8.460,-1.924) -- (9.840,-1.902) -- (11.192,-1.628) -- (12.486,-1.147) -- (13.671,-0.440);
\draw[dual] (4.678,-0.293) -- (5.863,0.414) -- (7.157,0.895) -- (8.509,1.170) -- (9.889,1.191) -- (11.213,0.803) -- (12.483,0.262) -- (13.671,-0.440);
\draw[stub] (4.678,-0.293) -- (6.058,-0.272);
\node[odot] at (6.058,-0.272) {};
\draw[stub] (4.678,-0.293) -- (6.030,-0.018);
\node[odot] at (6.030,-0.018) {};
\draw[stub] (4.678,-0.293) -- (5.972,0.188);
\node[odot] at (5.972,0.188) {};
\draw[stub] (7.157,0.895) -- (5.972,0.188);
\node[odot] at (5.972,0.188) {};
\draw[stub] (8.509,1.170) -- (7.324,0.463);
\node[odot] at (7.324,0.463) {};
\draw[stub] (9.889,1.191) -- (8.704,0.484);
\node[odot] at (8.704,0.484) {};
\draw[stub] (9.889,1.191) -- (11.077,0.489);
\node[odot] at (11.077,0.489) {};
\draw[stub] (11.213,0.803) -- (12.401,0.101);
\node[odot] at (12.401,0.101) {};
\draw[stub] (13.671,-0.440) -- (12.401,0.101);
\node[odot] at (12.401,0.101) {};
\draw[stub] (13.671,-0.440) -- (12.347,-0.052);
\node[odot] at (12.347,-0.052) {};
\node[dot] at (4.678,-0.293) {};
\node[dot] at (5.863,0.414) {};
\node[dot] at (7.157,0.895) {};
\node[dot] at (8.509,1.170) {};
\node[dot] at (9.889,1.191) {};
\node[dot] at (11.213,0.803) {};
\node[dot] at (12.483,0.262) {};
\node[dot] at (13.671,-0.440) {};
\draw[angm] (4.678,-0.293) ++(-30.6:0.72) arc (-30.6:30.8:0.72);
\node[red!70!black,font=\scriptsize,right] at (5.738,-0.293) {$\leq\delta$};
\draw[angm] (13.671,-0.440) ++(149.4:0.72) arc (149.4:210.8:0.72);
\node[red!70!black,font=\scriptsize,left] at (12.611,-0.440) {$\leq\delta$};
\draw[angm] (9.889,1.191) ++(-149.2:0.55) arc (-149.2:-30.6:0.55);
\node[red!70!black,font=\scriptsize,below] at (9.889,0.571) {$\alpha_{v_{4}i_0}$};

\node[font=\footnotesize,left] at (4.538,-0.293) {$A$};
\node[font=\footnotesize,right] at (13.811,-0.440) {$B$};
\node[font=\scriptsize,above] at (5.271,0.101) {$e_{1}^{*}$};
\node[font=\scriptsize,above] at (6.510,0.694) {$e_{2}^{*}$};
\node[font=\scriptsize,above] at (7.833,1.072) {$e_{3}^{*}$};
\node[font=\scriptsize,above] at (9.199,1.220) {$e_{4}^{*}$};
\node[font=\scriptsize,above] at (10.551,1.037) {$e_{5}^{*}$};
\node[font=\scriptsize,above] at (11.848,0.573) {$e_{6}^{*}$};
\node[font=\scriptsize,above] at (13.077,-0.049) {$e_{7}^{*}$};
\draw[line width=0.35pt,black!45] (5.863,0.544) -- (5.863,0.884);
\node[font=\scriptsize,above] at (5.863,0.844) {$v_{1}$};
\draw[line width=0.35pt,black!45] (7.157,1.025) -- (7.157,1.365);
\node[font=\scriptsize,above] at (7.157,1.325) {$v_{2}$};
\draw[line width=0.35pt,black!45] (8.509,1.300) -- (8.509,1.640);
\node[font=\scriptsize,above] at (8.509,1.600) {$v_{3}$};
\draw[line width=0.35pt,black!45] (9.889,1.321) -- (9.889,1.661);
\node[font=\scriptsize,above] at (9.889,1.621) {$v_{4}$};
\draw[line width=0.35pt,black!45] (11.213,0.933) -- (11.213,1.273);
\node[font=\scriptsize,above] at (11.213,1.233) {$v_{5}$};
\draw[line width=0.35pt,black!45] (12.483,0.392) -- (12.483,0.732);
\node[font=\scriptsize,above] at (12.483,0.692) {$v_{6}$};
\node[font=\footnotesize] at (9.174,-1.020) {$D$};
\node[black!45,font=\scriptsize,below] at (8.460,-2.024) {the other broken line};
\node[font=\footnotesize] at (9.174,-5.840) {(b)};
\end{tikzpicture}
}
\caption{The net along one of the two broken lines joining $A$ and $B$ in $\partial D$ (a), and the disc $D$ dual to it (b); everything belonging to the other broken line is grey, and angles are exaggerated.}
\label{fig:boundary-faces}
\end{figure}
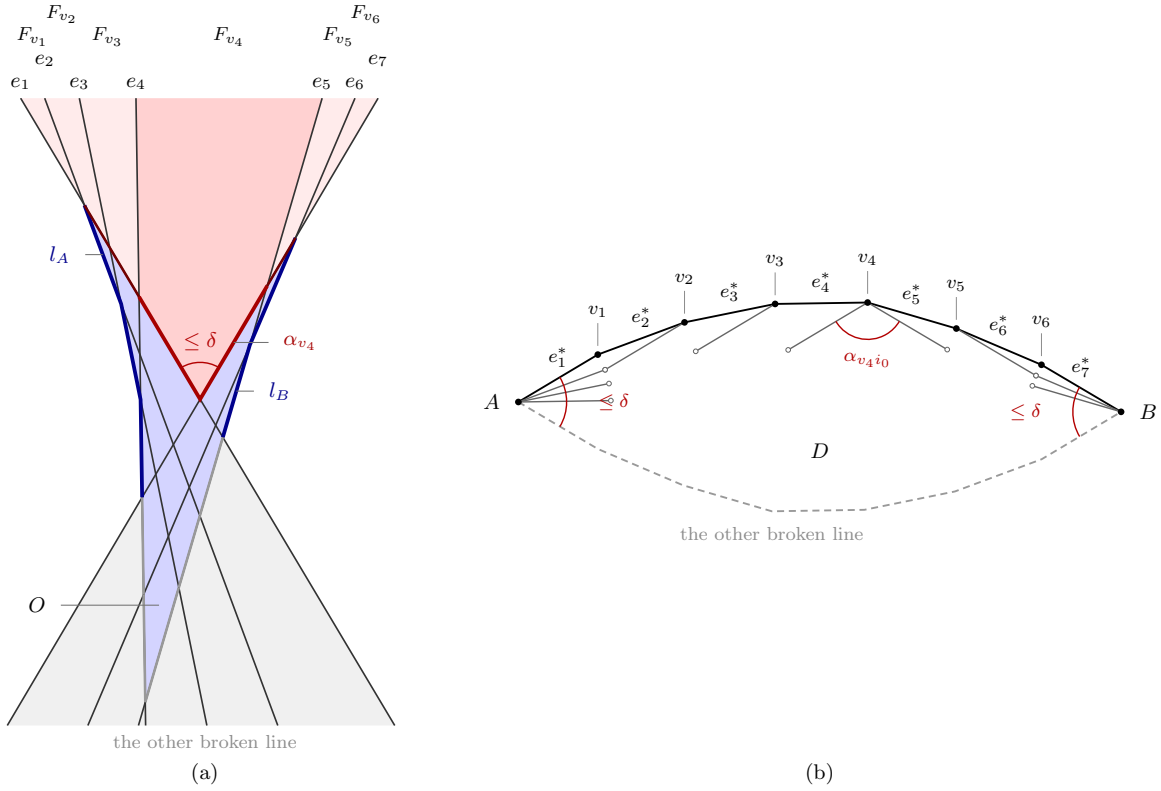

Two edges incident to $v$ in $\partial D$
correspond to ``nearly parallel" geodesic edges in the boundary of $F_v$. Looking
at all edges $e_i^*$ in one of the broken lines
connecting $A$ and $B$ in $\partial D$
denote these geodesic edges by $e_1,\ldots, e_k$. Consider the extensions $E_1,\dots , E_k$ of $e_i$, $i=1,\ldots, k$, until the intersection (in both directions) with
boundary of $B_a(r_*)$. 

We choose the geodesic $\gamma$ as being orthogonal to $E_1$ and passing through the center $a$ (Figure~\ref{fig:strip}).

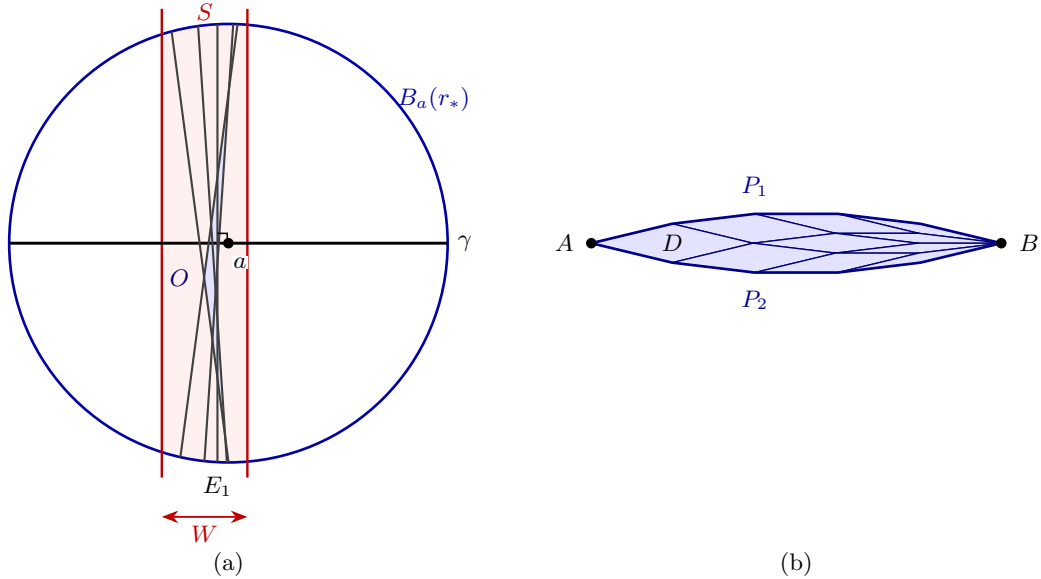
\begin{figure}[ht]
\centering
\resizebox{\ifdim\width>\textwidth \textwidth\else\width\fi}{!}{%
\begin{tikzpicture}[
   >={Stealth[length=2.2mm]},
   ball/.style={line width=1.0pt, blue!65!black},
   geo/.style={line width=1.1pt, black},
   ortho/.style={line width=0.9pt, red!75!black},
   netl/.style={line width=0.8pt, black!75},
   face/.style={fill=blue!11, draw=blue!45!black, line width=0.5pt},
   scale=1.0]

\begin{scope}
\clip (0,0) circle (2.9);
\fill[red!6] (-0.881,-3.1) rectangle (0.251,3.1);
\end{scope}
\fill[blue!11] (-0.212,-1.284) -- (-0.169,-0.641) -- (-0.228,0.252) -- (-0.321,-0.454) -- cycle;
\fill[blue!11] (-0.145,-1.791) -- (-0.145,-1.013) -- (-0.169,-0.641) -- (-0.212,-1.284) -- cycle;
\fill[blue!11] (-0.042,-2.569) -- (-0.145,-1.012) -- (-0.145,-1.791) -- cycle;
\fill[blue!11] (-0.169,-0.641) -- (-0.145,-0.269) -- (-0.145,0.884) -- (-0.228,0.252) -- cycle;
\fill[blue!11] (-0.145,-1.013) -- (-0.145,-0.269) -- (-0.169,-0.641) -- cycle;
\fill[blue!11] (-0.145,-0.269) -- (0.008,2.037) -- (-0.145,0.884) -- cycle;

\draw[ball] (0,0) circle (2.9);
\node[blue!65!black,font=\footnotesize] at (35:3.32) {$B_a(r_*)$};
\draw[ortho] (-0.881,-3.1) -- (-0.881,3.1);
\draw[ortho] (0.251,-3.1) -- (0.251,3.1);
\draw[geo] (-2.9,0) -- (2.9,0);
\node[font=\footnotesize,right] at (2.9,0) {$\gamma$};
\draw[netl] (-0.635,-2.830) -- (0.121,2.897);
\draw[netl] (-0.317,-2.883) -- (0.064,2.899);
\draw[netl] (-0.145,-2.896) -- (-0.145,2.896);
\draw[netl] (-0.020,-2.900) -- (-0.401,2.872);
\draw[netl] (0.002,-2.900) -- (-0.751,2.801);

\node[circle,fill=black,inner sep=1.4pt] at (0,0) {};
\node[font=\footnotesize,fill=white,inner sep=0.5pt] at (0.153,-0.26) {$a$};
\draw[line width=0.6pt,black] (-0.015,0) -- (-0.015,0.13) -- (-0.145,0.13);
\node[font=\footnotesize,below] at (-0.145,-2.956) {$E_1$};
\node[blue!45!black,font=\footnotesize,left] at (-0.381,-0.454) {$O$};
\node[red!75!black,font=\footnotesize] at (-0.315,3.060) {$S$};
\draw[<->, line width=0.7pt, red!75!black] (-0.881,-3.62) -- (0.251,-3.62);
\node[red!75!black,font=\footnotesize,below] at (-0.315,-3.60) {$W$};
\node[font=\footnotesize] at (0,-4.25) {(a)};

\fill[face] (8.062,0.388) -- (9.131,0.131) -- (10.224,0.000) -- (9.154,0.257) -- cycle;
\fill[face] (6.962,0.388) -- (8.031,0.131) -- (9.131,0.131) -- (8.062,0.388) -- cycle;
\fill[face] (5.870,0.257) -- (6.939,0.000) -- (8.031,0.131) -- (6.962,0.388) -- cycle;
\fill[face] (4.800,0.000) -- (5.870,-0.257) -- (6.939,0.000) -- (5.870,0.257) -- cycle;
\fill[face] (8.031,0.131) -- (9.124,0.000) -- (10.224,0.000) -- (9.131,0.131) -- cycle;
\fill[face] (6.939,0.000) -- (8.031,-0.131) -- (9.124,0.000) -- (8.031,0.131) -- cycle;
\fill[face] (5.870,-0.257) -- (6.962,-0.388) -- (8.031,-0.131) -- (6.939,0.000) -- cycle;
\fill[face] (8.031,-0.131) -- (9.131,-0.131) -- (10.224,0.000) -- (9.124,0.000) -- cycle;
\fill[face] (6.962,-0.388) -- (8.062,-0.388) -- (9.131,-0.131) -- (8.031,-0.131) -- cycle;
\fill[face] (8.062,-0.388) -- (9.154,-0.257) -- (10.224,0.000) -- (9.131,-0.131) -- cycle;
\draw[line width=1.0pt, blue!55!black] (4.800,0.000) -- (5.870,-0.257) -- (6.962,-0.388) -- (8.062,-0.388) -- (9.154,-0.257) -- (10.224,0.000);
\draw[line width=1.0pt, blue!55!black] (4.800,0.000) -- (5.870,0.257) -- (6.962,0.388) -- (8.062,0.388) -- (9.154,0.257) -- (10.224,0.000);
\node[circle,fill=black,inner sep=1.4pt] at (4.8, 0.0) {};
\node[circle,fill=black,inner sep=1.4pt] at (10.223581109028368, 0.0) {};
\node[font=\footnotesize,left] at (4.700,0) {$A$};
\node[font=\footnotesize,right] at (10.324,0) {$B$};
\node[blue!55!black,font=\footnotesize,above] at (6.962,0.488) {$P_1$};
\node[blue!55!black,font=\footnotesize,below] at (6.962,-0.488) {$P_2$};
\node[font=\footnotesize] at (5.870,0.000) {$D$};
\node[font=\footnotesize] at (7.512,-4.25) {(b)};
\end{tikzpicture}
}
\caption{The strip argument of Lemma 10: (a) the net in $B_a(r_*)$, together with the union $O$ of its faces, inside a strip $S$ of width $W$; (b) the dual disc $D\subset\Sigma$, a lune tiled by the narrow faces dual to the vertices of the net.}
\label{fig:strip}
\end{figure}

From now on we rely on the following lemma (Lemma 11) proven in the next section. Let $S$ be a reference geodesic and $L$ a piecewise geodesic curve with initial deviation $x$ and turning angle $\phi$, contained within a radius $r$ of $S$. For $r, x, \phi \le 0.5$, the length $P$ of the orthogonal projection of $L$ onto $S$ satisfies:
$P \le 20r(x + \phi)$. Here the initial deviation means the angle between the first segment and the geodesic perpendicular
to $S$ and passing through the initial vertex of (the first segment) of L.

We will apply it to $r=r_*$, $S=\gamma$ and to different curves $L$.
For each vertex $v$ the projection of the part of the boundary of $F_v$ in $\partial D$ can be regarded as the
projections of two broken geodesics: one
that starts from $e_i\subset E_i$ and goes
to the only vertex with an acute angle
in this arc of $F_v$, the second starts at this vertex, an connects
this vertex and $e_{j+1}$. We will need to
know initial deviations for all edges $e_i$.
The initial deviation of $e_1$ is zero,
and its change from $e_i$ to $e_{i+1}$ does not exceed $\delta+$ the area of the strip between geodesics orthogonal to $\gamma$ and passing through the ends 
of $e_i$ and $e_{i+1}$.
Thus, the initial deviation of $e_i$ for all $i$ does not exceed $(k-1)\delta+{\rm Area(S)}$. (The appearance of the area in our calculations of angles here and below
is, of course, due to an application of the Gauss-Bonnet theorem. As $|K|\leq 1$, we estimate above the total curvature in different domains first by the areas of these domains, and then by the area of $S$ that contains all these domains.)
We see that the projection of the relevant part of the boundary of the $i$th geodesic polygon $F_v$ to $\gamma$ will not exceed $20r_*(2p\delta+2{\rm Area}(S))$. The total projection of all $k\leq p$
polygons will not exceed $40r_*(p^2\delta +p{\rm Area}(S))$. The condition that initial deviations $x$ and turning angles $\phi$ do not exceed $0.5$ would follow from
$2({\rm Area}(S)+p\delta)\leq 0.5$ that can be replaced by (stronger) conditions $\delta\leq {1\over 8p}$ and $Area(S)\leq {1\over 8}$. The second condition will be automatically satisfied if $r_*\leq 0.1$.

Then we will need to project $l_A$ that has zero initial deviation and $l_B$ that has the same deviation as $e_k$. The turning angles of $l_A$ and $l_B$ are the angles at $A$ and $B$ in $D$. The sum of these angles can be majorized by the sum of outer angles at all
vertices of $\partial D$ other than $A$ and $B$ and the defect $\mu$. On the other hand, $\mu\leq {\rm Area}(O)\leq {\rm Area}(S)$.
Thus, $\phi$ for both $l_A$ and $l_B$
does not exceed $2p\delta+{\rm Area}(S))$. The initial deviation for $l_A$ is $0$, and for $l_B$ coincides
with the initial deviation for the last edge $e_k$.
The orthogonal projections of $l_A$ will not exceed $20r_*(2p\delta+{\rm Area}(S))$  and for $l_B$
will not exceed 
$20r_*(3p\delta+2{\rm Area}(S))$. It is plausible that the orthogonal
projection of the fourth arc in $\partial O$
will be covered by the projections of the remaining three. However, we prefer to simply add the possible contribution of the fourth side. It can
be calculated almost exactly as the contribution of
the first side (namely, the union of $k$ polygonal arcs between edges $e_i$ and $e_{i+1}$). The difference is 
in the first initial deviation (on the $A$ side),
and all the other initial deviations that are defined
inductively. Instead of $0$ (that was also the intial
deviation for the first edge of $l_A$) we need
the initial deviation for the last edge of $l_A$.
It does not exceed the sum of the turning angles of $l_A$ and the area of the domain in the disk that projects to the same interval of $\gamma$ as $l_A$. This domain is also in $S$. So, this initial deviation does not exceed 
$2p\delta+2{\rm Area}(S)$. Now proceeding as above, we
see that the projection of the fourth side will not exceed the sum of
$20r_*p(2p\delta+2{\rm Area}(S))$ plus the same
estimate that was made for the first side. This estimate was $40r_*(p^2\delta+p{\rm Area}(S)$, and our estimate for the projection of the fourth side will be
$80r_*p(p\delta+{\rm Area}(S))$.

Combining all these terms and rounding up
for simplification we obtain 
$W=160r_*p(p\delta+{\rm Area}(S))$. In other words,
$Const_1=Const_2=160$, and the estimate $Q=6Const_1r_*^2p^2\delta$ for $Area(S)$ becomes
$960r_*^2p^2\delta<960r^2p^2=3840l^2\delta$.



In order to prove (2) note that $D$ is tiled by elementary ``narrow" polygons. Each of them
has an even number of sides. A domain bounded  by two paths from $A$ to $B$ can be tiled by such polygons. We can remove them one by one without changing the number of sides in the boundary.

To prove (3) assume that $\delta$ is attained
at a vertex $v\in\partial D$. This means that 
one of the following is true: 
(a) one of the acute angles of a face $f$ in $D$ incident to $v$ in $D$ is $\delta$; (b)  there is a face $f$ in $D$ incident
to $v$ with the obtuse angle $\rho$ at $v$ equal to $\pi-\delta$; or (c), (d) the same as (a) (correspondingly, (b)) but for the face $f$ incident to $v$ but contained in the complement of $D$.

In the case (b) replace $D$ by its polyhedral subdomain $\tilde D$ with the same properties as the properties of $D$ stated in the lemma where the edges $d_1, d_2$ of $f$ forming the angle $\rho$ are now the part of the boundary of $\tilde D$. (To achieve this we will need to replace the polyhedral arc from $A$ to $B$ via $v$ by the new polyhedral arc formed by  shortest paths from $A$ to $v$ via $d_1$, and from $v$ to $B$ via $d_2$. Then we remove convex polygons of the partition of $D$ that are between
the old and the new boundaries.) 
Now we can shorten the path from $A$ to $B$ via $v$ by replacing $v_1v_2$ by the straight line segment connecting the endpoints of $v_1$
and $v_2$ by the diagonal in $f$. We will
shorten this path by at least $2(1-\cos{\delta\over 2})\geq {\delta^2\over 4}$.
So, in this case $c\geq {\delta^2\over 4}$.

In case (a), $\delta$ is the acute angle of a face $f$ in $D$. This face is an even-dimensional polygon. Denote the vertex opposite to $v$ in this polygon by $w$. We can shorten the path connecting $w$ with $v$
along sides of this polygon by the diagonal.
It is easy to see that the difference will
be at least $1-\cos {\delta\over 2}\geq {\delta^2\over 8}$. In this case, $c\geq {\delta^2\over 8}$.

The cases (c), (d) are similar to (a), (b) 
but we shorten $v_1v_2$ using the face $f$ in $D'\setminus D$.

(4) This inequality directly follows  from (3) and Corollary 9, and replacing the resulting upper bounds by larger but more nicely looking upper bounds. To apply
Corollary 9, note that the area of $D$ can be majorized
by the sum of areas of domains between a minimal geodesic connecting $A$ and $B$ in $D'$ and either of
the two broken lines connecting $A$ and $B$ in $\partial D$. This leads to an extra factor of $2$ in the formula.

{\bf Remark.} Assume that it is not known that $\delta\leq {1\over 8p}$. If $c\leq {1\over 512p^2}$, then (2) implies that $\delta\leq {1\over 8p}$, and the conclusions of (1) and, then (4) of the lemma are valid.
The upper bound for $c$ that guarantees $\delta\leq {1\over 8p}$ can be replaced by the upper bound $c\leq {1\over 10^{10}l^2\max\{l^2, 1, {1\over conv^2(M)}\}}$


\end{proof}

{\bf Proof of Theorem 2:}

Let $M$ be a closed surface with curvature between $-1$ and $1$. Consider a closed geodesic net of length $l$ in $M$.
Let $r={1\over 2000}\min\{1, conv(M), {1\over l}\}$,
where $conv(M)$ denotes the convexity radius of $M$.
Let $B_r(a)$ be a metric ball of radius $r$ centered at a point $a\in M$. Apply the coarea formula to find $r_*\in ({r\over 2}, r)$ such that the geodesic net
intersects the geodesic circle $C_{r^*}(a)$  of radius $r_*$ centered
at $a$ transversally and not at vertices, and at most $p={2l\over r}$ points
counted with the multiplicities equal to the weight of the corresponding edges intersecting $C_{r^*}(a)$ (Figure~\ref{fig:localization}).
Denote the intersection of the considered geodesic net 
with the metric ball $B_{r^*}(a)$ by $T_a$. This net is contained in $B_{r^*}(a)$, and is not closed. It has unbalanced points with imbalance vectors of integer length located on $C_{r^*}(a)$, and the total
imbalance of this net does not exceed $p$.

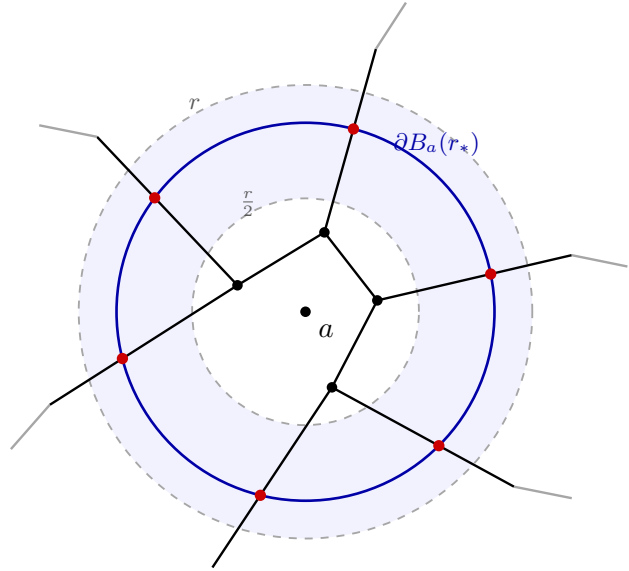
\begin{figure}[ht]
\centering
\resizebox{\ifdim\width>\textwidth \textwidth\else\width\fi}{!}{%
\begin{tikzpicture}[
   >={Stealth[length=2.2mm]},
   net/.style={line width=0.9pt, black},
   outnet/.style={line width=0.9pt, black!35},
   ball/.style={line width=1.0pt, blue!65!black},
   aux/.style={line width=0.7pt, dashed, gray!70},
   bdot/.style={circle, fill=black, inner sep=1.4pt},
   xdot/.style={circle, fill=red!80!black, inner sep=1.5pt},
   scale=1.0]

\def\rs{2.5}

\fill[blue!5] (0,0) circle (3.0);
\fill[white]  (0,0) circle (1.5);
\draw[aux] (0,0) circle (3.0);
\draw[aux] (0,0) circle (1.5);
\node[gray!70!black,font=\footnotesize] at (118:1.62) {$\tfrac r2$};
\node[gray!70!black,font=\footnotesize] at (118:3.12) {$r$};

\draw[ball] (0,0) circle (\rs);
\node[blue!65!black,font=\footnotesize] at (52:\rs+0.32) {$\partial B_a(r_*)$};

\node[bdot] at (0,0) {};
\node[below right=1pt] at (0,0) {$a$};

\draw[net] (200:3.6) -- (-0.9,0.35) -- (0.25,1.05) -- (0.95,0.15) -- (0.35,-1.0) -- (250:3.6);
\draw[outnet] (200:3.6) -- (205:4.3);
\draw[net] (0.25,1.05) -- (75:3.6);
\draw[outnet] (75:3.6) -- (72:4.3);
\draw[net] (0.95,0.15) -- (12:3.6);
\draw[outnet] (12:3.6) -- (8:4.3);
\draw[net] (-0.9,0.35) -- (140:3.6);
\draw[outnet] (140:3.6) -- (145:4.3);
\draw[net] (0.35,-1.0) -- (320:3.6);
\draw[outnet] (320:3.6) -- (325:4.3);

\foreach \p in {(-0.9,0.35),(0.25,1.05),(0.95,0.15),(0.35,-1.0)} \node[bdot] at \p {};

\node[xdot] at (-2.422,-0.619) {};
\node[xdot] at (-0.597,-2.428) {};
\node[xdot] at (0.634,2.418) {};
\node[xdot] at (2.450,0.499) {};
\node[xdot] at (-1.995,1.507) {};
\node[xdot] at (1.764,-1.772) {};

\end{tikzpicture}
}
\caption{Localizing a closed geodesic net: the geodesic circle $\partial B_a(r_*)$ crosses the net in at most $p$ points (red), yielding a subnet with all its imbalance on that circle.}
\label{fig:localization}
\end{figure}

Recall that $conv(M)$ denotes the radius of convexity of $M$.

Recall, that in section 2.3 we proved that the following proposition implies Theorem 2:

\begin{proposition} Let $c(M)=\max\left\{1,{1\over conv(M)}\right\}$. Then the geodesic net $T$ has at most $m$ branch points (e.g. vertices of degree $\geq 3$), for


$$m\leq (350c(M))^{(180c(M))^4},$$

if $l\leq c(M)$, and

$$m\leq (350l)^{(180l)^4},$$

if $l>c(M)$.

\end{proposition}

Let $\beta$ be the number of the branch points of $T$. As in the proof of Theorem 1, we use $T$
to construct a polyhedral surface $\Sigma$ homeomorphic to the disc with $\beta$ faces and with at most $p$ edges on boundary. 
If $E$ denotes the number of edges in $T$,
then $\beta\leq E$, and we need an upper bound for $E$.

Note that total absolute curvature $\mu$ of $\Sigma$ does not exceed the area of the disk of radius $r_*\le{1\over 2000}$ in the hyperbolic space, and, therefore, it is less than $0.1$. As it was mentioned in Remark 2 after the proof of Lemma 5, all results of section 4
including Lemma 5 remain valid for the considered here polyhedral surface $\Sigma$ if $\varepsilon\leq {1\over 3}$. (We need these results only for  $\varepsilon={1\over p}$.)
As above we construct a ``narrow" polyhedral surface $D'$ and $D\subset D'$ that satisfies the conditions of the previous lemma. The surface $D'$ contains $n> 0.15|E|^{1\over p^2}$ vertex-disjoint
paths $\gamma_1, \ldots , \gamma_n$ connecting the vertices $A$ and $B$ in
the conditions of the previous lemma, where
$\gamma_1$ and $\gamma_n$ form the boundary of $D'$, and for each $i$ $\gamma_i$ is between $\gamma_{i-1}$ and $\gamma_{i+1}$.
The surface $D$ is bounded by $\gamma_2$ and $\gamma_{n-1}$. All paths $\gamma_i$ connect the same points $A$ and $B$ and have
the same (integer) length. Therefore, they all have the same $c={\rm length}(\gamma_i)-|AB|$. 

As in the proof of Theorem 1, we would like
to combine the upper bound of the form $f(p)\sqrt{c}$ for the area of the domain 
$D$ and a lower bound of the form $g(p)\sqrt{c}n$ in order to obtain
an upper bound for $n$ of the form ${f(p)\over g(p)}$. In the flat case the relevant results are lemmata 6 and 7. In the case of polyhedral surfaces we have Corollary 8, Lemma 9 and Corollary 9. Both Lemma 9 and Corollary 9
involve $\mu$. However, Lemma 10 (1) and (3)
imply that $\mu< 12000l^2\sqrt{c}$. This inequality can be combined with Corollary 9 to
obtain an upper bound for the area provided 
by Lemma 10 (4).
The situation with using Corollary 8 is more complicated. The lower bound provided
by Corollary 8 is non-trivial only if $\mu<\sqrt{2c\over p}$, and is really useful when, for example,
$\mu<{1\over 2}\sqrt{2c\over p}$. However,
the upper bound for $\mu$ provided by Lemma 10 is much greater than $\sqrt{2c\over p}$.
To remedy this problem we would like to lower $\mu$ by the factor $20000l^2\sqrt{p}$. 
This can be achieved by partitioning the collection of $n-2$ paths $\gamma_2, \ldots \gamma_{n-2}$ into 
$K=\lceil 20000l^2\sqrt{p}\rceil$ 
sets of consecutive paths
with at least $[{n-2\over K}]> {n-2\over K}-1$ in each.
At least one of $K$ surfaces bounded by
paths $\gamma_{j+1}$ and $\gamma_{j+[{n-2\over K}]}$ for some $j$ (that are the first and the last paths in one of the $K$ considered sets of paths) will have the total absolute curvature $\mu$ at all inner vertices bounded above by ${\sqrt{c\over 2p}}$,
and we will apply Corollary 8 to this surface that we will denote $\tilde\Sigma$. It implies that the sum of the absolute values of outer angles at all inner vertices of each of the $[{n-2\over K}]$ paths between $A$ and $B$ satisfies $\Sigma_j|\theta_j|\geq \sqrt{c\over 2p}$. Take one of these paths. For each vertex $w_j$ in this path with adjacent edges $e_j$, $e_{j+1}$ of the path incident to these vertex. Consider the (disjoint) triangles with vertex at $w_j$ and sides directed along $e_j$ and $e_{j+1}$ but of length ${1\over 2}$. The total area of these triangles will be ${1\over 8}\Sigma_j\sin |\theta_j|>{1\over 16}\Sigma_j|\theta_j|\geq {1\over 16}\sqrt{c\over 2p}$. None of these triangles defined for one of these $[{n-2\over K}]$ paths can intersect a similarly defined triangle incident
to a vertex of another path. As the result, the area
of $\tilde\Sigma$ is at least ${1\over 16}\sqrt{c\over 2p}({n-2\over K}-1)$.

Applying Corollary 9 we see that the area of $\tilde\Sigma$ does not exceed $2\max\{7.26p^{3\over 2}\sqrt{c}+1.61\cdot{1\over 2}\sqrt{2c\over p}p^2,\ 2.38\cdot{1\over 2}\sqrt{2c\over p}p^2\}<16.8p^{3\over 2}\sqrt{c}$.
Therefore,
$16.8p^{3\over 2}\sqrt{c}\geq {1\over 16}\sqrt{c\over 2p}({n-2\over K}-1)$. Therefore,
$n\leq (381p^2+1)K+2\leq (381p^2+1)(20000l^2\sqrt{p}+1)+2<
(381p^2+1)(5p^{3\over 2}+1)+2< 1980p^{7\over 2}$. Therefore, applying Lemma 5, we see that the number of edges of the dual complex, and therefore, for the original net is not more than $({1980\over 0.15}p^{7\over 2})^{p^2}\leq
(246\max\{c(M),l\})^{7*16*10^6*(\max\{c(M),l\})^4}\leq
(246\max\{c(M),l\})^{(103\max\{c(M),l\})^4}$.


However, the upper bound for $\mu$ in terms of $c$ provided by Lemma 10 is available only under the assumption that $\sqrt{c}\leq {1\over 16\sqrt{2}p}$.
Therefore, so far we proved Proposition 2 only under this assumption, and need to consider the case, when $\sqrt{c}>{1\over 16\sqrt{2}p}$ separately.
In this case we use the trivial upper bound for $\mu$ equal to
the area of the ball of radius $r_*<r$ in the hyperbolic plane, which is less than $4r^2$. We need the total absolute curvature of our surface to be at most ${1\over\sqrt{2}}{\sqrt{c}\over\sqrt{p}}$. In view of our assumption about $c$
we can replace this bound by ${1\over 32p^{3\over 2}}={
r^3\over 64\sqrt{2}l^{3\over 2}}$. This time we would like to partition our polyhedral surface into at least
${4r^2\times (64\sqrt{2}l^{3\over 2})\over r^3}=128\sqrt{2}\sqrt{l}p<7.25*10^5(\max\{c(M),l\})^{5\over 2}$.
So, we can choose $K$ as $\lceil 7.25*10^5(\max\{l, c(M)\})^{5\over 2}\rceil$, which will similarly lead to the upper bound $n\leq (381p^2+1)K+2\leq (256\max\{c(M),l\})^{13\over 2}$. Now the number of the branch vertices in the original net is less than the number of edges in the original net and the number of edges in the dual complex. Now applying Lemma 5 we obtain an upper bound for the number of the edges of the dual complex (and, therefore, the number of the vertices
of the original net) of the form $|E|\leq (350\max\{c(M),l\})^{(6.5*16*10^6*\max\{c(M),l\}^4)}\leq (350\max\{c(M),l\})^{(180\max\{c(M),l\})^4}$.
This is greater than the upper bound $(246\max\{c(M),l\})^{(103\max\{c(M),l\})^4}$ bound obtained earlier using Corollary 8 under the assumption $c\geq {1\over 512p^2}$. Therefore, the larger estimate
$(350\max\{c(M), l\})^{(180\max\{c(M),l\})^4}$ will
work for both cases.

This completes the proof of Proposition 2 and, therefore,
Theorem 2.


%


\section{Bounding the orthogonal projection
of broken geodesics on surfaces} 

\begin{lemma}
Let $M$ be a Riemannian surface with Gaussian curvature $K$ satisfying $-1 \le K \le 1$. Let $S$ be a reference geodesic 
in a convex metric disc $D$ of radius
$r\leq {1\over 2000}$ in $M$ centered at a point of $S$. Finally, let $L$ be a piecewise geodesic 
in $D$ parametrized by the arclength.

Assume that $L$ has positive turning angles. Denote the total turning angle (=the sum of all turning angles at the break points) of $L$ by $\phi$. Denote
the initial angle between $L$
and
the geodesic perpendicular to $S$
through $L(0)$ by $x$.
Assume that both $x$ and $\phi$ are less than or equal to $0.5$ radians.
Then the length $P$ of the orthogonal projection of $L$ onto $S$ satisfies:
$$ P \le 20r(x + \phi). $$
\end{lemma}

\begin{figure}[ht]
\centering
\resizebox{\ifdim\width>\textwidth \textwidth\else\width\fi}{!}{%
\begin{tikzpicture}[
   >={Stealth[length=2.2mm]},
   geo/.style={line width=1.1pt, black},
   nrm/.style={line width=0.5pt, dashed, gray!65},
   curve/.style={line width=1.1pt, blue!70!black},
   chord/.style={line width=0.9pt, black!62},
   proj/.style={line width=0.9pt, red!75!black},
   ang/.style={line width=0.6pt, red!70!black},
   dot/.style={circle, fill=blue!70!black, inner sep=1.2pt},
   scale=1.35]

\fill[blue!7] (1.900,1.600) -- (2.100,2.733) -- (2.568,3.783) -- (3.276,4.689) -- cycle;

\draw[geo] (0.75,0) -- (4.25,0);
\node[right] at (4.27,0) {$S$};
\draw[->,line width=0.7pt] (1.08,0.26) -- (1.63,0.26);
\node[font=\footnotesize,right] at (1.65,0.26) {$p$};

\draw[nrm] (1.000,-0.35) -- (1.000,4.30);
\draw[nrm] (3.850,-0.35) -- (3.850,4.30);
\draw[nrm] (2.510,-0.35) -- (2.510,4.30);
\draw[nrm,red!50] (1.900,-0.62) -- (1.900,1.600);
\draw[nrm,red!50] (3.276,-0.62) -- (3.276,4.689);

\draw[proj,|-|] (1.900,-0.62) -- (3.276,-0.62);
\node[red!75!black,font=\footnotesize,below] at (2.588,-0.64) {$P$};

\draw[curve] (2.510,0.452) -- (1.900,1.600) -- (2.100,2.733) -- (2.568,3.783) -- (3.276,4.689);
\foreach \p in {(2.510,0.452),(1.900,1.600),(2.100,2.733),(2.568,3.783),(3.276,4.689)}
   \node[dot] at \p {};
\node[blue!70!black,font=\footnotesize,left] at (2.28,3.32) {$L$};

\draw[chord] (1.900,1.600) -- (3.276,4.689);
\node[black!62,font=\footnotesize] at (2.30,2.02) {$\rho$};

\draw[ang] (2.510,0.452) ++(90:0.60) arc (90:118:0.60);
\node[red!70!black,font=\footnotesize] at (2.304,1.277) {$x$};

\draw[gray!70!black,line width=0.5pt,dotted] (1.900,1.600) -- (1.524,2.306);
\draw[ang] (1.900,1.600) ++(80:0.62) arc (80:118:0.62);
\node[red!70!black,font=\footnotesize] at (1.775,2.390) {$\phi_1$};

\node[circle,fill=black!62,inner sep=1.2pt] at (2.510,2.970) {};
\draw[ang] (2.510,2.970) ++(246:0.50) arc (246:270:0.50);
\node[red!70!black,font=\footnotesize,right] at (2.575,2.62) {$\alpha(s)$};

\node[font=\footnotesize,right] at (2.55,0.452) {$L(0)$};
\node[font=\footnotesize,left]  at (1.85,1.580) {$L(t_1)$};
\node[font=\footnotesize,above right] at (3.276,4.689) {$L(t_2)$};

\end{tikzpicture}
}
\caption{The setting of Lemma 11: the piecewise geodesic $L$, the minimizing chord $\rho$ joining $L(t_1)$ and $L(t_2)$, and their common orthogonal projection of length $P$ onto $S$ (angles exaggerated).}
\label{fig:lemma11}
\end{figure}
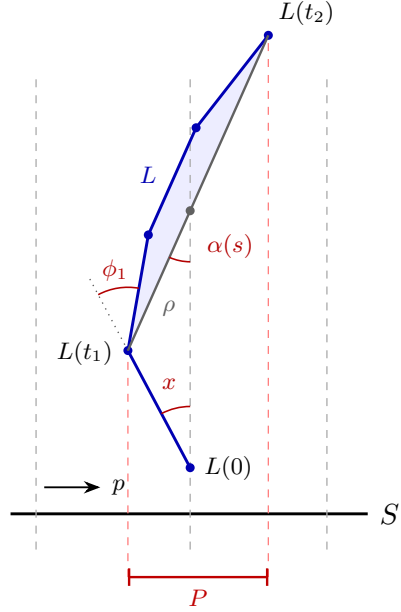
\begin{proof}


Observe, that in D we can use Fermi coordinates (p, y) (a.k.a the geodesic parallel coordinates), where p is the distance along S and y is the orthogonal distance from S. It is
well-known that Riemannian metric in these coordinates is given by
$$ ds^2=dy^2+J^2(p,y) dp^2.$$
where $J(p,y)=1$ and ${\partial J\over \partial y}(p,0)=0$. Observe that for each $p$ $J(p, y){\bf N}(p, y)$ is the Jacobi
field along the geodesic perpendicular to $S$ at $S(p)$. (Here ${\bf N}(p, y)$ denotes the unit vector
perpendicular to this geodesic at the point with coordinates $(p, y)$.) Therefore, $J(p, y)$ satisfies
the Jacobi equation $J_{yy}+ KJ = 0$ with initial conditions $J(p, 0) = 1$ and $J_y(p, 0) = 0$. It is easy to deduce from these equation and initial conditions that the curvature condition
$-1\leq K\leq 1$ implies that $\cosh y\geq J(p,y)\geq\cos y$,
which is also a corollary of the more general Rauch comparison theorem.


The projection of $L$ to $S$ is a projection of an arc of $L$ between $L(t_1)$ and $L(t_2)$ for some $t_1$ and $t_2$, and assume that no subarc of $L|_{[t_1, t_2]}$ has this property.
Connect $L(t_1)$ with $L(t_2)$ by a minimizing geodesic $\rho$ and parametrize $\rho$ by its arclength $s$.
Denote the length of $\rho$ by $\lambda$.
Obviously, $\lambda\leq 2r$.
Note that $\rho$ cannot intersect any geodesic perpendicular to $S$ at more than one point.
Therefore, $\rho$ has exactly
the same orthogonal projection to $S$ as $L$, and the length of this projection is equal to $P$.

Consider the domain $D_P\subset D$ that consists of all points $a$ in $D$ such that the orthogonal projection of $a$ to $S$ is a projection of some points in the image of $L$. In particular, $D_P$ project to an arc of length $P$ in $S$.
The inequality $J(p,y)\leq\cosh y$ implies that the area
of $D_P$ does not exceed $2P\sinh r$, which is less than $(10^{-3}+10^{-10})P$, when $r\leq {1\over 2000}$. Of course, this area also does not exceed the area of the whole disc $D$ that can be majorized by the area of the hyperbolic disc of radius $r\leq {1\over 2000}$ which is less than $10^{-6}$.


Now we are going to look into how the coordinate $p$ of $L(t)$ changes with $t$.  This coordinate can stop behaving monotonously only when the angle $\beta(t)$ between the tangent to $L(t)$ and the coordinate line $p=p(L(t))$ changes its sign. As these coordinate lines are geodesics, this can happen only at the vertices of $L$. We claim that this can happen at most once.
Indeed, if it happens twice, then at the moment when the sign of ${dp(L(t))\over dt}$ changes the second time, the total change of the angle $\beta$ needs to reach at least $\pi$, where 
the condition $\phi\leq 0.5$ implies that the turn at all vertices of $L$ does not exceed $0.5$ radian. Applying the Gauss-Bonnet theorem to each of two domains between two arcs of $L$ where $p(L(t))$ increases, $S$ , and vertical geodesics
we realize that there might be also contributions from the total curvature
terms in the Gauss-Bonnet formula. However, these contributions together cannot exceed $2Area(D)< 10^{-5}$. This contribution is not sufficient to cover
the missing $\pi-0.5$ radian.

This, $p(L(t))$ has at most two monotonicity intervals. This implies
that $p(L(t))$ is monotonous on the interval $t\in [t_1, t_2]$. It also
implies that $t_1=0$, or $t_2=$ length$(L)$, or both. If $L$ is not an arc of a vertical geodesic (in which case the lemma is trivial), $[t_1, t_2]$ and at most one of 
the intervals $[0, t_1)$, $(t_2, 1]$ have non-zero length. These (one or two) intervals are the intervals of monotonicity
of $p(L(t))$.

Denote the sum of turning angles of $L$ for vertices in $[0, t_1]$ interval by
$\phi'$ and in $(t_1, t_2]$ interval by $\phi''$. Obviously, $\phi'+\phi''\leq \phi\leq 0.5$. Now we are going to evaluate $\beta(t_1)$, that is the angle
between the tangent to $L$ at $t_1^+$ and the coordinate line $p=p(L(t_1))$.
If $t_1=0$, then this angle is equal to $x\leq 0.5$ radian. Otherwise, we can apply the Gauss-Bonnet theorem to a domain (or several domains) between $L|_{[0, t_1]}$, $S$, and vertical geodesic segments through $L(0)$ and $L(t_1)$ perpendicular to $S$. The interiors of this domain or domains are homeomorphic to discs as the coordinate $p$ is monotonous along this arc of $L$ and along $S$. As the result, we see that the
absolute value of the difference between $\beta(t_1)$ and $\beta(0)=x$ does not
exceed $\phi'$ plus the area of $D_P$, which is less than $\phi'+(10^{-3}+10^{-10})P$. Therefore, $\vert\beta(t_1)\vert\leq x+\phi'+(10^{-3}+10^{-10})P$.

Let $p(s)$ and $y(s)$ denote $p(\rho(s))$
and $y(\rho(s))$.  Note that
$|y'|=\cos\alpha$ and $J|p'|=\sin\alpha$,
where $\alpha=\alpha(s)$ is the angle 
between the tangent to $\rho$ at $\rho(s)$ and the coordinate line $p=p(s)$ which is a geodesic perpendicular to $S$. Choose the direction of coordinate lines so that
$p$ increases along $\rho$ (and, therefore, along $L|_{[t_1, t_2]}$), and $\alpha(0)$ and, therefore, $\alpha(s)$ for all $s$ is positive.

Note that
$P=\int_0^{\lambda} p'(s)ds\leq
\int_0^{\lambda}{\sin\alpha(s)\over J} ds\leq \max_s \alpha(s)\int_0^{\lambda}{ds\over J(p(s), y(s))}.$
Recall that $J\geq \cos(y(s))\geq \cos r$. Therefore, $P\leq {\lambda\over \cos r}\max_{s\in [0,\lambda]}\alpha(s)\leq 2r(1+10^{-10})\max_s\alpha(s)$.

\begin{lemma}
   $\max_s \alpha(s)\leq x+\phi+3(10^{-3}+10^{-10})P.$
\end{lemma}

Using this lemma we can finish the proof of Lemma 11 as follows. The last inequality before the lemma implies that
$P\leq 2(1+10^{-10})r(x+\phi)+6(1+10^{-10})(10^{-3}+10^{-10})rP$.
When $r\leq {1\over 2000}$, the last term in the right-hand side is less than
$10^{-5} P$.
Therefore, $(1-10^{-5})P\leq 2(1+10^{-10}) r(x+\phi)$. Dividing both sides by the coefficient at $P$ , we conclude that $P< 3r(x+\phi)< 20r(x+\phi)$.

This concludes the proof of Lemma 11 modulo Lemma 12.
\end{proof}

It remains to prove Lemma 12.

\begin{proof}
First, we are going to estimate $\alpha(0)$. As we already have an upper bound for $\beta(t_1)$, we would like to evaluate $|\alpha(0)-\beta(t_1)|$ which is the angle between the tangent to $\alpha$ at $\alpha(0)$ and the tangent to $L$ at $t_1^+$. Note that $p$ is strictly monotonous along $\alpha$, is monotonous along the arc of $L$ between
$t_1$ and $t_2$. Moreover, it is eather increasing on both these arcs or is decreasing along both this arcs. Therefore, the open domain between these curves is homeomorphic to an open disc or union of several open discs. Applying
the Gauss-Bonnet theorem we see that the
considered angle, that is, $|\alpha(0)-\beta(t_1)|$ does not exceed the sum of $\phi''$ and the area of $D_P$. Adding this estimate to our upper bound for $\beta(t_1)$ we see that $\alpha (0)$
doex not exceed $x+\phi''+\phi' + 2(10^{-3}+10^{-10})P\leq x+\phi +2(10^{-3}+10^{-10})P$.

Now we are going to evaluate $\vert\alpha(s)-\alpha(0)\vert$ for an arbitrary $s$. We are going to consider the open domain bounded by $\alpha$, $S$, and two vertical geodesics perpendicular to $S$ passing through the endpoints of $\alpha$. As $p$ is strictly monotonous along both $\alpha$
and $S$ this domain is homeomorphic to an open disc or the union of two open discs. (Two minimizing geodesics can intersect at at most one point.) Applying
the Gauss-Bonnet theorem, we see
that for each $s$ $|\alpha(s)-\alpha(0)|$
does not exceed the area of $D_P$, which together with the estimate $|\alpha(0)|\leq x+\phi+2(10^{-3}+10^{-10})P$ implies the assertion of the lemma.
\end{proof}

\forget
Note that the projection of $\alpha$ to
$S$ is the arc $A'B'$ of $S$, and it coincides with the projection of $[AB]$ to $S$. The length of the minimal geodesic $[AB]$ does not exceed the diameter of the disc $2r\leq 10^{-3}$.

We are going to evaluate the distance $d_{BA}$ form $B$ to $g_A$. Let $A''$ be the nearest point to $B$ on $g_A$. We will use Toponogov theorem for $K\geq -1$. The triangle $AA''B$ has side
$AB$ of length $\leq 2r$, the right angles at $A''$, and the angle $\alpha$ at $A'$ which is less than $x+\phi+(2+10^{-6})rP$. Consider the comparison triangle in the hyperbolic plate. The hyperbolic law of sines implies that length $d_{BA}$
of side $A'B$ satisfies the inequality
$\sinh(d_{BA})\leq \sinh (2r)\sin (x+\phi+(2+10^{-6})rP)$.
As $\sinh (d_{BA})\geq d_{BA}$, $\sin (x+\phi+(2+10^{-6})rP)\leq x+\phi+(2+10^{-6})rP$, and for $r\leq 5*10^{-4}$
$\sinh 2r\leq (2+10^{-6})r$, we obtain
$$d_{BA}\leq (2+10^{-6})r(x+\phi+(2+10^{-6})rP)\leq (2+10^{-6})r(x+\phi)+(4+10^{-5})r^2P.$$

In order to evaluate $P=dist(A', B')$ using this inequality for $d_{BA}$ consider first the geodesic triangle 
$AB'A''$. Denote the length of its side $aB'$ by $|AB'|$. 



{\bf Second proof:}
Observe, that in $D$ 
we can use Fermi coordinates $(p, y)$ (a.k.a the geodesic parallel coordinates), where $p$ is the distance along $S$ and $y$ is the orthogonal distance from $S$ (Figure~\ref{fig:fermi}). It is well-known that Riemannian metric in these coordinates is given by:
$$ ds^2 = J^2(p, y) dp^2 + dy^2, $$
where $J(p,y)=1$ and $J_y(p,0)=0$.
Observe that for each $y$ $J(p,y)N(p,y)$
is the Jacobi field along the geodesic
perpendicular to $S$ at $S(y)$. (Here
$N(p,y)$ denotes the unit vector perpendicular
to this geodesic at the points with coordinates $(p,y)$.)
Therefore, $J(p,y)$ satisfies the Jacobi equation $J_{yy} + K J = 0$ with initial conditions $J(p, 0) = 1$ and $J_y(p, 0) = 0$. By the Rauch comparison theorem, the curvature bounds $-1 \le K \le 1$ imply that for $y > 0$:
$$ \cos(y) \le J(p, y) \le \cosh(y). $$
Furthermore, the logarithmic derivative satisfies $|J_y/J| \le \tan(y)$.


\begin{figure}[ht]
\centering
\resizebox{\ifdim\width>\textwidth \textwidth\else\width\fi}{!}{%
\begin{tikzpicture}[
   >={Stealth[length=2.2mm]},
   geo/.style={line width=1.1pt, black},
   nrm/.style={line width=0.5pt, dashed, gray!65},
   curve/.style={line width=1.1pt, blue!70!black},
   proj/.style={line width=0.9pt, red!75!black},
   ang/.style={line width=0.6pt, red!70!black},
   dot/.style={circle, fill=blue!70!black, inner sep=1.2pt},
   scale=1.0]

\draw[geo] (-0.6,0) -- (9.2,0);
\node[right] at (9.2,0) {$S$};
\draw[->,line width=0.7pt] (-0.55,-0.40)--(0.35,-0.40);
\node[font=\footnotesize,right] at (0.40,-0.40) {$p$};

\foreach \x in {1,2,...,8} \draw[nrm] (\x,-0.35) -- (\x,3.1);

\coordinate (L0) at (1.55,0.95);
\coordinate (L1) at (2.7,1.85);
\coordinate (L2) at (4.4,2.35);
\coordinate (L3) at (6.2,2.05);
\coordinate (L4) at (7.35,1.25);
\draw[curve] (L0)--(L1)--(L2)--(L3)--(L4);
\node[blue!70!black] at (4.4,2.62) {$L$};
\foreach \p in {(L0),(L1),(L2),(L3),(L4)} \node[dot] at \p {};

\draw[nrm,gray!85!black] (L0) -- ($(L0)+(90:1.05)$);
\draw[ang] (L0) ++(90:0.72) arc (90:38:0.72);
\node[red!70!black,font=\footnotesize] at ($(L0)+(63:1.0)$) {$x$};

\draw[gray!70!black,line width=0.5pt,dotted] (L1) -- ($(L1)+(38:1.15)$);
\draw[ang] (L1) ++(38:0.85) arc (38:16:0.85);
\node[red!70!black,font=\footnotesize] at ($(L1)+(27:1.28)$) {$\phi_1$};

\draw[proj,|-|] (1.55,-0.62) -- (7.35,-0.62);
\node[red!75!black,font=\footnotesize,below] at (4.45,-0.60) {$P$};
\draw[nrm,red!45] (L0) -- (1.55,-0.62);
\draw[nrm,red!45] (L4) -- (7.35,-0.62);

\draw[<->,line width=0.6pt,gray!70!black] (8.65,0) -- (8.65,2.35);
\node[gray!70!black,font=\footnotesize,right] at (8.68,1.2) {$\le 2r$};

\end{tikzpicture}
}
\caption{Fermi coordinates $(p,y)$ along the reference geodesic $S$: the curve $L$, its initial deviation $x$, a turning angle $\phi_1$, and its projection $P$ onto $S$.}
\label{fig:fermi}
\end{figure}
Before proving Lemma 11 we will need the following two lemmas:



\begin{lemma}[Upper Bound on Arc Length]
The total arc length $\mathcal{L}$ of the curve $L$ is bounded by:
\[ \mathcal{L} \le \frac{2r \cosh(2r)}{\cos (\phi+2*10^{-6})} \]
\end{lemma}

\begin{proof}
Let $\gamma$ be a minimizing geodesic chord in $M$ connecting the endpoints of $L$, with length $\lambda \le 2r$. In Fermi coordinates $(p', y')$ along $\gamma$, the metric is $ds^2 = J_\gamma^2 dp'^2 + dy'^2$. As above, the Rauch comparison theorem implies $J_\gamma(p', y') \le \cosh(y') \le \cosh(2r)$. Let $\alpha(s)$ denote the angle between the tangent vector of $L$ at $L(s)$ and the coordinate line $y'=const$ through $L(s)$. 

We would like to find a good upper bound for $\vert\alpha(s)\vert$. Assume that $L(s)$ is not one the vertices. If tangent
lines of $L$ at $L(s)$ and the line $y'=y'(L(s))$ do not coincide then an arc
of $y'=y'=y'(L(s))$ enters the bounded domain between $\gamma$ and $L$. It does not intersect $\gamma$, so it will exit
through $L$. Thus, we will obtain a domain bounded by the union of the coordinate line $y'=y'(L(s))$. The angle
of interest for us is one of two (acute) inner angles; the second angle will be at another points of the intersection.
We are going to apply the Gauss-Bonnet
theorem to this disc-shaped domain. Its area, and, therefore, the total curvature will not exceed the area of the disc, which is less than or equal to the area of the hyperbolic disc of radius $r\leq {1\over 2000}$, which is less than $10^{-6}$.
The sum of the turning angles at all other vertices does not exceed $\phi$.
Therefore, $|\alpha(S)|$ will be less that or equal $\phi+10^{-6}+$ the total absolute geodesic curvature of the coordinate line $y'=y'(L(s))$. The length of each coordinate line does not exceed the maximal length of a coordinate line in a hyperbolic ball of radius $2r$ which is trivially bounded by $2r\cosh r\leq 2*10^{-3}$. (In fact, $r$ is very small ($\leq 0.0005$), the length
of each coordinate line does not exceed
$2r\leq 10^{-3}$.) Now let us estimate
the geodesic curvature $\kappa_g$ of 
coordinate lines. Applying the formula for the geodesic curvature of the geodesic coordinate curves $\kappa_g=-{\partial J\over \partial y'}$.
Therefore, $\vert\kappa_g\vert\leq \sinh(r)$. So, the total absolute geodesic curvature does not exceed $2r\cosh r\sinh r=r\sinh (2r)\leq 10^{-3}\sinh(2*10^{-3})<10^{-6}$.
Summarizing, $|\alpha(s)|<\phi+2*10^{-6}$.


Thus:
$$2r\geq \lambda = \int_0^\mathcal{L} \frac{\cos \alpha(s)}{J_\gamma} ds \ge \int_0^\mathcal{L} \frac{\cos (\phi+2*10^{-6})}{\cosh(2r)} ds = \mathcal{L} \frac{\cos (\phi+2*10^{-6})}{\cosh(2r)}. $$
Rearranging yields the stated bound.
\end{proof}

The angle deviation, $\theta(s)$, is, by definition, the angle between the tangent
vector of $L$ at $L(s)$ and the normal to
$\gamma$ that passes through $L(s)$.

\begin{lemma}[Evolution of the Angle]
The angle deviation $\theta(s)$ is bounded by:
\[ |\theta(s)| \le (x + \phi) \exp \left( \frac{2.06r^2 \cosh^2(2r)}{\cos^2 (\phi+2*10^{-6})} \right) \]
\end{lemma}

\begin{proof}
The differential evolution of the angle deviation $\theta(s)$ relative to the normal is given
by the formula $\frac{d\theta(s)}{ds} = -\frac{J_y}{J} \sin(\theta)$.
Here is how this formula can be deduced.
Calculating the Christoffel symbols for the Riemannian metric, we discover that one of two geodesic equations is 
${d^2y\over ds^2}=JJ_y({dp\over ds})^2$.
Further, for a unit speed geodesic $1=({dy\over ds})^2+ J^2({dp\over ds})^2$, and
$\cos\theta(s)={dy\over ds}$. Therefore,
$\sin\theta(s)=J{dp\over ds}$. Now,
differentiating both sides of $\cos\theta(s)={dy\over ds}$, and using the mentioned geodesic equation, we obtain
$-\sin\theta(s){d\theta\over ds}=JJ_y({dp\over ds})^2={J_y\over J}\sin^2\theta(s)$. Dividing both sides by $-\sin\theta(s)$
we obtain the desired equality.

Using the comparison bound $|J_y/J| \le \tan(y)<1.03y<1.03s$ plus the inequality $\sin \theta\leq \theta$ for positive $\theta$, we obtain the following differential inequality: $\frac{d|\theta(s)|}{ds}\leq 1.03s|\theta|$. Dividing both sides by $|\theta|$, integrating and exponentiating, we obtain
$|\theta(s)|\leq (x+\phi)\exp(1.03{\mathcal{L}^2\over 2})$. Now
we substitute the bound for $\mathcal{L}$ from Lemma 12 to reach the desired result.
\end{proof}


Now we can prove Lemma 11.

\begin{proof}
The change of $p$-coordinate along $S$ is given by $\frac{dp}{ds} = \frac{\sin \theta}{J}$. Using $J \ge \cos(2r)$ and $\sin \theta \le \theta$:
\[ P = \int_0^\mathcal{L} \frac{\sin \theta(s)}{J} ds \le \int_0^\mathcal{L} \frac{|\theta(s)|}{\cos(2r)} ds \le \mathcal{L} \frac{|\theta_{max}|}{\cos(2r)} \]
The substitution of Lemmas 12 and 13 yields $P \le C \cdot r(x + \phi)$ where the constant $C$ is:
\[ C = \frac{2 \cosh(2r)}{\cos (\phi+2*10^{-6})\cos(2r)} \exp \left( \frac{2.06r^2 \cosh^2(2r)}{\cos^2 (\phi+2*10^{-6})} \right) \]
Evaluating at the threshold $r = 0.0005, \phi = 0.5$ gives $C<3$, which is strictly less than 20.
\end{proof}
\forgotten
\par\noindent
{\bf Acknowledgements:} The work of one of the authors (A.N.) was partially supported by his NSERC Discovery Grant. He worked on this project during his visits at the SLMath and Stanford University in the fall semester of 2024 and at the Hausdorff Institute of Mathematics in the winter semester of 2025. He would like to thank these institutions for their hospitality. D.K. was supported
by the European Research Council (ERC) Grant No. 101045750 (HodgeGeoComb). He is grateful to Sergey Avvakumov for communicating the planar version of the problem. 

While at Stanford, A.N. learned that Henry Bosch was also working on the same problem and in Summer 2024 announced a result that looks similar to Theorem 1 in the present paper but proved using a completely different method.
But Bosch's paper [B] appeared on arXiv only approximately a month and a half
after the first version of this paper. Henry's result
is about stationary geodesic nets
in (Euclidean) $\mathbb{R}^n$ for all $n$ and not only $n=2$ (as in the present paper). However, his result holds only for geodesic nets that are sufficiently close to a multiplicity k straight line.
(Given a sequence of stationary geodesic nets $N_i$ converging to a multiplicity $k$ straight line $l$, Henry's upper bound will hold only for all sufficiently large $i$.)
His result is an upper bound for the number of branch points of stationary geodesic networks sufficiently close to $kl$, and for $n=2$ this upper bound is $(10k)^{k^2}$, which is the same type of growth as in our Theorem 1 yet with better numerical constants.

Henry Bosch is a Ph.D. student at Stanford working under supervision of Otis Chodosh, and Chodosh observed that Henry's proof can be combined with a compactness argument that for $n=2$ extends Henry's result for geodesic nets in the Euclidean plane to closed geodesic nets on surfaces ([C]). However, this is a non-constructive argument, and the knowledge of a specific upper bound for the number of branch points for a class of geodesic nets in the Euclidean plane does not lead to a specific upper bound for the number of branch points for a class of closed geodesic nets on surfaces.


\vskip 0.5truecm
\par\noindent
I.F. IMPA, Estrada Dona Castorina, 110, Jardim Botânico, Rio de Janeiro, RJ, CEP 22460-320, Brazil.
\par\noindent
D.G. Department of Mathematics, University of Toronto, Bahen Center, Toronto, ON, M5S2E4, Canada.
\par\noindent
D.K. Sorbonne Université - IMJ-PRG
Tour 15-16 (Tower 15-16)
4, place Jussieu
75252 Paris Cedex 05, France.
\par\noindent
A.N. Department of Mathematics, University of Toronto, Bahen Center, Toronto, ON, M5S2E4, Canada.

\end{document}